\documentclass[11pt,leqno]{amsart}
\usepackage{verbatim,amssymb,amsfonts,latexsym,amsmath,amsthm,bbm,nicefrac,xspace,enumerate,tikz,marvosym,mathtools}
\usetikzlibrary{decorations.pathmorphing,shapes}

\newtheorem{theorem}{Theorem}[section]
\newtheorem*{theorem*}{Theorem}
\newtheorem*{corollary*}{Corollary}
\newtheorem*{maintheorem}{Main Theorem}

\newtheorem{lemma}[theorem]{Lemma}

\newtheorem{corollary}[theorem]{Corollary}

\theoremstyle{definition}

\newtheorem{definition}[theorem]{Definition}

\theoremstyle{remark}

\newcommand{\N}{\mathbb{N}}
\newcommand{\Z}{\mathbb{Z}}
\newcommand{\CC}{\mathbb{C}}

\newcommand{\A}{\mathcal{A}}
\newcommand{\B}{\mathcal{B}}
\newcommand{\C}{\mathcal{C}}

\newcommand{\F}{\mathcal{F}}

\newcommand{\LL}{\mathcal{L}}
\renewcommand{\P}{\mathcal{P}}

\newcommand{\U}{\mathcal{U}}
\newcommand{\V}{\mathcal{V}}
\newcommand{\W}{\mathcal{W}}

\newcommand{\explicitSet}[1]{\left\lbrace #1 \right\rbrace}
\newcommand{\brackets}[1]{\left\langle #1 \right\rangle}
\newcommand{\set}[2]{\explicitSet{#1 \colon #2}}
\newcommand{\seq}[2]{\brackets{#1 \colon #2}}
\newcommand{\<}{\langle}
\renewcommand{\>}{\rangle}
\renewcommand{\a}{\alpha}
\renewcommand{\b}{\beta}
\newcommand{\g}{\gamma}
\newcommand{\dlt}{\delta}

\newcommand{\z}{\zeta}

\newcommand{\s}{\sigma}
\renewcommand{\t}{\tau}
\newcommand{\w}{\omega}
\newcommand{\0}{\emptyset}

\newcommand{\sub}{\subseteq}
\newcommand{\rest}{\!\restriction\!}

\newcommand{\iso}{\cong}

\newcommand{\card}[1]{\left\lvert #1 \right\rvert}

\newcommand{\PP}{\mathbb{P}}

\newcommand{\zero}{0\hspace{-2.1mm}0}
\newcommand{\1}{1\hspace{-1.75mm}1}
\newcommand{\pwmf}{\mathcal{P}(\w)/\mathrm{Fin}}
\newcommand{\pwmff}{\nicefrac{\mathcal{P}(\w)}{\mathrm{Fin}}}
\newcommand{\pzmff}{\nicefrac{\mathcal{P}(\Z)}{\mathrm{Fin}}}

\renewcommand{\AA}{\mathbb{A}}
\newcommand{\BB}{\mathbb{B}}
\newcommand{\DD}{\mathbb{D}}

\newcommand{\continuum}{\mathfrak{c}}

\newcommand{\ch}{\ensuremath{\mathsf{CH}}\xspace}

\newcommand{\zfc}{\ensuremath{\mathsf{ZFC}}\xspace}
\newcommand{\pfa}{\ensuremath{\mathsf{PFA}}\xspace}
\newcommand{\ocama}{\ensuremath{\mathsf{OCA + MA}}\xspace}

\newcommand{\oca}{\ensuremath{\mathsf{OCA}}\xspace}

\newcommand{\tosi}{\xrightarrow{\text{\tiny $\,\s\,$}}}
\newcommand{\toa}{\xrightarrow{\scalebox{.9}{$\,{}_{\a}\,$}}}
\newcommand{\tob}{\xrightarrow{\scalebox{.8}{$\,{}_{\b}\,$}}}

\newcommand{\toA}{\xrightarrow{\scalebox{.75}{$\,{}_{\A}\,$}}}
\newcommand{\toB}{\xrightarrow{\scalebox{.75}{$\,{}_{\B}\,$}}}
\newcommand{\toC}{\xrightarrow{\scalebox{.75}{$\,{}_{\C}\,$}}}
\newcommand{\toU}{\xrightarrow{\scalebox{.75}{$\,{}_{\U}\,$}}}
\newcommand{\toV}{\xrightarrow{\scalebox{.75}{$\,{}_{\V}\,$}}}
\newcommand{\toW}{\xrightarrow{\scalebox{.75}{$\,{}_{\W}\,$}}}
\newcommand{\toSp}{\xrightarrow{\scalebox{.75}{$\,{}_{\mathcal S^\phi}\,$}}}
\newcommand{\St}{\mathcal{S}}
\newcommand{\E}{\mathcal{E}}
\newcommand{\fS}{\mathrm{S}}

\newcommand{\fB}{\mathrm{T}}

\begin{document}

\title[Does $\mathcal P(\omega) / \mathrm{Fin}$ know its right hand from its left?]{Does $\mathcal P(\omega) / \mathrm{Fin}$ know its right hand \linebreak from its left? 
}
\author{Will Brian}
\address {
W. R. Brian\\
Department of Mathematics and Statistics\\
University of North Carolina at Charlotte\\
Charlotte, NC 
(USA)}
\email{wbrian.math@gmail.com}
\urladdr{wrbrian.wordpress.com}


\thanks{The author is supported by NSF grant DMS-2154229}

\begin{abstract}
Let $\s$ denote the shift automorphism on $\pwmff$, defined by setting
$\s([A]) = [A+1]$ for all $A \sub \w$. 
We show that the Continuum Hypothesis implies the shift automorphism $\sigma$ and its inverse $\sigma^{-1}$ are conjugate in the automorphism group of $\pwmff$.

Due to work of van Douwen and Shelah, it has been known since the 1980's that it is consistent with $\mathsf{ZFC}$ that $\sigma$ and $\sigma^{-1}$ are not conjugate. Our result shows that the question of whether $\sigma$ and $\sigma^{-1}$ are conjugate is independent of $\mathsf{ZFC}$.

As a corollary to the main theorem, we deduce that the structures $\langle \pwmff,\sigma \rangle$ and $\langle \pwmff,\sigma^{-1} \rangle$ are elementarily equivalent in the language of algebraic dynamical systems (an algebraic dynamical system being a Boolean algebra together with an automorphism). This corollary does not depend on the Continuum Hypothesis. 
\end{abstract}

\maketitle

\section{Introduction}

A \emph{topological dynamical system} is a pair $\< X,f \>$, where $X$ is a compact Hausdorff space and $f: X \to X$ is a homeomorphism. In the large and expanding literature on topological dynamical systems, sometimes more is required of $X$ (such as metrizability), or less of $f$ (perhaps it may be just a continuous surjection rather than a homeomorphism). 
However the objects of this category are defined, the isomorphisms are defined as follows. 
Two topological dynamical systems $\< X,f \>$ and $\< Y,g \>$ are \emph{conjugate} 
if there is a homeomorphism $h: X \to Y$, called a \emph{conjugacy}, such that $h \circ f = g \circ h$.

\begin{center}
\begin{tikzpicture}

\node at (0,2) {$X$};
\node at (2,2) {$X$};
\node at (0,0) {$Y$};
\node at (2,0) {$Y$};

\draw[->] (.35,2) -- (1.65,2);
\draw[->] (.35,0) -- (1.65,0);
\draw[->] (0,1.65) -- (0,.35);
\draw[->] (2,1.65) -- (2,.35);

\node at (1,2.25) {\footnotesize $f$};
\node at (1,-.25) {\footnotesize $g$};
\node at (-.22,1) {\footnotesize $h$};
\node at (2.22,1) {\footnotesize $h$};

\end{tikzpicture}
\end{center}

\noindent This is the natural notion of isomorphism in the category of topological dynamical systems. 
The term ``conjugate'' (rather than \emph{isomorphic}) arises from the case $X=Y$, where the above definition says $f$ and $g$ are conjugate elements of the autohomeomorphism group of $X$. 
The main theorem of this paper can be phrased as a theorem of topological dynamics:

\begin{maintheorem}
The Continuum Hypothesis implies that the topological dynamical systems $\< \w^*,\s \>$ and $\< \w^*,\s^{-1} \>$ are conjugate.
\end{maintheorem}

\noindent In the statement of this theorem, $\w^*$ denotes the \v{C}ech-Stone remainder of the countable discrete space $\w$ of all finite ordinals, and $\s$ denotes the shift homeomorphism of $\w^*$, which maps an ultrafilter $u \in \w^*$ to the ultrafilter generated by $\set{A+1}{A \in u}$. 
Roughly speaking, this theorem states that, assuming the Continuum Hypothesis (henceforth abbreviated \ch), the maps $\s$ and $\s^{-1}$ are indistinguishable from the topological point of view.

This theorem answers question 11 from Hart and van Mill's ``Problems on $\b \N$" \cite{HvM}. 
(The question is also raised directly in \cite{Geschke}, \cite{Fetal}, \cite{Brian2}, and \cite{YCor}.)
The roots of the question reach back to work of van Douwen and Shelah from the 1980's. Near the end of his life, van Douwen considered the question of whether $\< \w^*,\s \>$ and $\< \w^*,\s^{-1} \>$ are conjugate. He proved that if in fact they are conjugate, then any map $h$ witnessing their conjugacy cannot be a \emph{trivial} homeomorphism (meaning a homeomorphism induced on $\w^*$ by a map $\w \to \w$, in the way that $\s$ is induced by the successor map $n \mapsto n+1$). This work was published posthumously in \cite{vanDouwen}. Not long before van Douwen's work on this question, Shelah proved in \cite{Shelah} that it is consistent with \zfc that every homeomorphism $\w^* \to \w^*$ is trivial. Combining these results of Shelah and van Douwen yields the following:

\begin{theorem}[van Douwen and Shelah]
It is consistent with \zfc that $\<\w^*,\s\>$ and $\<\w^*,\s^{-1}\>$ are not conjugate.
\end{theorem}

Later work of Shelah and Stepr\={a}ns in \cite{ShelahSteprans} shows that ``every autohomeomorphism of $\w^*$ is trivial'' is not merely consistent with \zfc, but is in fact a consequence of the $\mathsf{P}$roper $\mathsf{F}$orcing $\mathsf{A}$xiom, or $\pfa$. 
This was improved by Veli\v{c}kovi\'c, who showed in \cite{Velickovic} that \ocama, a significant weakening of \pfa, implies all autohomeomorphisms of $\w^*$ are trivial. 
(Here \oca denotes the $\mathsf{O}$pen $\mathsf{C}$oloring $\mathsf{A}$xiom defined by Todor\v{c}evi\'c in \cite{Tod}.) 
Building on work of Moore in \cite{Moore}, De Bondt, Farah, and Vignati showed recently in \cite{dBFV} that \oca alone suffices. 
See \cite{BFV}, \cite{FS}, and \cite{Vignati} for further results in this vein. 
Thus \oca implies $\<\w^*,\s\>$ and $\<\w^*,\s^{-1}\>$ are not conjugate. 

Combined with the work of van Douwen and Shelah, our result shows that the question of whether $\sigma$ and $\sigma^{-1}$ are conjugate is independent of $\mathsf{ZFC}$. 

\vspace{2mm}

An \emph{algebraic dynamical system} is a pair $\< \AA,\a \>$, where $\AA$ is a Boolean algebra and $\a$ is an automorphism of $\AA$. Algebraic dynamical systems have received less attention in the literature than their topological counterparts. However, the machinery of Stone duality reveals that the category of algebraic dynamical systems is naturally equivalent to the category of zero-dimensional topological dynamical systems.

It turns out that several aspects of the proof of our main theorem are stated more simply and proved more cleanly in the algebraic category rather than the topological one. For this reason, we do not say much about topological dynamical systems beyond this introduction, and do all of our work (or nearly all) with algebraic dynamical systems instead. Henceforth, a \emph{dynamical system} (with no adjective) means an algebraic dynamical system. 
It is likely helpful, but not strictly necessary, for the reader to be comfortable with the basics of Stone duality. 

\vspace{2mm}

\begin{center}
\begin{tikzpicture}[decoration=snake,yscale=.8]

\node at (0,2.5) {\small Boolean};
\node at (0,2) {\small algebras};
\node at (8,2.5) {\small zero-dimensional};
\node at (8,2) {\small compact $T_2$ spaces};

\node at (0,1) {\small isomorphisms};
\node at (0,.5) {\small of algebras};
\node at (8,1) {\small homeomorphisms};
\node at (8,.5) {\small of spaces};

\node at (0,-.5) {\small embeddings};
\node at (0,-1) {\small of algebras};
\node at (8,-.5) {\small continuous surjections};
\node at (8,-1) {\small between spaces};

\node at (0,-2) {\small $\P(\w) \text{ and } \pwmff$};
\node at (8,-2) {\small $\b\w \text{ and } \w^*$};


\draw[decorate] (2.2,2.25) -- (5.4,2.25);
\draw[->] (5.4,2.25) -- (5.5,2.25);
\draw[<-] (2,2.25) -- (2.2,2.25);
\draw[decorate] (2.2,.75) -- (5.4,.75);
\draw[->] (5.4,.75) -- (5.5,.75);
\draw[<-] (2,.75) -- (2.2,.75);
\draw[decorate] (2.2,-.75) -- (5.4,-.75);
\draw[->] (5.4,-.75) -- (5.5,-.75);
\draw[<-] (2,-.75) -- (2.2,-.75);
\draw[decorate] (2.2,-2) -- (5.4,-2);
\draw[->] (5.4,-2) -- (5.5,-2);
\draw[<-] (2,-2) -- (2.2,-2);

\end{tikzpicture}
\end{center}

\vspace{2mm}

The power set of $\w$, $\P(\w)$, when ordered by inclusion, is a complete Boolean algebra. The roles of ``join'' and ``meet'' in this algebra are played by the familiar set operations $\cup$ and $\cap$, respectively. 
The Boolean algebra $\pwmff$ 
is the quotient of the algebra $\P(\w)$ by the ideal of finite sets. For $A,B \sub \w$, we write $A =^* B$ to mean that $A$ and $B$ differ by finitely many elements, i.e. $\card{(A \setminus B) \cup (B \setminus A)} < \aleph_0$.
The members of $\nicefrac{\P(\w)}{\mathrm{Fin}}$ are equivalence classes of the $=^*$ relation, and we denote the equivalence class of $A \sub \w$ by $[A]$. The Boolean-algebraic relation $[A] \leq [B]$ is denoted on the level of representatives by writing $A \sub^* B$: that is, $A \sub^* B$ means that $[A] \leq [B]$, or (equivalently) that $B$ contains all but finitely many members of $A$.
The \emph{shift map} is the automorphism $\s$ of $\pwmff$ defined by setting
$$\s([A]) = [A+1] = [\set{a+1}{a \in A}]$$
for every $A \in \P(\w)$. The inverse mapping is
$$\s^{-1}([A]) = [A-1] = [\set{a-1}{a \in A \setminus \{0\}}].$$
In other words, $\s$ is the automorphism of $\pwmff$ that shifts every set one unit to the right, and $\s^{-1}$ shifts every set one unit to the left.

Two dynamical systems $\< \AA,\a \>$ and $\< \BB,\b \>$ are \emph{conjugate} if there is an isomorphism $\phi: \AA \to \BB$ such that $\phi \circ \a = \b \circ \phi$.

\begin{center}
\begin{tikzpicture}

\node at (0,2) {$\BB$};
\node at (2,2) {$\BB$};
\node at (0,0) {$\AA$};
\node at (2,0) {$\AA$};

\draw[<-] (.35,2) -- (1.65,2);
\draw[<-] (.35,0) -- (1.65,0);
\draw[<-] (0,1.65) -- (0,.35);
\draw[<-] (2,1.65) -- (2,.35);

\node at (1,2.25) {\footnotesize $\b$};
\node at (1,-.2) {\footnotesize $\a$};
\node at (-.22,1) {\footnotesize $\phi$};
\node at (2.22,1) {\footnotesize $\phi$};

\end{tikzpicture}
\end{center}

\noindent In other words, two (algebraic) dynamical systems are conjugate if and only if their Stone duals are conjugate as topological dynamical systems.
Weakening this notion, we say that $\eta$ is an \emph{embedding} of $\< \AA,\a \>$ into $\< \BB,\b \>$ if $\eta$ is a Boolean-algebraic embedding $\AA \to \BB$ with the additional property that $\eta \circ \a = \b \circ \eta$ (i.e., $\eta$ is a conjugacy from $\<\AA,\a\>$ onto a subsystem of $\< \BB,\b \>$). In other words, an embedding of dynamical systems is like a conjugacy, but without requiring the connecting map to be surjective. 
The Stone dual of this notion is called a \emph{factor map}, or more rarely a \emph{semi-conjugacy} or a \emph{quotient map}, in the topological dynamics literature.

Our main theorem, stated in the algebraic category, says:

\begin{maintheorem}
\ch implies $\< \pwmff,\s \>$ and $\< \pwmff,\s^{-1} \>$ are conjugate.
\end{maintheorem}

Several results related to van Douwen's question and this theorem were known before the present paper: 
\begin{itemize}
\item[$\circ$] As mentioned already, if there is a conjugacy $\phi$ from $\<\pwmff,\s\>$ to $\<\pwmff,\s^{-1}\>$, then $\phi$ is a non-trivial automorphism of $\pwmff$. Consequently, it is consistent that no such automorphism exists. 
But also, it has been known since work of Rudin \cite{Rudin} that the Continuum Hypothesis implies the existence of many non-trivial automorphisms of $\pwmff$. In fact, \ch seems to be an optimal hypothesis for constructing non-trivial automorphisms (see \cite[Section 1.6]{vanMill}).

\vspace{2mm}

\item[$\circ$] If it is consistent with \zfc that $\<\pwmff,\s\>$ and $\<\pwmff,\s^{-1}\>$ are conjugate, then it is also consistent with $\zfc+\ch$. This result is folklore. (A sketch of the proof: If $\<\pwmff,\s\>$ and $\<\pwmff,\s^{-1}\>$ are conjugate, then force with countable conditions to collapse the continuum to $\aleph_1$. In the forcing extension, \ch is true, and because this notion of forcing is countably closed, the fact that $\<\pwmff,\s\>$ and $\<\pwmff,\s^{-1}\>$ are conjugate is preserved in the extension.)

\vspace{2mm}

\item[$\circ$] The existence of certain large cardinals implies that if ``$\<\pwmff,\s\>$ and $\<\pwmff,\s^{-1}\>$ are conjugate" is true in some forcing extension, then it is true in every forcing extension satisfying \ch. This follows from Woodin's celebrated $\Sigma^2_1$ absoluteness theorem. 
Farah's account in \cite{Farah} contains a precise statement and proof of Woodin's result.

\vspace{2mm}

\item[$\circ$] The Continuum Hypothesis implies that $\<\pwmff,\s\>$ {embeds} into $\<\pwmff,\s^{-1}\>$, and vice-versa. Like with conjugacies, the axiom \ocama implies there are no embeddings from $\<\pwmff,\s\>$ into $\<\pwmff,\s^{-1}\>$ or vice-versa; thus the existence of such embeddings is independent of \zfc. These results are found in \cite{Brian1}.

\vspace{2mm}

\item[$\circ$] The conjugacy class of $\s$ and the conjugacy class of $\s^{-1}$ cannot be separated by a Borel set in $\mathcal Aut(\pwmff)$, the automorphism group of $\pwmff$ equipped with the topology of pointwise convergence. This is another sense in which $\s$ and $\s^{-1}$ are topologically indistinguishable. This result is proved in \cite{Brian2}.

\end{itemize}

The first four of these five points suggest, and have suggested for some time, that if ``$\<\pwmff,\s\>$ and $\<\pwmff,\s^{-1}\>$ are conjugate" is consistent with \zfc, then it may be outright provable from $\zfc+\ch$. 
Indeed, this has turned out to be the case. 
Further partial results concerning this problem, and a summary of its status as of 15 years ago, can be found in the survey of Geschke \cite{Geschke}. 

The behavior of $\pwmff$ under \ch is usually considered to be well understood: see, e.g., \cite[Section 1]{vanMill} and \cite{Farah2}. Our main result here further reinforces this feeling, filling what was one of the few remaining gaps (as far as we are able to see) in our knowledge of how \ch affects $\pwmff$ and its automorphisms.

\vspace{2mm}

Ultimately, our proof of the main theorem relies on a transfinite back-and-forth argument. This recursive argument needs to deal with $\continuum$ tasks in succession, but the recursive hypotheses cannot survive more than $\w_1$ stages. Thus the argument can only succeed if $\continuum = \w_1$, i.e., if \ch holds. This is in fact the only point in the proof where \ch is needed. 

This back-and-forth argument is laid out in Section~\ref{sec:back&forth}. 
The limit steps of the recursion are easy (just take unions); all the difficulty lies in the successor steps. 
At successor steps, we wish to take a conjugacy between countable substructures of $\< \pwmff,\s \>$ and $\< \pwmff,\s^{-1} \>$ and extend it to a conjugacy between strictly larger substructures. 
Furthermore, we must have at least some control over the growth of these substructures as the recursion progresses, in order to ensure that they cover $\pwmff$ in $\w_1$ steps. 

This problem at successor steps is isolated and distilled into a problem of infinite combinatorics that we call the \emph{lifting problem for $\< \pwmff,\s \>$}. What precisely constitutes an instance of the lifting problem for $\< \pwmff,\s \>$, and a solution to it, are defined in Section~\ref{sec:Incompressibility}. 

We show in Section 3 that if every instance of the lifting problem for $\< \pwmff,\s \>$ were solvable, then there would be a fairly straightforward path to proving our main theorem. However, we show in Section~\ref{sec:NonSaturation} that this is too much to hope for: the most general form of the lifting problem does not always have a solution. 
To put it another way, the kind of task we face at the successor steps of our recursive argument is not always doable. 

In light of this, our recursive argument must be done very carefully, taking care to avoid running into unsolvable instances of the lifting problem at successor steps. 
The key to this is to use elementary substructures, rather than arbitrary countable substructures, wherever it is possible to do so.

Instead of dealing with the lifting problem first, and then proving the main theorem, we present things in the reverse order. 
Section~\ref{sec:back&forth} gives a contingent proof of the main theorem, relying on a key lemma that offers, under certain circumstances, a solution to the lifting problem (see Lemma~\ref{lem:main}). 
The remainder of the paper, except the last section, is devoted to proving this lemma. 
Our proof of the lemma (and even its statement) uses some model-theoretic ideas, but the main part of the proof, the heart of the whole argument, is more combinatorial: a careful analysis of the finite directed graphs that, in a sense to be made precise in Section~\ref{sec:Digraphs}, each capture some finite amount of information about $\< \pwmff,\s \>$ and $\< \pwmff,\s^{-1} \>$. 
This part of the argument is found in Sections \ref{sec:Diagonalization}$\,$-\ref{sec:Exhale}. 

In the final section of the paper, Section~\ref{sec:Corollaries}, we obtain several further applications of main theorem and the main lemma. Each of these answers a question that does not seem to appear in the literature, but has been part of the folklore surrounding dynamics in $\w^*$ for some time. 
For example, we show, as a corollary to the main theorem, that the structures $\< \pwmff,\sigma \>$ and $\< \pwmff,\sigma^{-1} \>$ are elementarily equivalent. Unlike the main theorem, this corollary does not assume \ch. 
We also show, as a corollary to the main lemma, that \ch implies any conjugacy between two countable elementary substructures of $\< \pwmff,\s \>$ can be extended to a conjugacy from $\< \pwmff,\s \>$ to itself. 
This implies for example that, assuming \ch, the shift map is conjugate to a nontrivial automorphism of $\pwmff$, and that there are nontrivial automorphisms commuting with the shift map, i.e., automorphisms $\phi$ such that $\phi \circ \s = \s \circ \phi$, but $\phi \neq \s^n$ for any $n \in \Z$.

\vspace{3mm}

\noindent \textbf{Acknowledgments.} 
Since the time I started working on van Douwen's problem, a decade ago as a postdoc at Tulane University, I have had many valuable and enlightening conversations about the shift map with Hector Barriga-Acosta, Dana Barto\v{s}ov\'a, Alan Dow, Ilijas Farah, K. P. Hart, Karl Hofmann, Paul Larson, Paul McKenny, Mike Mislove, Brian Raines, Paul Szeptycki, and Andy Zucker. This paper would not have been possible without their insights. 
I am especially grateful to Alan Dow, with whom I discussed several wrong ideas for proving this theorem over the years, and who patiently listened to one (finally) right idea over the course of several months in our seminar; and to Ilijas Farah, who provided invaluable feedback on an earlier draft of this paper.

\section{A blueprint}\label{sec:Prelims}

This section examines the analogue of the lifting problem for $\< \pwmff,\s \>$, a problem of algebraic dynamics, in the simpler category of Boolean algebras. 
This simplified version of the lifting problem was completely solved about 60 years ago in the work of Parovi\v{c}enko \cite{Parovicenko}, J\'onsson and Olin \cite{JO}, and Keisler \cite{Keisler}. 
We review the main points of their solution here. 
Our reason for doing this is that the proof of our main theorem, in its broad strokes, is built from the blueprint laid out in this section. 
However, obstacles are present in the dynamical category that do not appear in the simpler algebraic setting, and these obstacles force us to deviate from this pattern. 
Before understanding our approach to the lifting problem for $\< \pwmff,\s \>$, it is useful to understand the pattern we are building from, and to see both where and why it is necessary to abandon it.

\begin{definition}\label{def:AlgebraicLP}
Let $\CC$ be a Boolean algebra. An \emph{instance of the lifting problem} for $\CC$ is a $4$-tuple
$(\AA,\BB,\iota,\eta),$
where $\AA$ and $\BB$ are countable Boolean algebras, $\iota$ is an embedding $\AA \to \BB$, and $\eta$ is an embedding $\AA \to \CC$. 
A \emph{solution} to this instance of the lifting problem is an embedding $\bar \eta: \BB \to \CC$ such that $\bar \eta \circ \iota = \eta$.

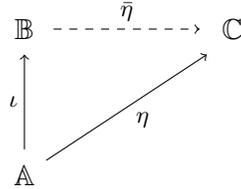
\vspace{-2mm}

\begin{center}
\begin{tikzpicture}[scale=.97]

\node at (0,0) {$\AA$};
\node at (0,2) {$\BB$};
\node at (2.85,2) {$\CC$};

\draw[dashed,->] (.4,2) -- (2.4,2);
\draw[->] (.3,.2) -- (2.5,1.65);
\draw[<-] (0,1.65) -- (0,.35);

\node at (-.15,1) {\footnotesize $\iota$};
\node at (1.62,.75) {\footnotesize $\eta$};
\node at (1.4,2.24) {\footnotesize $\bar \eta$};

\end{tikzpicture}
\end{center}

\noindent In this situation we say that $\bar \eta$ is a \emph{lifting} of $\eta$ from $\AA$ to $\BB$. 
\hfill
\Coffeecup
\end{definition}

Equivalently, one may think of $\eta$ as an isomorphism from $\AA$ onto some subalgebra $\CC_0 \sub \CC$, and a lifting $\bar \eta$ as an isomorphism from $\BB$ onto a larger subalgebra $\CC_1 \supseteq \CC_0$ such that $\bar \eta \circ \iota = \eta$.

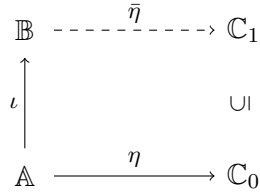
\begin{center}
\begin{tikzpicture}[scale=.97]

\node at (0,0) {$\AA$};
\node at (0,2) {$\BB$};
\node at (3,2) {$\CC_1$};
\node at (3,0) {$\CC_0$};

\draw[dashed,->] (.4,2) -- (2.6,2);
\draw[->] (.4,0) -- (2.6,0);
\draw[<-] (0,1.62) -- (0,.38);

\node at (-.15,1) {\footnotesize $\iota$};
\node at (1.5,.23) {\footnotesize $\eta$};
\node at (1.5,2.23) {\footnotesize $\bar \eta$};

\node at (2.95,1) {\footnotesize \rotatebox{90}{$\sub$}};

\end{tikzpicture}
\end{center}

\begin{theorem}[Parovi\v{c}enko, 1963 \cite{Parovicenko}]\label{thm:Parovicenko}
Every instance of the lifting problem for $\pwmff$ has a solution. 
\end{theorem}

Let us note that $\pwmff$ also has the (closely related) property of being $\aleph_1$-saturated. (We will say more on saturation in Section~\ref{sec:Corollaries} below.) 
This was proved by Keisler in \cite{Keisler} and independently, two years later, by J\'{o}nsson and Olin in \cite{JO}. 
Let us consider some consequences of Theorem~\ref{thm:Parovicenko}. 

\begin{theorem}\label{thm:AlgebraicUniv}
Suppose $\CC$ is a Boolean algebra such that every instance of the lifting problem for $\CC$ has a solution. 
Then every Boolean algebra of size $\leq\!\aleph_1$ embeds in $\CC$. 
\end{theorem}
\begin{proof}
Let $\AA$ be a Boolean algebra of size $\leq\!\aleph_1$, and write $\AA$ as the union of a length-$\w_1$ sequence $\seq{\AA_\a}{\a < \w_1}$ of subalgebras, such that $\AA_0 = \{\mathbf{0},\mathbf{1}\}$, and $\AA_\a \sub \AA_\b$ whenever $\a \leq \b$, and $\AA_\a = \bigcup_{\xi < \a}\AA_\xi$ for limit $\a < \w_1$. 

Using transfinite recursion, we now build a sequence $\seq{\eta_\a}{\a < \w_1}$ of maps such that each $\eta_\a$ is an embedding $\AA_\a \to \CC$, and the sequence is ``coherent'' in the sense that $\eta_\a = \eta_\b \rest \AA_\a$ whenever $\a \leq \b$. The base step is trivial: there is an embedding $\eta_0: \AA_0 \to \CC$ because $\AA_0 = \{\mathbf{0},\mathbf{1}\}$. At successor steps, observe that $(\AA_\a,\AA_{\a+1},\iota,\eta_\a)$ is an instance of the lifting problem for $\CC$, taking $\iota$ to be the inclusion map $\AA_\a \xhookrightarrow{} \AA_{\a+1}$. 
By assumption, there is an embedding $\eta_{\a+1}: \AA_{\a+1} \to \CC$ that extends $\eta_\a$. 
At limit stages there is nothing to do: because $\AA_\a = \bigcup_{\xi < \a}\AA_\xi$, we can (and must, to preserve coherence) define $\eta_\a = \bigcup_{\xi < \a} \eta_\xi$, which is well-defined because the $\eta_\xi$'s are coherent. 
In the end, $\eta = \bigcup_{\a < \w_1}\eta_\a$ is an embedding $\AA \to \CC$. 
\end{proof}

In particular, combining this with Parovi\v{c}enko's Theorem~\ref{thm:Parovicenko}, we see that every Boolean algebra of size $\leq\!\aleph_1$ embeds in $\pwmff$.

\begin{theorem}\label{thm:AlgebraicB&F}
Suppose $\CC$ and $\DD$ are Boolean algebras of size $\aleph_1$, and every instance of the lifting problem for $\CC$ or for $\DD$ has a solution. Then $\CC \iso \DD$.
\end{theorem}
\begin{proof}
Let $\seq{c_\xi}{\xi < \w_1}$ be an enumeration of $\CC$, and let $\seq{d_\xi}{ \xi < \w_1 }$ be an enumeration of $\DD$. 
Using transfinite recursion, we now build a sequence $\seq{\phi_\a}{\a < \w_1}$ of maps such that each $\phi_\a$ is an isomorphism of its domain (which we denote $\CC_\a$) to its codomain (which we denote $\DD_\a$); furthermore, we shall have $\CC_\a \sub \CC_\b$ and $\DD_\a \sub \DD_\b$ whenever $\a \leq \b$, $\CC_\a = \bigcup_{\xi < \a}\CC_\xi$ and $\DD_\a = \bigcup_{\xi < \a}\DD_\xi$ for limit $\a$, and the $\phi_\a$'s will be coherent, in the sense that $\phi_\a = \phi_\b \rest \CC_\a$ whenever $\a \leq \b$. 
Given $X \sub \CC$, let $\<\hspace{-1mm}\< X \>\hspace{-1mm}\>$ 
denote the subalgebra of $\CC$ generated by $X$, and similarly for $\DD$.

For the base case, let $\CC_0 = \{ \mathbf{0}_\CC,\mathbf{1}_\CC \}$ and $\DD_0 = \{ \mathbf{0}_\DD,\mathbf{1}_\DD \}$, and let $\phi_0$ denote the unique isomorphism $\CC_0 \to \DD_0$. 
At limit stages there is nothing to do: take $\CC_\a = \bigcup_{\xi < \a}\CC_\xi$, $\DD_\a = \bigcup_{\xi < \a}\DD_\xi$, and $\phi_\a = \bigcup_{\xi < \a}\phi_\xi$, which is well-defined because the $\phi_\xi$'s are coherent. 

At successor steps we obtain $\phi_{\a+1}$ from $\phi_\a$ by solving the lifting problem twice, once in the forwards direction and once in the backwards direction. 
First, let $\DD_{\a+1}^0 = \<\hspace{-1mm}\< \DD_\a \cup \{d_\a\} \>\hspace{-1mm}\>$ and observe that $(\DD_\a,\DD^0_{\a+1},\iota,\phi_\a^{-1})$ is an instance of the lifting problem for $\CC$, where $\iota$ denotes the inclusion $\DD_\a \xhookrightarrow{} \DD_{\a+1}^0$. 
Consequently, there is an isomorphism $\phi^0_{\a+1}$ from $\DD_{\a+1}^0$ to some subalgebra $\CC^0_{\a+1}$ of $\CC$ with $\CC^0_{\a+1} \supseteq \CC_\a$. 
Second, let $\CC_{\a+1} = \<\hspace{-1mm}\< \CC_\a^0 \cup \{c_\a\} \>\hspace{-1mm}\>$ and observe that $(\CC_{\a+1}^0,\CC_{\a+1},\iota,(\phi_\a^0)^{-1})$ is an instance of the lifting problem for $\DD$, where $\iota$ denotes the inclusion $\CC_{\a+1}^0 \xhookrightarrow{} \CC_{\a+1}$. 
Thus there is an isomorphism $\phi_{\a+1}$ from $\CC_{\a+1}$ to some subalgebra $\DD_{\a+1}$ of $\DD$ with $\DD_{\a+1} \supseteq \DD_{\a+1}^0$. 

\begin{center}
\begin{tikzpicture}[scale=.8]

\node at (-.01,0) {\small $\CC_\a$};
\node at (3.99,0) {\small $\DD_\a$};
\node at (-.21,2) {\small $\CC_{\a+1}^0$};
\node at (5.7,2) {\small $\DD_{\a+1}^0 = \<\hspace{-1mm}\< \DD_\a \cup \{d_\a\} \>\hspace{-1mm}\>$};
\node at (-.21,3.5) {\small $\CC_{\a+1}^0$};
\node at (4.2,3.5) {\small $\DD_{\a+1}^0$};
\node at (-.21,5.5) {\small $\CC_{\a+1}$};
\node at (4.21,5.5) {\small $\DD_{\a+1}$};
\node at (6.28,5.5) {\small $= \mathrm{Image}(\phi_{\a+1})$};
\node at (-2.35,2) {\small $\mathrm{Image}(\phi_{\a+1}^0) =$};
\node at (-2.53,5.5) {\small $\<\hspace{-1mm}\< \CC_{\a+1}^0 \cup \{c_\a\} \>\hspace{-1mm}\> =$};

\draw[<-] (.6,0) -- (3.4,0);
\draw[<-] (.6,2) -- (3.4,2);
\draw[->] (.6,3.5) -- (3.4,3.5);
\draw[->] (.6,5.5) -- (3.4,5.5);

\node at (3.95,1) {\footnotesize \rotatebox{90}{$\sub$}};
\node at (-.3,2.75) {\footnotesize \rotatebox{90}{$=$}};
\node at (4,2.75) {\footnotesize \rotatebox{90}{$=$}};
\node at (-.3,4.5) {\footnotesize \rotatebox{90}{$\sub$}};

\node at (2,.3) {\footnotesize $\phi_\a^{-1}$};
\node at (2,2.3) {\footnotesize $\phi_{\a+1}^0$};
\node at (2,3.8) {\footnotesize $(\phi_{\a+1}^0)^{-1}$};
\node at (2,5.8) {\footnotesize $\phi_{\a+1}$};

\end{tikzpicture}
\end{center}

This completes the recursion. Because the $\phi_\a$ are coherent partial isomorphisms, $\phi = \bigcup_{\a < \w_1}\phi_\a$ is an isomorphism from its domain $\bigcup_{\a < \w_1} \CC_\a$ to its image $\bigcup_{\a < \w_1} \DD_\a$. 
Furthermore, $\bigcup_{\a < \w_1} \CC_\a = \CC$, because we ensured at stage $\a+1$ of the recursion that $c_\a \in \CC_{\a+1}$, and $\bigcup_{\a < \w_1} \DD_\a = \DD$, because we ensured at stage $\a+1$ of the recursion that $d_\a \in \DD_{\a+1}$. 
\end{proof}

\begin{corollary}\label{cor:Par}
Assuming \ch, $\pwmff$ is (up to isomorphism) the unique Boolean algebra of size $\aleph_1$ for which every instance of the lifting problem has a solution.
\end{corollary}
\begin{proof}
This follows immediately from Theorems \ref{thm:Parovicenko} and \ref{thm:AlgebraicB&F}.
\end{proof}

Sometimes the name ``Parovi\v{c}enko's theorem'' refers to Corollary~\ref{cor:Par}, or sometimes to the following statement in the topological category: 

\begin{itemize}
\item[$\ $] Assuming \ch, every crowded zero-dimensional compact Hausdorff $F$-space of weight $\mathfrak{c}$ in which nonempty $G_\dlt$'s have nonempty interior is homeomorphic to $\w^*$.
\end{itemize} 
\noindent This is, more or less, the Stone dual of Corollary~\ref{cor:Par}. 
Despite its topological flavor, this latter form of Parovi\v{c}enko's theorem is most naturally proved in the algebraic category (see, e.g., \cite[Section 1.1]{vanMill}). 
A direct topological proof, including a topological analogue of Theorem~\ref{thm:Parovicenko}, can be found in \cite{BS}.

\section{Incompressibility}\label{sec:Incompressibility}

The results in the previous section suggest a natural approach to proving our main theorem, namely, by proceeding in the dynamical category via analogy with the Boolean-algebraic category: 
\begin{enumerate}
\item First, show that every instance of the lifting problem for $\< \pwmff,\s \>$ and for $\< \pwmff,\s^{-1} \>$ has a solution (analogous to Theorem~\ref{thm:Parovicenko}), once this lifting problem has been suitably defined.
\item Consequently, by an argument more or less identical to our proofs of Theorem~\ref{thm:AlgebraicB&F} and Corollary~\ref{cor:Par}, \ch implies these two structures are isomorphic.
\end{enumerate}

\noindent In this section, we turn our attention to the phrase ``suitably defined'' in $(1)$. To answer the question of how the lifting problem for $\< \pwmff,\s \>$ should be defined, we need to look at the dynamical notion of incompressibility. 

\begin{definition}
A dynamical system $\< \AA,\a \>$ is \emph{compressible} if there is some $x \in \AA \setminus \{ \mathbf{0},\mathbf{1} \}$ such that $\a(x) \leq x$.
Otherwise $\< \AA,\a \>$ is \emph{incompressible}.
\hfill{\Coffeecup}
\end{definition}

\begin{theorem}\label{thm:It'sIncompressible}
$\< \pwmff,\s \>$ and $\< \pwmff,\s^{-1} \>$ are incompressible.
\end{theorem}
\begin{proof}
This can be found in \cite[Theorem 5.3]{Brian0}; but the proof is short and not difficult, so we reproduce it here.

Let $a \in \pwmff$ with $a \neq \mathbf{0},\mathbf{1}$. In other words, let $a$ be the equivalence class $[A]$ of some $A \sub \w$ that is neither finite (since then $[A] = [\0] = \mathbf{0}$) nor co-finite (since then $[A] = [\w] = \mathbf{1}$). Because $A$ is neither finite nor co-finite, the set $B = \set{n \in A}{n+1 \notin A}$ is infinite. But then
$$\s([A]) - [A] \,=\, [A+1] - [A] \,=\, [(A+1) \setminus A] \,=\, [B+1] \,\neq\, \mathbf{0},$$
which means that $\s(a) = \s([A]) \not\leq [A] = a$. 

A similar proof works for $\< \pwmff,\s^{-1} \>$. Alternatively, one may observe that, in general, $\< \AA,\a \>$ is compressible if and only if $\< \AA,\a^{-1} \>$ is, because if $\a(x) \leq x$ then $\a^{-1}(\mathbf{1}-x) \leq \mathbf{1}-x$. 
\end{proof}

\begin{theorem}\label{thm:Incompressibility}
If $\<\AA,\a\>$ embeds in an incompressible dynamical system, it is incompressible. 
\end{theorem}
\begin{proof}
Like the previous theorem, this is a known fact and can be found (in dual form) in \cite[Chapter 4]{Akin} or \cite[Section 5]{Brian0}. But like the previous theorem, the proof is short and not difficult, so we reproduce it here.

Suppose $\eta$ is an embedding of $\< \AA,\a \>$ into an incompressible dynamical system $\< \BB,\b \>$, and let $x \in \AA \setminus \{ \mathbf{0},\mathbf{1} \}$. Then $\b(\eta(x)) \not\leq \eta(x)$ (by the incompressibility of $\< \BB,\b \>$), which implies $\eta(\a(x)) \not\leq \eta(x)$ (because $\eta \circ \a = \b \circ \eta$), which implies $\a(x) \not\leq x$ (because $\eta^{-1}$ preserves inequalities). Thus $\< \AA,\a \>$ is incompressible.
\end{proof}

\begin{corollary}\label{cor:Substructures}
Every dynamical system that embeds in $\< \pwmff,\s \>$ or $\< \pwmff,\s^{-1} \>$ is incompressible.
\end{corollary} 

For example, of the two dynamical systems on the $4$-element Boolean algebra, one of them embeds in $\< \pwmff,\s \>$ and the other one does not. 

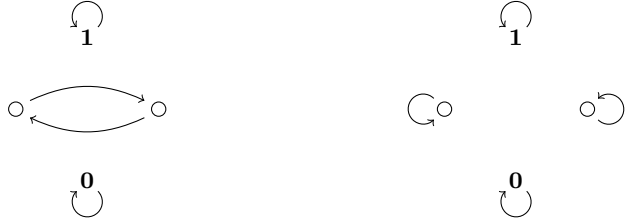
\vspace{2mm}

\begin{center}
\begin{tikzpicture}[scale=.95]

\node at (1,-1) {\footnotesize $\mathbf{0}$};
\node at (1,1) {\footnotesize $\mathbf{1}$};
\draw (0,0) ellipse (1mm and 1mm);
\draw (2,0) ellipse (1mm and 1mm);

\draw[->] (1.15,1.15) arc (-45:225:2.1mm);
\draw [->] (.2,.12) to [out=25,in=155] (1.8,.12);
\draw [<-] (.2,-.12) to [out=-25,in=-155] (1.8,-.12);
\draw[->] (1.15,-1.15) arc (45:-225:2.1mm);

\node at (7,-1) {\footnotesize $\mathbf{0}$};
\node at (7,1) {\footnotesize $\mathbf{1}$};
\draw (6,0) ellipse (1mm and 1mm);
\draw (8,0) ellipse (1mm and 1mm);

\draw[->] (7.15,1.15) arc (-45:225:2.1mm);
\draw[->] (5.85,.15) arc (45:315:2.1mm);
\draw[->] (8.15,-.15) arc (-135:135:2.1mm);
\draw[->] (7.15,-1.15) arc (45:-225:2.1mm);

\end{tikzpicture}
\end{center}

\vspace{1mm}

\noindent In fact, the dynamical system illustrated on the left admits exactly two (essentially identical) embeddings into $\< \pwmff,\s \>$: $\mathbf 0$ must map to $[\0]$ and $\mathbf 1$ to $[\w]$, and the two other elements of the algebra must map to $[\set{2n}{n \in \w}]$ and $[\set{2n+1}{n \in \w}]$.

The most direct analogue of Definition~\ref{def:AlgebraicLP} for dynamical systems is: 

\begin{definition}\label{def:GeneralLP}
An \emph{instance of the general lifting problem} for a dynamical system $\<\CC,\g\>$ is a $4$-tuple 
$(\<\AA,\a\>,\<\BB,\b\>,\iota,\eta),$
where $\<\AA,\a\>$ and $\<\BB,\b\>$ are countable dynamical systems, $\iota$ is an embedding of $\<\AA,\a\>$ into $\<\BB,\b\>$, and $\eta$ is an embedding of $\<\AA,\a\>$ into $\< \CC,\g \>$. 

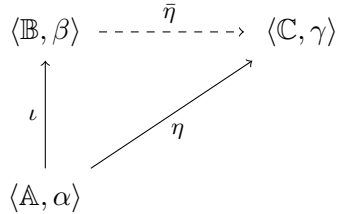
\begin{center}
\begin{tikzpicture}[scale=1.1]

\node at (0,0) {$\<\AA,\a\>$};
\node at (0,2) {$\<\BB,\b\>$};
\node at (3.1,2) {$\<\CC,\g\>$};

\draw[dashed,->] (.65,2) -- (2.4,2);
\draw[->] (.55,.35) -- (2.5,1.65);
\draw[<-] (0,1.65) -- (0,.35);

\node at (-.15,1) {\footnotesize $\iota$};
\node at (1.6,.8) {\footnotesize $\eta$};
\node at (1.5,2.22) {\footnotesize $\bar \eta$};

\end{tikzpicture}
\end{center}

\noindent A \emph{solution} to this instance of the general lifting problem is an embedding $\bar \eta$ from $\<\BB,\b\>$ into $\<\CC,\g\>$ such that $\bar \eta \circ \iota = \eta$.
\hfill
\Coffeecup
\end{definition}

However, allowing $\< \AA,\a \>$ and $\< \BB,\b \>$ to be any countable dynamical systems is too broad, because an analogue of Theorem~\ref{thm:AlgebraicUniv} holds in this context:

\begin{theorem}
Suppose $\< \CC,\g \>$ is a dynamical system such that every instance of the general lifting problem for $\< \CC,\g \>$ has a solution. Then every dynamical system of size $\leq\!\aleph_1$ embeds in $\<\CC,\g\>$.
\end{theorem}
\begin{proof}
This is a straightforward adaptation of the proof of Theorem~\ref{thm:AlgebraicUniv}.
\end{proof} 

As we have already seen, not every countable dynamical system, or even every finite dynamical system, embeds in $\< \pwmff,\s \>$ and $\< \pwmff,\s^{-1} \>$; only incompressible systems do. This observation motivates our definition of the lifting problem for $\< \pwmff,\s \>$. 

\begin{definition}\label{def:LP}
An \emph{instance of the lifting problem} for $\<\pwmff,\s\>$ is a $4$-tuple 
$(\<\AA,\a\>,\<\BB,\b\>,\iota,\eta),$
where $\<\AA,\a\>$ and $\<\BB,\b\>$ are countable incompressible dynamical systems, $\iota$ is an embedding of $\<\AA,\a\>$ into $\<\BB,\b\>$, and $\eta$ is an embedding of $\<\AA,\a\>$ into $\< \pwmff,\s \>$. 

\begin{center}
\begin{tikzpicture}[scale=1.2]

\node at (0,0) {$\<\AA,\a\>$};
\node at (0,2) {$\<\BB,\b\>$};
\node at (3.4,2) {$\<\pwmff,\s\>$};

\draw[dashed,->] (.65,2) -- (2.4,2);
\draw[->] (.55,.35) -- (2.5,1.65);
\draw[<-] (0,1.65) -- (0,.35);

\node at (-.15,1) {\footnotesize $\iota$};
\node at (1.6,.8) {\footnotesize $\eta$};
\node at (1.5,2.22) {\footnotesize $\bar \eta$};

\end{tikzpicture}
\end{center}

\noindent A \emph{solution} to this instance of the lifting problem for $\< \pwmff,\s \>$ is an embedding $\bar \eta$ from $\<\BB,\b\>$ into $\<\pwmff,\s\>$ such that $\bar \eta \circ \iota = \eta$. 

An \emph{instance of the lifting problem for $\< \pwmff,\s^{-1} \>$} is defined similarly, as are the notions of \emph{solution} and \emph{lifting} in that context.
\hfill
\Coffeecup
\end{definition}

Once again, an analogue of Theorem~\ref{thm:AlgebraicUniv} holds in this context: 

\begin{theorem}\label{thm:AU2}
{If} every instance of the lifting problem for $\< \pwmff,\s \>$ has a solution, {then} every incompressible dynamical system of size $\leq\!\aleph_1$ embeds in $\<\pwmff,\s\>$.
\end{theorem}

Similarly, an analogue of Theorem~\ref{thm:AlgebraicB&F} and Corollary~\ref{cor:Par} is also provable in this context. The proof is a straightforward adaptation of the proofs of Theorem~\ref{thm:AlgebraicB&F} and Corollary~\ref{cor:Par} in the previous section.

\begin{theorem}\label{thm:BF2}
If every instance of the lifting problem for $\< \pwmff,\s \>$ and for $\< \pwmff,\s^{-1}\>$ has a solution, then \ch implies that $\< \pwmff,\s \>$ and $\< \pwmff,\s^{-1}\>$ are conjugate.
\end{theorem}


To complete step $(1)$ in the strategy outlined at the beginning of this section, we should now prove that every instance of the lifting problem for $\< \pwmff,\s \>$ and for $\< \pwmff,\s^{-1}\>$ has a solution. The paper's main theorem then follows via Theorem~\ref{thm:BF2}. 

But alas, it turns out this is too much to hope for. 
We prove in Section 5 below that not every instance of the lifting problem for $\< \pwmff,\s \>$ has a solution. 
This forces us to travel a more difficult road in order to reach a proof of the main theorem. 

Interestingly, Theorems~\ref{thm:AU2} and \ref{thm:BF2} are both vacuously true. 
That is, both theorems assert conditional statements -- \emph{if $A$ then $B$} -- 
and in both cases, $A$ is false while $B$ is true nonetheless. 
So while not every instance of the lifting problem for $\< \pwmff,\s \>$ has a solution, the most important consequences of that statement hold anyway. There is a reason for this, as we shall see: all ``sufficiently nice'' instances of the lifting problem for $\< \pwmff,\s \>$ do have a solution, and this suffices to derive the conclusions of Theorems~\ref{thm:AU2} and \ref{thm:BF2}, albeit by more nuanced arguments. 

A rough outline of the next few sections is:
\begin{itemize}
\item[$\S$\ref{sec:Embeddings}:] We develop some techniques for constructing embeddings of a countable dynamical system into $\< \pwmff,\s \>$ and $\< \pwmff,\s^{-1} \>$.
\item[$\S$\ref{sec:NonSaturation}:] We identify an instance of the lifting problem for $\< \pwmff,\s \>$ that does not have a solution. In other words, we show the hypotheses of Theorems \ref{thm:AU2} and \ref{thm:BF2} do not hold. 
\item[$\S$\ref{sec:back&forth}:] We show that, nevertheless, if every ``sufficiently nice'' instance of the lifting problem has a solution, then \ch implies $\< \pwmff,\s \>$ and $\< \pwmff,\s^{-1}\>$ are conjugate, via a modification of the back-and-forth argument used for Theorem~\ref{thm:AlgebraicB&F}.
\end{itemize}

In other words, the strategy outlined at the beginning of this section fails in step $(1)$, because the natural analogue of Theorem~\ref{thm:Parovicenko} for $\< \pwmff,\s \>$ and $\< \pwmff,\s^{-1} \>$ is not true. But we can rescue the strategy by considering only ``sufficiently nice'' instances of the lifting problem. 

The fact that ``sufficiently nice'' instances of the lifting problem are solvable is not proved in Section 6, but in Sections 7-11. The main content of Section 6 is to define precisely what ``sufficiently nice'' means, and to prove that being able to solve only these instances of the lifting problem still suffices to imply the main theorem. 

The conclusion of Theorem~\ref{thm:AU2}, that every weakly incompressible dynamical system of size $\leq\!\aleph_1$ embeds in $\< \pwmff,\s \>$, is the main result of \cite{Brian1}. 
But this result is also a relatively straightforward consequence of some of the techniques used in this paper, and we include a proof of it at the end of Section~\ref{sec:Diagonalization} below.

\section{Embeddings of Boolean algebras and dynamical systems}\label{sec:Embeddings}

Recall that a dynamical system $\< \AA,\a \>$ {embeds} in another dynamical system $\< \BB,\b \>$ if there is a Boolean-algebraic embedding $\eta: \AA \to \BB$ such that $\eta \circ \a = \b \circ \eta$. 
In other words, an embedding of $\< \AA,\a \>$ into $\< \BB,\b \>$ is a conjugacy from $\< \AA,\a \>$ onto a subsystem of $\< \BB,\b \>$.

In this section, we develop some techniques for constructing embeddings of dynamical systems into $\< \pwmff,\s \>$ and $\< \pwmff,\s^{-1} \>$. Ultimately, the goal is to show that every embedding of a countable dynamical system $\< \AA,\a \>$ into $\< \pwmff,\s \>$ or $\< \pwmff,\s^{-1} \>$ is completely described by a particular kind of sequence in $\AA$ (see Theorem~\ref{thm:CharacterizingGoodSequences}).

In what follows, we have no need to consider infinite partitions of unity in a Boolean algebra, but will work quite a lot with finite partitions. To avoid repeating the word ``finite'' many times, we will use the word ``partition'' to always mean what is usually meant by ``finite partition."

\begin{definition}
A \emph{partition of unity}, or simply a \emph{partition}, in a Boolean algebra $\AA$ is a finite set $\U$ of nonzero members of $\AA$ such that $\bigvee \U = \1$ and $u \wedge v = \zero$ for all distinct $u,v \in \U$. 

Given a partition $\U$ of $\AA$ and some $a \in \AA$, we say that $a$ is \emph{below} $\U$ if $a \leq u$ for some $u \in \U$. 
If $\U$ and $\V$ are two partitions of $\AA$, $\V$ \emph{refines} $\U$ if every member of $\V$ is below $\U$. 
In this case, the \emph{natural map} from $\V$ to $\U$ sends each $v \in \V$ to the unique $u \in \U$ with $v \leq u$.

A sequence $\seq{a_n}{n \in \w}$ in $\AA$ is \emph{eventually small} if for every partition $\U$ of $\AA$, there is some $N \in \w$ such that $a_n$ is below $\U$ for all $n \geq N$. 

A sequence $\seq{a_n}{n \in \w}$ in $\AA$ is \emph{eventually dense} if for every $x \in \AA$, there are infinitely many $n \in \w$ such that $a_n \leq x$. 
\hfill{\Coffeecup}
\end{definition}

\begin{theorem}\label{thm:InducedMap}
Suppose $\AA$ is a Boolean algebra, and $\seq{a_n}{n \in \w}$ is an eventually small, eventually dense sequence in $\AA \setminus \{\zero\}$.
Then the mapping
$a \,\mapsto\, [\set{n \in \w}{a_n \leq a}]$ 
is an embedding $\AA \to \pwmff$.
\end{theorem}
\begin{proof}
Suppose $\seq{a_n}{n \in \w}$ is an eventually small, eventually dense sequence in $\AA \setminus \{\zero\}$, and define $\eta: \AA \to \pwmff$ by setting $\eta(a) = [\set{n \in \w}{a_n \leq a}]$.

The definition of $\eta$ readily implies $\eta(\mathbf{0}) = [\0]$ and $\eta(\mathbf{1}) = [\w]$. 
Now suppose $a,b \in \AA$. Let $\U = \{ a-b,b-a,a \wedge b, \mathbf{1} - (a \vee b)\} \setminus \{\mathbf{0}\}$ (the partition of $\AA$ generated by $a$ and $b$). 
Because $\seq{a_n}{n \in \w}$ is eventually small, there is some $N \in \w$ such that if $n \geq N$ then $a_n$ is below $\U$. It follows that
\begin{align*}
\set{n \in \w}{a_n \leq a \vee b} =^* & \\
\set{n \in \w}{a_n \leq a-b} &\cup \set{n \in \w}{a_n \leq b-a} \cup \set{n \in \w}{a_n \leq a \wedge b} \\
=\, \set{n \in \w}{a_n \leq a} &\cup \set{n \in \w}{a_n \leq b}
\end{align*}
which means that
\begin{align*}
\eta(a \vee b) &\,=\, [\set{n \in \w}{a_n \leq a \vee b}] 
 \,=\, [\set{n \in \w}{a_n \leq a} \cup \set{n \in \w}{a_n \leq b}] \\
&\,=\, [\set{n \in \w}{a_n \leq a}] \vee [\set{n \in \w}{a_n \leq b}] \,=\, \eta(a) \vee \eta(b).
\end{align*}
A similar computation shows $\eta(a \wedge b) = \eta(a) \wedge \eta(b)$ and $\eta(\mathbf 1 - a) = \mathbf 1 - \eta(a)$.

Thus $\eta$ is a homomorphism $\AA \to \pwmff$. To check that it is injective, suppose $a \in \AA$ and $a \neq \zero$. 
Because $\seq{a_n}{n \in \w}$ is eventually dense, 
$\set{n}{a_n \leq a}$ is infinite, which means $\eta(a) = [\set{n}{a_n \leq a}] \neq [\0]$.
In other words, the kernel of the homomorphism $\eta$ is $\{\zero\}$, and this implies $\eta$ is injective.
\end{proof}

\begin{definition}\label{def:Induced}
Suppose $\AA$ is a Boolean algebra, and $\vec s = \seq{a_n}{n \in \w}$ is an eventually small, eventually dense sequence in $\AA \setminus \{\zero\}$.
The mapping
$$a \,\mapsto\, [\set{n \in \w}{a_n \leq a}]$$
is called the embedding \emph{induced} by $\vec s$.
If $\eta: \AA \to \pwmff$ is an embedding, we say that $\eta$ is \emph{induced} if there is a sequence $\vec s$ in $\AA$ that induces it.
\hfill{\Coffeecup}
\end{definition}

Observe that if $\seq{a_n}{n \in \w}$ is an eventually dense sequence in $\AA$, then taking $\F_n = \set{x \in \AA}{a_n \leq x}$ gives a countable family of filters whose union is $\AA \setminus \{\mathbf 0\}$; hence $\AA$ is $\s$-centered. 
Contrapositively, if $\AA$ is not $\s$-centered, then there is no eventually dense sequence in $\AA$, and no embedding $\AA \to \pwmff$ can be induced. 
For example, there are no induced mappings $\pwmff \to \pwmff$ in the sense of Definition~\ref{def:Induced}. 
In contrast to this, the following theorem (Theorem~\ref{thm:CBA}) shows that if $\AA$ is a countable Boolean algebra, then every embedding $\AA \to \pwmff$ is induced. 

Theorem~\ref{thm:CBA} is well known. Its Stone dual states that $\w^*$ continuously surjects onto every compact zero-dimensional metric space $X$, and every such surjection can be continuously extended to a map $\b\w \to X$. We include a proof (in the algebraic category) for the sake of completeness.

\begin{lemma}\label{lem:NicePartitionSequence}
If $\AA$ is a countable Boolean algebra, then there is a sequence $\seq{\U_n}{n \in \w}$ of partitions of $\AA$ such that $\U_n$ refines $\U_m$ whenever $m \leq n$, and for every partition $\V$ of $\AA$, $\U_n$ refines $\V$ for all sufficiently large $n$.
\end{lemma}
\begin{proof}
Let $\seq{a_k}{k \in \w}$ be an enumeration of $\AA \setminus \{\mathbf{0}\}$, and for each $n \in \w$ let $\U_n$ be the partition of $\AA$ generated by $\set{a_k}{k \leq n}$. 
\end{proof}

\begin{theorem}\label{thm:CBA}
If $\AA$ is a countable Boolean algebra, then
\begin{enumerate}
\item there is an embedding of $\AA$ into $\pwmff$, and
\item every embedding of $\AA$ into $\pwmff$ is induced.
\end{enumerate}
\end{theorem}
\begin{proof}
Let $\AA$ be a countable Boolean algebra. Fix a sequence $\U_0,\U_1,\U_2,\dots$ of partitions of $\AA$ with the properties stated in Lemma~\ref{lem:NicePartitionSequence}.

For $(1)$, let $\seq{a_n}{n \in \w}$ be a sequence of members of $\AA$ obtained by simply listing the (finitely many) members of $\U_0$, then listing the members of $\U_1$, then $\U_2$, etc. This sequence is both eventually small and eventually dense, so it induces an embedding $\AA \to \pwmff$ by Theorem~\ref{thm:InducedMap}.

For $(2)$, suppose $\eta: \AA \to \pwmff$ is an embedding. 
For each $a \in \AA$ fix a representative $X_a$ of the equivalence class $\eta(a)$: that is, $X_a \sub \w$ and $\eta(a) = [X_a]$. 
For convenience, define $n_{-1} = 0$, and then, for each $k$, define $n_k$ to be some natural number large enough that: 
\begin{enumerate}
\item $n_k > n_{k-1}$.
\item For each $a \in \U_k$, $X_a \cap [n_{k-1},n_k) \neq \0$.
\item If $a \in \AA$ and $a = \bigvee \set{u \in \U_k}{u \leq a}$, then 
$$\set{X_u \setminus n_k}{u \in \U_{k+1} \text{ and } u \leq a}$$ 
is a partition of $X_a \setminus n_k$.
\end{enumerate}
It is clear that all sufficiently large $n_k$ satisfy $(1)$ and $(2)$. For condition $(3)$, let $A = \set{u \in \U_{k+1}}{u \leq a}$, suppose $a = \bigvee A$, 
and recall that, in $\pwmff$, 
$$[X_u] \wedge [X_{u'}] = \eta(u) \wedge \eta(u') = \eta(u \wedge u') = \eta(\zero) = [\0]$$
for all $u,u' \in A$ with $u \neq u'$. In other words, the sets in $\set{X_u}{u \in A}$ are ``pairwise almost disjoint" in the sense that $\set{X_u \setminus n}{u \in A}$ are pairwise disjoint above $N$ for large enough $N$. Similarly, $[X_a] = \bigvee_{u \in A}[X_u]$ in $\pwmff$, and this implies $\bigcup \set{X_u \setminus N}{u \in A} = X_a \setminus N$ for large enough $N$. 

This describes the $n_k$. Next, for each $n \in \w$, define
$$a_n = a \qquad \text{ if and only if } \qquad n \in [n_{k-1},n_k), \, a \in \U_k, \text{ and }n \in X_a.$$
We claim that this sequence is eventually small and eventually dense, and that it induces the embedding $\eta: \AA \to \pwmff$.

The sequence is eventually small because for every partition $\V$ of $\AA$, $\U_k$ refines $\V$ for all sufficiently large $k$, which implies (by the definition of the $a_n$) that $a_n$ is below $\V$ for all sufficiently large $n$.

To see that $\seq{a_n}{n \in \w}$ is eventually dense, fix some $a \in \AA \setminus \{\zero\}$. There is some $K \in \w$ such that $\U_k$ refines $\{a,\mathbf 1 - a\}$ whenever $k \geq K$. But then condition $(2)$ of our choice of the $n_k$ implies that in every interval $[n_{k-1},n_k)$, there is some $n$ such that $a_n \leq a$. Hence $\set{n \in \w}{a_n \leq a}$ is infinite.

Finally, we must show that $\eta$ is induced by $\seq{a_n}{n \in \w}$.
Let $\eta'$ be the embedding $\AA \to \pwmff$ induced by $\seq{a_n}{n \in \w}$ and fix $a \in \AA$. By our choice of the $\U_k$, there is some $K \in \w$ such that $\U_k$ refines $\{a,\mathbf 1 - a\}$ whenever $k \geq K$, which implies 
$a = \bigvee \set{u \in \U_k}{u \leq a}$ for all $k \geq K$.
By part $(3)$ of our choice of the $n_k$, and by the definition of the $a_n$, this implies that if $n \geq n_K$, then $a_n \leq a$ if and only if $n \in X_a$. 
Thus
$$\eta'(a) \,=\, [\set{n \in \w}{a_n \leq a}] \,=\, [X_a \setminus n_K] \,=\, [X_a] \,=\, \eta(a).$$
As $a$ was arbitrary, this shows $\eta = \eta'$.
\end{proof}

\begin{definition}
Let $\AA$ be a Boolean algebra. If $\U$ is a partition of $\AA$ and $a,b \in \AA$, we write $a \approx_\U b$ to mean that $a \wedge u \neq \zero$ if and only if $b \wedge u \neq \zero$ for every $u \in \U$. 
Define $\seq{a_n}{n \in \w} \approx_\U^* \seq{b_n}{n \in \w}$ to mean that $a_n \approx_\U b_n$ for all sufficiently large $n$.
Two sequences $\seq{a_n}{n \in \w}$ and $\seq{b_n}{n \in \w}$ in $\AA$ are \emph{eventually close} if $\seq{a_n}{n \in \w} \approx_\U^* \seq{b_n}{n \in \w}$
for every partition $\U$ of $\AA$.
This is denoted $\seq{a_n}{n \in \w} \approx \seq{b_n}{n \in \w}$.
\hfill{\Coffeecup}
\end{definition}

\begin{theorem}\label{thm:EC}
Suppose $\AA$ is a Boolean algebra, and that $\seq{a_n}{n \in \w}$ and $\seq{b_n}{n \in \w}$ are eventually small, eventually dense sequences in $\AA$. These sequences induce the same embedding $\AA \to \pwmff$ if and only if they are eventually close.
\end{theorem}

\begin{proof}
For the ``if'' direction, fix two eventually small, eventually dense sequences $\seq{a_n}{n \in \w}$, $\seq{b_n}{n \in \w}$ in $\AA$ with $\seq{a_n}{n \in \w} \approx \seq{b_n}{n \in \w}$.
Fix some $a \in \AA$, and let $\U = \{a,\1-a\}$. 
Because $\seq{a_n}{n \in \w}$ and $\seq{b_n}{n \in \w}$ are eventually small, $a_n$ and $b_n$ are below $\U$ for all sufficiently large $n$. But also $a_n \approx_\U b_n$ for sufficiently large $n$. 
Combining these two things, if $n$ is sufficiently large then 
either $a_n,b_n \leq a$ or else $a_n,b_n \leq \mathbf 1 - a$. 
Therefore $\set{n}{a_n \leq a} =^* \set{n}{b_n \leq a}$, which means the embedding induced by $\seq{a_n}{n \in \w}$ and the embedding induced by $\seq{b_n}{n \in \w}$ both map $a$ to the same member of $\pwmff$, namely $[\set{n}{a_n \leq a}] = [\set{n}{b_n \leq a}]$. 
As $a$ was arbitrary, these two sequences induce the same embedding.

For the ``only if'' direction, fix two eventually small, eventually dense sequences $\seq{a_n}{n \in \w}$, $\seq{b_n}{n \in \w}$ in $\AA$ that induce the same embedding $\AA \to \pwmff$. 
Let $\U$ be a partition of $\AA$. 
Because $\seq{a_n}{n \in \w}$ and $\seq{b_n}{n \in \w}$ are eventually small, $a_n$ and $b_n$ are below $\U$ for sufficiently large $n$. 
In particular, $a_n$ and $b_n$ are each below exactly one member of $\U$ for all sufficiently large $n$. 
For any particular $u \in \U$, 
$$\set{n}{a_n \leq u} \,=^*\, \set{n}{b_n \leq u}$$
(because the sequences induce the same embedding $\AA \to \pwmff$). 
As this is true for every $u \in \U$, $a_n$ and $b_n$ are each below the same member of $\U$ for sufficiently large $n$. Consequently, $\seq{a_n}{n \in \w} \approx_\U \seq{b_n}{n \in \w}$.
\end{proof}

\begin{theorem}\label{thm:ESED}
Suppose $\<\AA,\a\>$ is a dynamical system, and $\seq{a_n}{n \in \w}$ is an eventually small, eventually dense sequence in $\AA$. 
Let $\eta$ denote the embedding of $\AA$ into $\pwmff$ that is induced by this sequence. 
Then $\eta \circ \a = \s \circ \eta$ if and only if $\seq{a_{n+1}}{n \in \w} \approx \seq{\a(a_n)}{n \in \w}$.
\end{theorem}

In other words, this theorem gives a criterion for determining when an embedding of Boolean algebras $\AA \to \pwmff$ induced by a sequence $\seq{a_n}{n \in \w}$ is also an embedding of dynamical systems $\< \AA,\a \> \to \< \pwmff,\s \>$: namely, this holds if and only if $\seq{a_{n+1}}{n \in \w} \approx \seq{\a(a_n)}{n \in \w}$.

\begin{proof}
Let $\eta$ denote the embedding $\AA \to \pwmff$ induced by $\seq{a_n}{n \in \w}$. Observe that
$$\s^{-1} \circ \eta(a) = \s^{-1}([\set{n}{a_n \leq a}]) = [\set{n-1}{a_n \leq a}] = [\set{n}{a_{n+1} \leq a}]$$
for every $a$, which means that $\s^{-1} \circ \eta$ is the embedding $\AA \to \pwmff$ induced by the sequence $\seq{a_{n+1}}{n \in \w}$.
Likewise,
$$\eta \circ \a^{-1}(a) = [\set{n}{a_n \leq \a^{-1}(a)}] = [\set{n}{\a(a_n) \leq a}]$$
for every $a$, which means that $\eta \circ \a^{-1}$ is the embedding $\AA \to \pwmff$ induced by the sequence $\seq{\a(a_n)}{n \in \w}$.

By Theorem~\ref{thm:EC}, this implies that $\eta$ is an embedding from $\<\AA,\a^{-1}\>$ to $\<\pwmff,\s^{-1}\>$ if and only if $\seq{a_{n+1}}{n \in \w} \approx \seq{\a(a_n)}{n \in \w}$.
This proves the theorem, because $\eta$ is an embedding from $\<\AA,\a^{-1}\>$ to $\<\pwmff,\s^{-1}\>$ if and only if it is also an embedding from $\<\AA,\a\>$ to $\<\pwmff,\s\>$. (This is because $\eta \circ \a^{-1} = \s^{-1} \circ \eta$ implies $\s \circ \eta = \s \circ (\eta \circ \a^{-1}) \circ \a = \s \circ (\s^{-1} \circ \eta) \circ \a = \eta \circ \a$, and similarly $\s \circ \eta = \eta \circ \a$ implies $\eta \circ \a^{-1} = \s^{-1} \circ \eta$.)
\end{proof}


Of course, by suitably modifying the proof of this theorem, one can obtain a similar result for $\s^{-1}$:

\begin{theorem}\label{thm:ESED2}
Suppose $\<\AA,\a\>$ is a dynamical system, and $\seq{a_n}{n \in \w}$ is an eventually small, eventually dense sequence in $\AA$. 
Let $\eta$ denote the embedding of $\AA$ into $\pwmff$ that is induced by this sequence. 
Then $\eta \circ \a = \s^{-1} \circ \eta$ if and only if $\seq{a_{n-1}}{n \in \w} \approx \seq{\a(a_n)}{n \in \w}$ (where we set $a_{-1} = \mathbf{1}$ in order to make the sequence on the left well-defined).
\end{theorem}

Alternatively, one could derive Theorem~\ref{thm:ESED2} directly from Theorem~\ref{thm:ESED} using the general fact that $\eta$ is an embedding of $\< \AA,\a \>$ into $\< \BB,\b \>$ if and only if $\eta$ is also an embedding from $\< \AA,\a^{-1} \>$ into $\< \BB,\b^{-1} \>$. 

\begin{theorem}\label{thm:CharacterizingGoodSequences}
Suppose $\< \AA,\a \>$ is a countable dynamical system and $\eta$ is a Boolean-algebraic embedding $\AA \to \pwmff$. The following are equivalent:

\begin{enumerate}
\item The map $\eta$ is an embedding of $\< \AA,\a \>$ into $\< \pwmff,\s \>$.

\vspace{1mm}

\item The map $\eta$ is induced by an eventually small, eventually dense sequence $\seq{a_n}{n \in \w}$ in $\AA$ such that $\seq{\a(a_n)}{n \in \w} \approx \seq{a_{n+1}}{n \in \w}$.

\vspace{1mm}

\item The map $\eta$ is induced by an eventually small, eventually dense sequence $\seq{a_n}{n \in \w}$ in $\AA$ such that $\a(a_n) \wedge a_{n+1} \neq \zero$ for all $n \in \w$.
\end{enumerate}

\end{theorem}
\begin{proof}
We've proved already that $(1) \Leftrightarrow (2)$: that $(2) \Rightarrow (1)$ follows from Theorem~\ref{thm:ESED}, and that $(1) \Rightarrow (2)$ follows from Theorem~\ref{thm:ESED} together with Theorem~\ref{thm:CBA}(2).
It remains to prove $(2) \Leftrightarrow (3)$. 

To see that $(3) \Rightarrow (2)$, suppose $\seq{a_n}{n \in \w}$ satisfies condition $(3)$. Let $\U$ be a partition of $\AA$. 
Because $\seq{a_n}{n \in \w}$ is eventually small, $a_{n+1}$ and $\a(a_n)$ are both below $\U$ for sufficiently large $n$. 
Because $\a(a_n) \wedge a_{n+1} \neq \mathbf 0$, this implies $a_{n+1}$ and $\a(a_n)$ are below the same member of $\U$ for sufficiently large $n$. 
Hence $\a(a_n) \approx_\U^* a_{n+1}$. 
As $\U$ was arbitrary, this shows that $\seq{\a(a_n)}{n \in \w} \approx \seq{a_{n+1}}{n \in \w}$.

To see that $(2) \Rightarrow (3)$, suppose $\seq{a_n}{n \in \w}$ satisfies condition $(2)$. 
Let $\U_0,\U_1,\U_2,\dots$ be a sequence of partitions of $\AA$ with the properties stated in Lemma~\ref{lem:NicePartitionSequence}. 
For every $k \in \w$, let $\V_k$ be the partition of $\AA$ generated by the (finite) set $\U_k \cup \set{\a^{-1}(u)}{u \in \U_k}$. 
Because $\seq{a_n}{n \in \w}$ is eventually small, for every $k \in \w$ there is some $N_k \in \w$ such that $a_n$ is below $\V_k$ for all $n \geq N_k$. 
By our choice of $\V_k$, this implies that both $\a(a_n)$ and $a_{n+1}$ are below $\U_k$ whenever $n \geq N_k$. 
Furthermore, because $\seq{\a(a_n)}{n \in \w} \approx \seq{a_{n+1}}{n \in \w}$, $\a(a_n)$ and $a_{n+1}$ are below the same member of $\U_k$ whenever $n \geq N_k$. 

Define a new sequence $\seq{b_n}{n \in \w}$ in $\AA$ by setting 
$b_n$ equal to the unique $u \in \U_k$ with $a_n \leq u$ whenever $n \in [N_k,N_{k+1})$, and $b_n = \mathbf{1}$ whenever $n < N_0$. 
By the conclusion of the previous paragraph, this new sequence satisfies the condition in $(3)$.
\end{proof}

Just as with Theorems \ref{thm:ESED} and \ref{thm:ESED2}, one can obtain a version of  Theorem~\ref{thm:CharacterizingGoodSequences} for $\s^{-1}$. This can be done by suitably modifying the proof above, or by using Theorem~\ref{thm:CharacterizingGoodSequences} directly, together with the fact that $\eta$ is an embedding $\< \AA,\a \> \to \< \BB,\b \>$ if and only if $\eta$ is an embedding $\< \AA,\a^{-1} \> \to \< \BB,\b^{-1} \>$.

\begin{theorem}\label{thm:CharacterizingGoodSequences2}
Suppose $\< \AA,\a \>$ is a countable dynamical system, and let $\eta$ be an embedding $\AA \to \pwmff$. The following are equivalent:

\begin{enumerate}
\item The map $\eta$ is an embedding of $\< \AA,\a \>$ into $\< \pwmff,\s^{-1} \>$.

\vspace{1mm}

\item The map $\eta$ is induced by an eventually small, eventually dense sequence $\seq{a_n}{n \in \w}$ in $\AA$ such that $\seq{a_{n-1}}{n \in \w} \approx \seq{\a(a_n)}{n \in \w}$.

\vspace{1mm}

\item The map $\eta$ is induced by an eventually small, eventually dense sequence $\seq{a_n}{n \in \w}$ in $\AA$ such that $a_{n-1} \wedge \a(a_n) \neq \zero$ for all $n \in \w$.
\end{enumerate}

\end{theorem}

\section{A false hope cruelly dashed}\label{sec:NonSaturation}

The goal of this section is to show, as promised at the end of Section~\ref{sec:Incompressibility}, that not every instance of the lifting problem for $\< \pwmff,\s \>$ or for $\< \pwmff,\s^{-1} \>$ has a solution. 

This counterexample motivates a good deal of what follows, such as the definition of ``incompatibility'' and ``polarization'' in Section 9. Understanding the example thoroughly may help to guide one's intuition through the proof of the main theorem. That being said, this counterexample is not strictly necessary for understanding the proof of the main theorem, and this section can be skipped by the reader who is feeling impatient about getting to the proof of the main theorem. 

\begin{theorem}\label{thm:NoGo}
There is an instance of the lifting problem for $\< \pwmff,\s \>$ that has no solution.
\end{theorem}
\begin{proof}
We begin by describing a particular instance of the lifting problem for $\< \pwmff,\s \>$. Recall that an instance of the lifting problem consists of two countable incompressible dynamical systems $\<\AA,\a\>$ and $\<\BB,\b\>$, an embedding $\iota$ from $\<\AA,\a\>$ into $\<\BB,\b\>$, and an embedding $\eta$ from $\< \AA,\a \>$ into $\< \pwmff,\s \>$. Let us start by describing the Stone duals of $\AA$ and $\BB$.

The topological space $X$ (whose Stone dual gives us $\AA$) is a $3$-point compactification of $4$ copies of $\Z$:
$$X \,=\, \big( \{0,1,2,3\} \times \Z \big) \cup \{ a,b,c \}$$
where points of the form $(i,n)$ are discrete, 
the basic neighborhoods of $a$ have the form $\{a\} \cup \big( \{0,1\} \times (-\infty,n] \big)$, 
the basic neighborhoods of $c$ have the form $\{c\} \cup \big( \{2,3\} \times [n,\infty) \big)$, 
and the basic neighborhoods of $b$ have the form $\{b\} \cup \big(\{0,1\} \times [n,\infty) \big) \cup \big( \{2,3\} \times (-\infty,n] \big)$. 

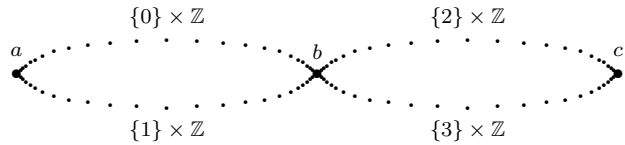
\begin{center}
\begin{tikzpicture}[yscale=.75]

\node at (0,0) {\scriptsize $\bullet$};
\node at (4,0) {\scriptsize $\bullet$};
\node at (8,0) {\scriptsize $\bullet$};

\node at (0,.4) {\scriptsize $a$};
\node at (4,.4) {\scriptsize $b$};
\node at (8,.4) {\scriptsize $c$};

\node at (2,1) {\scriptsize $\{0\} \times \Z$};
\node at (2,-1) {\scriptsize $\{1\} \times \Z$};
\node at (6,1) {\scriptsize $\{2\} \times \Z$};
\node at (6,-1) {\scriptsize $\{3\} \times \Z$};

\node at (2,.6) {\Large .};
\node at (1.6,.58) {\Large .};
\node at (2.4,.58) {\Large .};
\node at (1.25,.55) {\Large .};
\node at (2.75,.55) {\Large .};
\node at (.95,.51) {\Large .};
\node at (3.05,.51) {\Large .};
\node at (.7,.46) {\Large .};
\node at (3.3,.46) {\Large .};
\node at (.5,.4) {\Large .};
\node at (3.5,.4) {\Large .};
\node at (.35,.32) {\Large .};
\node at (3.65,.32) {\Large .};
\node at (.25,.25) {\Large .};
\node at (3.75,.25) {\Large .};
\node at (.19,.2) {\Large .};
\node at (3.81,.2) {\Large .};
\node at (.13,.14) {\Large .};
\node at (3.87,.14) {\Large .};
\node at (.09,.1) {\Large .};
\node at (3.91,.1) {\Large .};
\node at (.06,.06) {\Large .};
\node at (3.94,.06) {\Large .};
\node at (.03,.03) {\Large .};
\node at (3.97,.03) {\Large .};

\begin{scope}[yscale={-1}]
\node at (2,.6) {\Large .};
\node at (1.6,.58) {\Large .};
\node at (2.4,.58) {\Large .};
\node at (1.25,.55) {\Large .};
\node at (2.75,.55) {\Large .};
\node at (.95,.51) {\Large .};
\node at (3.05,.51) {\Large .};
\node at (.7,.46) {\Large .};
\node at (3.3,.46) {\Large .};
\node at (.5,.4) {\Large .};
\node at (3.5,.4) {\Large .};
\node at (.35,.32) {\Large .};
\node at (3.65,.32) {\Large .};
\node at (.25,.25) {\Large .};
\node at (3.75,.25) {\Large .};
\node at (.19,.2) {\Large .};
\node at (3.81,.2) {\Large .};
\node at (.13,.14) {\Large .};
\node at (3.87,.14) {\Large .};
\node at (.09,.1) {\Large .};
\node at (3.91,.1) {\Large .};
\node at (.06,.06) {\Large .};
\node at (3.94,.06) {\Large .};
\node at (.03,.03) {\Large .};
\node at (3.97,.03) {\Large .};
\end{scope}

\begin{scope}[shift={(4,0)}]
\node at (2,.6) {\Large .};
\node at (1.6,.58) {\Large .};
\node at (2.4,.58) {\Large .};
\node at (1.25,.55) {\Large .};
\node at (2.75,.55) {\Large .};
\node at (.95,.51) {\Large .};
\node at (3.05,.51) {\Large .};
\node at (.7,.46) {\Large .};
\node at (3.3,.46) {\Large .};
\node at (.5,.4) {\Large .};
\node at (3.5,.4) {\Large .};
\node at (.35,.32) {\Large .};
\node at (3.65,.32) {\Large .};
\node at (.25,.25) {\Large .};
\node at (3.75,.25) {\Large .};
\node at (.19,.2) {\Large .};
\node at (3.81,.2) {\Large .};
\node at (.13,.14) {\Large .};
\node at (3.87,.14) {\Large .};
\node at (.09,.1) {\Large .};
\node at (3.91,.1) {\Large .};
\node at (.06,.06) {\Large .};
\node at (3.94,.06) {\Large .};
\node at (.03,.03) {\Large .};
\node at (3.97,.03) {\Large .};

\begin{scope}[yscale={-1}]
\node at (2,.6) {\Large .};
\node at (1.6,.58) {\Large .};
\node at (2.4,.58) {\Large .};
\node at (1.25,.55) {\Large .};
\node at (2.75,.55) {\Large .};
\node at (.95,.51) {\Large .};
\node at (3.05,.51) {\Large .};
\node at (.7,.46) {\Large .};
\node at (3.3,.46) {\Large .};
\node at (.5,.4) {\Large .};
\node at (3.5,.4) {\Large .};
\node at (.35,.32) {\Large .};
\node at (3.65,.32) {\Large .};
\node at (.25,.25) {\Large .};
\node at (3.75,.25) {\Large .};
\node at (.19,.2) {\Large .};
\node at (3.81,.2) {\Large .};
\node at (.13,.14) {\Large .};
\node at (3.87,.14) {\Large .};
\node at (.09,.1) {\Large .};
\node at (3.91,.1) {\Large .};
\node at (.06,.06) {\Large .};
\node at (3.94,.06) {\Large .};
\node at (.03,.03) {\Large .};
\node at (3.97,.03) {\Large .};
\end{scope}
\end{scope}

\end{tikzpicture}
\end{center}

\noindent Let $\AA$ denote the Boolean algebra of clopen subsets of $X$. Define a homeomorphism $f: X \to X$ by setting
$$f(a) = a, \quad f(b) = b, \quad f(c) = c,$$
$$f(i,n) = (i,n+1) \text{ for $i = 0,2$},$$
$$f(i,n) = (i,n-1) \text{ for $i = 1,3$}.$$

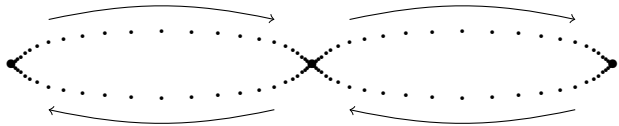
\begin{center}
\begin{tikzpicture}[yscale=.75]

\node at (0,0) {\scriptsize $\bullet$};
\node at (4,0) {\scriptsize $\bullet$};
\node at (8,0) {\scriptsize $\bullet$};

\node at (2,.6) {\Large .};
\node at (1.6,.58) {\Large .};
\node at (2.4,.58) {\Large .};
\node at (1.25,.55) {\Large .};
\node at (2.75,.55) {\Large .};
\node at (.95,.51) {\Large .};
\node at (3.05,.51) {\Large .};
\node at (.7,.46) {\Large .};
\node at (3.3,.46) {\Large .};
\node at (.5,.4) {\Large .};
\node at (3.5,.4) {\Large .};
\node at (.35,.32) {\Large .};
\node at (3.65,.32) {\Large .};
\node at (.25,.25) {\Large .};
\node at (3.75,.25) {\Large .};
\node at (.19,.2) {\Large .};
\node at (3.81,.2) {\Large .};
\node at (.13,.14) {\Large .};
\node at (3.87,.14) {\Large .};
\node at (.09,.1) {\Large .};
\node at (3.91,.1) {\Large .};
\node at (.06,.06) {\Large .};
\node at (3.94,.06) {\Large .};
\node at (.03,.03) {\Large .};
\node at (3.97,.03) {\Large .};
\draw [<-] (3.5,.8) to [out=164,in=16] (.5,.8);

\begin{scope}[yscale={-1}]
\node at (2,.6) {\Large .};
\node at (1.6,.58) {\Large .};
\node at (2.4,.58) {\Large .};
\node at (1.25,.55) {\Large .};
\node at (2.75,.55) {\Large .};
\node at (.95,.51) {\Large .};
\node at (3.05,.51) {\Large .};
\node at (.7,.46) {\Large .};
\node at (3.3,.46) {\Large .};
\node at (.5,.4) {\Large .};
\node at (3.5,.4) {\Large .};
\node at (.35,.32) {\Large .};
\node at (3.65,.32) {\Large .};
\node at (.25,.25) {\Large .};
\node at (3.75,.25) {\Large .};
\node at (.19,.2) {\Large .};
\node at (3.81,.2) {\Large .};
\node at (.13,.14) {\Large .};
\node at (3.87,.14) {\Large .};
\node at (.09,.1) {\Large .};
\node at (3.91,.1) {\Large .};
\node at (.06,.06) {\Large .};
\node at (3.94,.06) {\Large .};
\node at (.03,.03) {\Large .};
\node at (3.97,.03) {\Large .};
\draw [->] (3.5,.8) to [out=164,in=16] (.5,.8);
\end{scope}

\begin{scope}[shift={(4,0)}]
\node at (2,.6) {\Large .};
\node at (1.6,.58) {\Large .};
\node at (2.4,.58) {\Large .};
\node at (1.25,.55) {\Large .};
\node at (2.75,.55) {\Large .};
\node at (.95,.51) {\Large .};
\node at (3.05,.51) {\Large .};
\node at (.7,.46) {\Large .};
\node at (3.3,.46) {\Large .};
\node at (.5,.4) {\Large .};
\node at (3.5,.4) {\Large .};
\node at (.35,.32) {\Large .};
\node at (3.65,.32) {\Large .};
\node at (.25,.25) {\Large .};
\node at (3.75,.25) {\Large .};
\node at (.19,.2) {\Large .};
\node at (3.81,.2) {\Large .};
\node at (.13,.14) {\Large .};
\node at (3.87,.14) {\Large .};
\node at (.09,.1) {\Large .};
\node at (3.91,.1) {\Large .};
\node at (.06,.06) {\Large .};
\node at (3.94,.06) {\Large .};
\node at (.03,.03) {\Large .};
\node at (3.97,.03) {\Large .};
\draw [<-] (3.5,.8) to [out=164,in=16] (.5,.8);

\begin{scope}[yscale={-1}]
\node at (2,.6) {\Large .};
\node at (1.6,.58) {\Large .};
\node at (2.4,.58) {\Large .};
\node at (1.25,.55) {\Large .};
\node at (2.75,.55) {\Large .};
\node at (.95,.51) {\Large .};
\node at (3.05,.51) {\Large .};
\node at (.7,.46) {\Large .};
\node at (3.3,.46) {\Large .};
\node at (.5,.4) {\Large .};
\node at (3.5,.4) {\Large .};
\node at (.35,.32) {\Large .};
\node at (3.65,.32) {\Large .};
\node at (.25,.25) {\Large .};
\node at (3.75,.25) {\Large .};
\node at (.19,.2) {\Large .};
\node at (3.81,.2) {\Large .};
\node at (.13,.14) {\Large .};
\node at (3.87,.14) {\Large .};
\node at (.09,.1) {\Large .};
\node at (3.91,.1) {\Large .};
\node at (.06,.06) {\Large .};
\node at (3.94,.06) {\Large .};
\node at (.03,.03) {\Large .};
\node at (3.97,.03) {\Large .};
\draw [->] (3.5,.8) to [out=164,in=16] (.5,.8);
\end{scope}
\end{scope}

\end{tikzpicture}
\end{center}

\noindent Define $\a: \AA \to \AA$ by setting $\a(A) = f[A]$ (where $f[A]$ denotes the image of $A$ under the map $f$). This is well-defined, because homeomorphisms map clopen sets to clopen sets, and it is an automorphism of $\AA$. So $\< \AA,\a \>$ is an algebraic dynamical system.

The topological space $Y$ (whose Stone dual gives us $\BB$) is a $4$-point compactification of $6$ copies of $\Z$:
$$Y \,=\, \big( \{0,1,2,3,4,5\} \times \Z \big) \cup \{ a,b,c,d \}$$
where points of the form $(i,n)$ are discrete, 
the basic neighborhoods of $a$ have the form $\{a\} \cup \big( \{0,1\} \times (-\infty,n] \big)$, 
the basic neighborhoods of $d$ have the form $\{d\} \cup \big( \{4,5\} \times [n,\infty) \big)$, 
the basic neighborhoods of $b$ have the form $\{b\} \cup \big( \{0,1\} \times [n,\infty) \big) \cup \big( \{2,3\} \times (-\infty,n] \big)$, 
and the basic neighborhoods of $c$ have the form $\{c\} \cup \big( \{2,3\} \times [n,\infty) \big) \cup \big( \{4,5\} \times (-\infty,n] \big)$. 

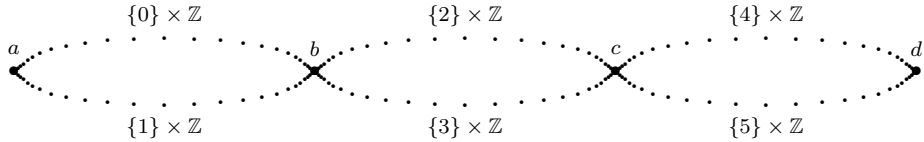
\begin{center}
\begin{tikzpicture}[yscale=.75]

\node at (0,0) {\scriptsize $\bullet$};
\node at (4,0) {\scriptsize $\bullet$};
\node at (8,0) {\scriptsize $\bullet$};
\node at (12,0) {\scriptsize $\bullet$};

\node at (0,.4) {\scriptsize $a$};
\node at (4,.4) {\scriptsize $b$};
\node at (8,.4) {\scriptsize $c$};
\node at (12,.4) {\scriptsize $d$};

\node at (2,1) {\scriptsize $\{0\} \times \Z$};
\node at (2,-1) {\scriptsize $\{1\} \times \Z$};
\node at (6,1) {\scriptsize $\{2\} \times \Z$};
\node at (6,-1) {\scriptsize $\{3\} \times \Z$};
\node at (10,1) {\scriptsize $\{4\} \times \Z$};
\node at (10,-1) {\scriptsize $\{5\} \times \Z$};

\node at (2,.6) {\Large .};
\node at (1.6,.58) {\Large .};
\node at (2.4,.58) {\Large .};
\node at (1.25,.55) {\Large .};
\node at (2.75,.55) {\Large .};
\node at (.95,.51) {\Large .};
\node at (3.05,.51) {\Large .};
\node at (.7,.46) {\Large .};
\node at (3.3,.46) {\Large .};
\node at (.5,.4) {\Large .};
\node at (3.5,.4) {\Large .};
\node at (.35,.32) {\Large .};
\node at (3.65,.32) {\Large .};
\node at (.25,.25) {\Large .};
\node at (3.75,.25) {\Large .};
\node at (.19,.2) {\Large .};
\node at (3.81,.2) {\Large .};
\node at (.13,.14) {\Large .};
\node at (3.87,.14) {\Large .};
\node at (.09,.1) {\Large .};
\node at (3.91,.1) {\Large .};
\node at (.06,.06) {\Large .};
\node at (3.94,.06) {\Large .};
\node at (.03,.03) {\Large .};
\node at (3.97,.03) {\Large .};

\begin{scope}[yscale={-1}]
\node at (2,.6) {\Large .};
\node at (1.6,.58) {\Large .};
\node at (2.4,.58) {\Large .};
\node at (1.25,.55) {\Large .};
\node at (2.75,.55) {\Large .};
\node at (.95,.51) {\Large .};
\node at (3.05,.51) {\Large .};
\node at (.7,.46) {\Large .};
\node at (3.3,.46) {\Large .};
\node at (.5,.4) {\Large .};
\node at (3.5,.4) {\Large .};
\node at (.35,.32) {\Large .};
\node at (3.65,.32) {\Large .};
\node at (.25,.25) {\Large .};
\node at (3.75,.25) {\Large .};
\node at (.19,.2) {\Large .};
\node at (3.81,.2) {\Large .};
\node at (.13,.14) {\Large .};
\node at (3.87,.14) {\Large .};
\node at (.09,.1) {\Large .};
\node at (3.91,.1) {\Large .};
\node at (.06,.06) {\Large .};
\node at (3.94,.06) {\Large .};
\node at (.03,.03) {\Large .};
\node at (3.97,.03) {\Large .};
\end{scope}

\begin{scope}[shift={(4,0)}]
\node at (2,.6) {\Large .};
\node at (1.6,.58) {\Large .};
\node at (2.4,.58) {\Large .};
\node at (1.25,.55) {\Large .};
\node at (2.75,.55) {\Large .};
\node at (.95,.51) {\Large .};
\node at (3.05,.51) {\Large .};
\node at (.7,.46) {\Large .};
\node at (3.3,.46) {\Large .};
\node at (.5,.4) {\Large .};
\node at (3.5,.4) {\Large .};
\node at (.35,.32) {\Large .};
\node at (3.65,.32) {\Large .};
\node at (.25,.25) {\Large .};
\node at (3.75,.25) {\Large .};
\node at (.19,.2) {\Large .};
\node at (3.81,.2) {\Large .};
\node at (.13,.14) {\Large .};
\node at (3.87,.14) {\Large .};
\node at (.09,.1) {\Large .};
\node at (3.91,.1) {\Large .};
\node at (.06,.06) {\Large .};
\node at (3.94,.06) {\Large .};
\node at (.03,.03) {\Large .};
\node at (3.97,.03) {\Large .};

\begin{scope}[yscale={-1}]
\node at (2,.6) {\Large .};
\node at (1.6,.58) {\Large .};
\node at (2.4,.58) {\Large .};
\node at (1.25,.55) {\Large .};
\node at (2.75,.55) {\Large .};
\node at (.95,.51) {\Large .};
\node at (3.05,.51) {\Large .};
\node at (.7,.46) {\Large .};
\node at (3.3,.46) {\Large .};
\node at (.5,.4) {\Large .};
\node at (3.5,.4) {\Large .};
\node at (.35,.32) {\Large .};
\node at (3.65,.32) {\Large .};
\node at (.25,.25) {\Large .};
\node at (3.75,.25) {\Large .};
\node at (.19,.2) {\Large .};
\node at (3.81,.2) {\Large .};
\node at (.13,.14) {\Large .};
\node at (3.87,.14) {\Large .};
\node at (.09,.1) {\Large .};
\node at (3.91,.1) {\Large .};
\node at (.06,.06) {\Large .};
\node at (3.94,.06) {\Large .};
\node at (.03,.03) {\Large .};
\node at (3.97,.03) {\Large .};
\end{scope}
\end{scope}

\begin{scope}[shift={(8,0)}]
\node at (2,.6) {\Large .};
\node at (1.6,.58) {\Large .};
\node at (2.4,.58) {\Large .};
\node at (1.25,.55) {\Large .};
\node at (2.75,.55) {\Large .};
\node at (.95,.51) {\Large .};
\node at (3.05,.51) {\Large .};
\node at (.7,.46) {\Large .};
\node at (3.3,.46) {\Large .};
\node at (.5,.4) {\Large .};
\node at (3.5,.4) {\Large .};
\node at (.35,.32) {\Large .};
\node at (3.65,.32) {\Large .};
\node at (.25,.25) {\Large .};
\node at (3.75,.25) {\Large .};
\node at (.19,.2) {\Large .};
\node at (3.81,.2) {\Large .};
\node at (.13,.14) {\Large .};
\node at (3.87,.14) {\Large .};
\node at (.09,.1) {\Large .};
\node at (3.91,.1) {\Large .};
\node at (.06,.06) {\Large .};
\node at (3.94,.06) {\Large .};
\node at (.03,.03) {\Large .};
\node at (3.97,.03) {\Large .};

\begin{scope}[yscale={-1}]
\node at (2,.6) {\Large .};
\node at (1.6,.58) {\Large .};
\node at (2.4,.58) {\Large .};
\node at (1.25,.55) {\Large .};
\node at (2.75,.55) {\Large .};
\node at (.95,.51) {\Large .};
\node at (3.05,.51) {\Large .};
\node at (.7,.46) {\Large .};
\node at (3.3,.46) {\Large .};
\node at (.5,.4) {\Large .};
\node at (3.5,.4) {\Large .};
\node at (.35,.32) {\Large .};
\node at (3.65,.32) {\Large .};
\node at (.25,.25) {\Large .};
\node at (3.75,.25) {\Large .};
\node at (.19,.2) {\Large .};
\node at (3.81,.2) {\Large .};
\node at (.13,.14) {\Large .};
\node at (3.87,.14) {\Large .};
\node at (.09,.1) {\Large .};
\node at (3.91,.1) {\Large .};
\node at (.06,.06) {\Large .};
\node at (3.94,.06) {\Large .};
\node at (.03,.03) {\Large .};
\node at (3.97,.03) {\Large .};
\end{scope}
\end{scope}

\end{tikzpicture}
\end{center}

\noindent Let $\BB$ denote the Boolean algebra of clopen subsets of $Y$. Define a homeomorphism $g: Y \to Y$ by setting
$$g(a) = a, \quad g(b) = b, \quad g(c) = c, \quad g(d) = d,$$
$$g(i,n) = (i,n+1) \text{ for $i = 0,2,4$},$$
$$g(i,n) = (i,n-1) \text{ for $i = 1,3,5$}.$$

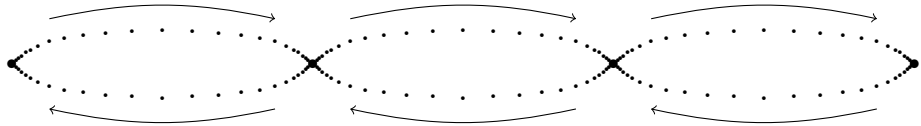
\begin{center}
\begin{tikzpicture}[yscale=.75]

\node at (0,0) {\scriptsize $\bullet$};
\node at (4,0) {\scriptsize $\bullet$};
\node at (8,0) {\scriptsize $\bullet$};
\node at (12,0) {\scriptsize $\bullet$};

\node at (2,.6) {\Large .};
\node at (1.6,.58) {\Large .};
\node at (2.4,.58) {\Large .};
\node at (1.25,.55) {\Large .};
\node at (2.75,.55) {\Large .};
\node at (.95,.51) {\Large .};
\node at (3.05,.51) {\Large .};
\node at (.7,.46) {\Large .};
\node at (3.3,.46) {\Large .};
\node at (.5,.4) {\Large .};
\node at (3.5,.4) {\Large .};
\node at (.35,.32) {\Large .};
\node at (3.65,.32) {\Large .};
\node at (.25,.25) {\Large .};
\node at (3.75,.25) {\Large .};
\node at (.19,.2) {\Large .};
\node at (3.81,.2) {\Large .};
\node at (.13,.14) {\Large .};
\node at (3.87,.14) {\Large .};
\node at (.09,.1) {\Large .};
\node at (3.91,.1) {\Large .};
\node at (.06,.06) {\Large .};
\node at (3.94,.06) {\Large .};
\node at (.03,.03) {\Large .};
\node at (3.97,.03) {\Large .};
\draw [<-] (3.5,.8) to [out=164,in=16] (.5,.8);

\begin{scope}[yscale={-1}]
\node at (2,.6) {\Large .};
\node at (1.6,.58) {\Large .};
\node at (2.4,.58) {\Large .};
\node at (1.25,.55) {\Large .};
\node at (2.75,.55) {\Large .};
\node at (.95,.51) {\Large .};
\node at (3.05,.51) {\Large .};
\node at (.7,.46) {\Large .};
\node at (3.3,.46) {\Large .};
\node at (.5,.4) {\Large .};
\node at (3.5,.4) {\Large .};
\node at (.35,.32) {\Large .};
\node at (3.65,.32) {\Large .};
\node at (.25,.25) {\Large .};
\node at (3.75,.25) {\Large .};
\node at (.19,.2) {\Large .};
\node at (3.81,.2) {\Large .};
\node at (.13,.14) {\Large .};
\node at (3.87,.14) {\Large .};
\node at (.09,.1) {\Large .};
\node at (3.91,.1) {\Large .};
\node at (.06,.06) {\Large .};
\node at (3.94,.06) {\Large .};
\node at (.03,.03) {\Large .};
\node at (3.97,.03) {\Large .};
\draw [->] (3.5,.8) to [out=164,in=16] (.5,.8);
\end{scope}

\begin{scope}[shift={(4,0)}]
\node at (2,.6) {\Large .};
\node at (1.6,.58) {\Large .};
\node at (2.4,.58) {\Large .};
\node at (1.25,.55) {\Large .};
\node at (2.75,.55) {\Large .};
\node at (.95,.51) {\Large .};
\node at (3.05,.51) {\Large .};
\node at (.7,.46) {\Large .};
\node at (3.3,.46) {\Large .};
\node at (.5,.4) {\Large .};
\node at (3.5,.4) {\Large .};
\node at (.35,.32) {\Large .};
\node at (3.65,.32) {\Large .};
\node at (.25,.25) {\Large .};
\node at (3.75,.25) {\Large .};
\node at (.19,.2) {\Large .};
\node at (3.81,.2) {\Large .};
\node at (.13,.14) {\Large .};
\node at (3.87,.14) {\Large .};
\node at (.09,.1) {\Large .};
\node at (3.91,.1) {\Large .};
\node at (.06,.06) {\Large .};
\node at (3.94,.06) {\Large .};
\node at (.03,.03) {\Large .};
\node at (3.97,.03) {\Large .};
\draw [<-] (3.5,.8) to [out=164,in=16] (.5,.8);

\begin{scope}[yscale={-1}]
\node at (2,.6) {\Large .};
\node at (1.6,.58) {\Large .};
\node at (2.4,.58) {\Large .};
\node at (1.25,.55) {\Large .};
\node at (2.75,.55) {\Large .};
\node at (.95,.51) {\Large .};
\node at (3.05,.51) {\Large .};
\node at (.7,.46) {\Large .};
\node at (3.3,.46) {\Large .};
\node at (.5,.4) {\Large .};
\node at (3.5,.4) {\Large .};
\node at (.35,.32) {\Large .};
\node at (3.65,.32) {\Large .};
\node at (.25,.25) {\Large .};
\node at (3.75,.25) {\Large .};
\node at (.19,.2) {\Large .};
\node at (3.81,.2) {\Large .};
\node at (.13,.14) {\Large .};
\node at (3.87,.14) {\Large .};
\node at (.09,.1) {\Large .};
\node at (3.91,.1) {\Large .};
\node at (.06,.06) {\Large .};
\node at (3.94,.06) {\Large .};
\node at (.03,.03) {\Large .};
\node at (3.97,.03) {\Large .};
\draw [->] (3.5,.8) to [out=164,in=16] (.5,.8);
\end{scope}
\end{scope}

\begin{scope}[shift={(8,0)}]
\node at (2,.6) {\Large .};
\node at (1.6,.58) {\Large .};
\node at (2.4,.58) {\Large .};
\node at (1.25,.55) {\Large .};
\node at (2.75,.55) {\Large .};
\node at (.95,.51) {\Large .};
\node at (3.05,.51) {\Large .};
\node at (.7,.46) {\Large .};
\node at (3.3,.46) {\Large .};
\node at (.5,.4) {\Large .};
\node at (3.5,.4) {\Large .};
\node at (.35,.32) {\Large .};
\node at (3.65,.32) {\Large .};
\node at (.25,.25) {\Large .};
\node at (3.75,.25) {\Large .};
\node at (.19,.2) {\Large .};
\node at (3.81,.2) {\Large .};
\node at (.13,.14) {\Large .};
\node at (3.87,.14) {\Large .};
\node at (.09,.1) {\Large .};
\node at (3.91,.1) {\Large .};
\node at (.06,.06) {\Large .};
\node at (3.94,.06) {\Large .};
\node at (.03,.03) {\Large .};
\node at (3.97,.03) {\Large .};
\draw [<-] (3.5,.8) to [out=164,in=16] (.5,.8);

\begin{scope}[yscale={-1}]
\node at (2,.6) {\Large .};
\node at (1.6,.58) {\Large .};
\node at (2.4,.58) {\Large .};
\node at (1.25,.55) {\Large .};
\node at (2.75,.55) {\Large .};
\node at (.95,.51) {\Large .};
\node at (3.05,.51) {\Large .};
\node at (.7,.46) {\Large .};
\node at (3.3,.46) {\Large .};
\node at (.5,.4) {\Large .};
\node at (3.5,.4) {\Large .};
\node at (.35,.32) {\Large .};
\node at (3.65,.32) {\Large .};
\node at (.25,.25) {\Large .};
\node at (3.75,.25) {\Large .};
\node at (.19,.2) {\Large .};
\node at (3.81,.2) {\Large .};
\node at (.13,.14) {\Large .};
\node at (3.87,.14) {\Large .};
\node at (.09,.1) {\Large .};
\node at (3.91,.1) {\Large .};
\node at (.06,.06) {\Large .};
\node at (3.94,.06) {\Large .};
\node at (.03,.03) {\Large .};
\node at (3.97,.03) {\Large .};
\draw [->] (3.5,.8) to [out=164,in=16] (.5,.8);
\end{scope}
\end{scope}

\end{tikzpicture}
\end{center}

\noindent Define $\b: \BB \to \BB$ by setting $\b(B) = g[B]$. This is an automorphism of $\BB$, so $\< \BB,\b \>$ is an algebraic dynamical system.

Note that $X$ and $Y$ have been defined so that $X \sub Y$, and $f = g \rest X$. Define a projection map $\pi: Y \to X$ by setting $\pi(x) = x$ for all $x \in X$, and otherwise set
$$\pi(d) = b, \qquad \pi(4,n) = (3,-n), \qquad \pi(5,n) = (2,-n).$$
Roughly speaking, $\pi$ takes the rightmost third of $Y$ from our picture, and rotates it around the point $c$ precisely one half turn, until it sits on top of the middle third. We then define $\iota(A) = \pi^{-1}(A)$ for all $A \in \AA$. This is a Boolean-algebraic embedding (the Stone dual of $\pi$). On the topological side, we have $\pi \circ g = f \circ \pi$, and it follows that $\iota \circ \a = \b \circ \iota$ on the algebraic side. Hence $\iota$ is an embedding of $\< \AA,\a \>$ into $\< \BB,\b \>$.

To finish our description of the counterexample, we must define an embedding $\eta$ from $\< \AA,\a \>$ into $\< \pwmff,\s \>$. 
We do this by describing an eventually small, eventually dense sequence $\seq{a_n}{n \in \w}$ of members of $\AA$ such that $\seq{\a(a_n)}{n \in \w} \approx \seq{a_{n+1}}{n \in \w}$. 

The sequence is chosen so that each $a_n$ is a clopen singleton in $X$. The pattern of the sequence is depicted below. Roughly, the idea is to follow the action of $\a$ for increasingly long stretches along each of the $4$ copies of $\Z$ in $X$, jumping to the next stretch at some prescribed time. 

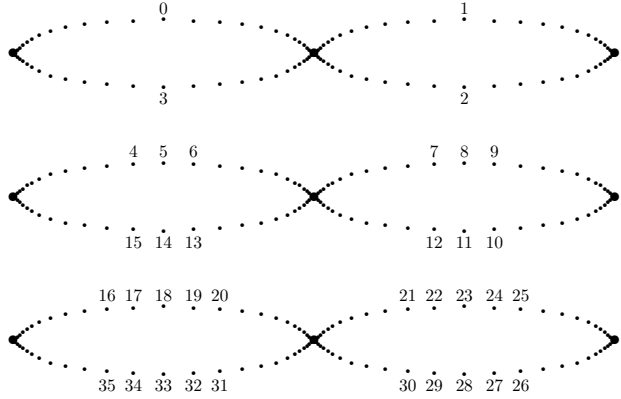
\begin{center}
\begin{tikzpicture}[yscale=.75]

\node at (0,0) {\scriptsize $\bullet$};
\node at (4,0) {\scriptsize $\bullet$};
\node at (8,0) {\scriptsize $\bullet$};

\node at (2,.8) {\scalebox{.6}{$0$}};
\node at (6,.8) {\scalebox{.6}{$1$}};
\node at (6,-.8) {\scalebox{.6}{$2$}};
\node at (2,-.8) {\scalebox{.6}{$3$}};

\node at (2,.6) {\Large .};
\node at (1.6,.58) {\Large .};
\node at (2.4,.58) {\Large .};
\node at (1.25,.55) {\Large .};
\node at (2.75,.55) {\Large .};
\node at (.95,.51) {\Large .};
\node at (3.05,.51) {\Large .};
\node at (.7,.46) {\Large .};
\node at (3.3,.46) {\Large .};
\node at (.5,.4) {\Large .};
\node at (3.5,.4) {\Large .};
\node at (.35,.32) {\Large .};
\node at (3.65,.32) {\Large .};
\node at (.25,.25) {\Large .};
\node at (3.75,.25) {\Large .};
\node at (.19,.2) {\Large .};
\node at (3.81,.2) {\Large .};
\node at (.13,.14) {\Large .};
\node at (3.87,.14) {\Large .};
\node at (.09,.1) {\Large .};
\node at (3.91,.1) {\Large .};
\node at (.06,.06) {\Large .};
\node at (3.94,.06) {\Large .};
\node at (.03,.03) {\Large .};
\node at (3.97,.03) {\Large .};

\begin{scope}[yscale={-1}]
\node at (2,.6) {\Large .};
\node at (1.6,.58) {\Large .};
\node at (2.4,.58) {\Large .};
\node at (1.25,.55) {\Large .};
\node at (2.75,.55) {\Large .};
\node at (.95,.51) {\Large .};
\node at (3.05,.51) {\Large .};
\node at (.7,.46) {\Large .};
\node at (3.3,.46) {\Large .};
\node at (.5,.4) {\Large .};
\node at (3.5,.4) {\Large .};
\node at (.35,.32) {\Large .};
\node at (3.65,.32) {\Large .};
\node at (.25,.25) {\Large .};
\node at (3.75,.25) {\Large .};
\node at (.19,.2) {\Large .};
\node at (3.81,.2) {\Large .};
\node at (.13,.14) {\Large .};
\node at (3.87,.14) {\Large .};
\node at (.09,.1) {\Large .};
\node at (3.91,.1) {\Large .};
\node at (.06,.06) {\Large .};
\node at (3.94,.06) {\Large .};
\node at (.03,.03) {\Large .};
\node at (3.97,.03) {\Large .};
\end{scope}

\begin{scope}[shift={(4,0)}]
\node at (2,.6) {\Large .};
\node at (1.6,.58) {\Large .};
\node at (2.4,.58) {\Large .};
\node at (1.25,.55) {\Large .};
\node at (2.75,.55) {\Large .};
\node at (.95,.51) {\Large .};
\node at (3.05,.51) {\Large .};
\node at (.7,.46) {\Large .};
\node at (3.3,.46) {\Large .};
\node at (.5,.4) {\Large .};
\node at (3.5,.4) {\Large .};
\node at (.35,.32) {\Large .};
\node at (3.65,.32) {\Large .};
\node at (.25,.25) {\Large .};
\node at (3.75,.25) {\Large .};
\node at (.19,.2) {\Large .};
\node at (3.81,.2) {\Large .};
\node at (.13,.14) {\Large .};
\node at (3.87,.14) {\Large .};
\node at (.09,.1) {\Large .};
\node at (3.91,.1) {\Large .};
\node at (.06,.06) {\Large .};
\node at (3.94,.06) {\Large .};
\node at (.03,.03) {\Large .};
\node at (3.97,.03) {\Large .};

\begin{scope}[yscale={-1}]
\node at (2,.6) {\Large .};
\node at (1.6,.58) {\Large .};
\node at (2.4,.58) {\Large .};
\node at (1.25,.55) {\Large .};
\node at (2.75,.55) {\Large .};
\node at (.95,.51) {\Large .};
\node at (3.05,.51) {\Large .};
\node at (.7,.46) {\Large .};
\node at (3.3,.46) {\Large .};
\node at (.5,.4) {\Large .};
\node at (3.5,.4) {\Large .};
\node at (.35,.32) {\Large .};
\node at (3.65,.32) {\Large .};
\node at (.25,.25) {\Large .};
\node at (3.75,.25) {\Large .};
\node at (.19,.2) {\Large .};
\node at (3.81,.2) {\Large .};
\node at (.13,.14) {\Large .};
\node at (3.87,.14) {\Large .};
\node at (.09,.1) {\Large .};
\node at (3.91,.1) {\Large .};
\node at (.06,.06) {\Large .};
\node at (3.94,.06) {\Large .};
\node at (.03,.03) {\Large .};
\node at (3.97,.03) {\Large .};
\end{scope}
\end{scope}

\end{tikzpicture}
\end{center}

\begin{center}
\begin{tikzpicture}[yscale=.75]

\node at (0,0) {\scriptsize $\bullet$};
\node at (4,0) {\scriptsize $\bullet$};
\node at (8,0) {\scriptsize $\bullet$};

\node at (1.6,.8) {\scalebox{.6}{$4$}};
\node at (2,.8) {\scalebox{.6}{$5$}};
\node at (2.4,.8) {\scalebox{.6}{$6$}};
\node at (5.6,.8) {\scalebox{.6}{$7$}};
\node at (6,.8) {\scalebox{.6}{$8$}};
\node at (6.4,.8) {\scalebox{.6}{$9$}};
\node at (6.4,-.8) {\scalebox{.6}{$10$}};
\node at (6,-.8) {\scalebox{.6}{$11$}};
\node at (5.6,-.8) {\scalebox{.6}{$12$}};
\node at (2.4,-.8) {\scalebox{.6}{$13$}};
\node at (2,-.8) {\scalebox{.6}{$14$}};
\node at (1.6,-.8) {\scalebox{.6}{$15$}};

\node at (2,.6) {\Large .};
\node at (1.6,.58) {\Large .};
\node at (2.4,.58) {\Large .};
\node at (1.25,.55) {\Large .};
\node at (2.75,.55) {\Large .};
\node at (.95,.51) {\Large .};
\node at (3.05,.51) {\Large .};
\node at (.7,.46) {\Large .};
\node at (3.3,.46) {\Large .};
\node at (.5,.4) {\Large .};
\node at (3.5,.4) {\Large .};
\node at (.35,.32) {\Large .};
\node at (3.65,.32) {\Large .};
\node at (.25,.25) {\Large .};
\node at (3.75,.25) {\Large .};
\node at (.19,.2) {\Large .};
\node at (3.81,.2) {\Large .};
\node at (.13,.14) {\Large .};
\node at (3.87,.14) {\Large .};
\node at (.09,.1) {\Large .};
\node at (3.91,.1) {\Large .};
\node at (.06,.06) {\Large .};
\node at (3.94,.06) {\Large .};
\node at (.03,.03) {\Large .};
\node at (3.97,.03) {\Large .};

\begin{scope}[yscale={-1}]
\node at (2,.6) {\Large .};
\node at (1.6,.58) {\Large .};
\node at (2.4,.58) {\Large .};
\node at (1.25,.55) {\Large .};
\node at (2.75,.55) {\Large .};
\node at (.95,.51) {\Large .};
\node at (3.05,.51) {\Large .};
\node at (.7,.46) {\Large .};
\node at (3.3,.46) {\Large .};
\node at (.5,.4) {\Large .};
\node at (3.5,.4) {\Large .};
\node at (.35,.32) {\Large .};
\node at (3.65,.32) {\Large .};
\node at (.25,.25) {\Large .};
\node at (3.75,.25) {\Large .};
\node at (.19,.2) {\Large .};
\node at (3.81,.2) {\Large .};
\node at (.13,.14) {\Large .};
\node at (3.87,.14) {\Large .};
\node at (.09,.1) {\Large .};
\node at (3.91,.1) {\Large .};
\node at (.06,.06) {\Large .};
\node at (3.94,.06) {\Large .};
\node at (.03,.03) {\Large .};
\node at (3.97,.03) {\Large .};
\end{scope}

\begin{scope}[shift={(4,0)}]
\node at (2,.6) {\Large .};
\node at (1.6,.58) {\Large .};
\node at (2.4,.58) {\Large .};
\node at (1.25,.55) {\Large .};
\node at (2.75,.55) {\Large .};
\node at (.95,.51) {\Large .};
\node at (3.05,.51) {\Large .};
\node at (.7,.46) {\Large .};
\node at (3.3,.46) {\Large .};
\node at (.5,.4) {\Large .};
\node at (3.5,.4) {\Large .};
\node at (.35,.32) {\Large .};
\node at (3.65,.32) {\Large .};
\node at (.25,.25) {\Large .};
\node at (3.75,.25) {\Large .};
\node at (.19,.2) {\Large .};
\node at (3.81,.2) {\Large .};
\node at (.13,.14) {\Large .};
\node at (3.87,.14) {\Large .};
\node at (.09,.1) {\Large .};
\node at (3.91,.1) {\Large .};
\node at (.06,.06) {\Large .};
\node at (3.94,.06) {\Large .};
\node at (.03,.03) {\Large .};
\node at (3.97,.03) {\Large .};

\begin{scope}[yscale={-1}]
\node at (2,.6) {\Large .};
\node at (1.6,.58) {\Large .};
\node at (2.4,.58) {\Large .};
\node at (1.25,.55) {\Large .};
\node at (2.75,.55) {\Large .};
\node at (.95,.51) {\Large .};
\node at (3.05,.51) {\Large .};
\node at (.7,.46) {\Large .};
\node at (3.3,.46) {\Large .};
\node at (.5,.4) {\Large .};
\node at (3.5,.4) {\Large .};
\node at (.35,.32) {\Large .};
\node at (3.65,.32) {\Large .};
\node at (.25,.25) {\Large .};
\node at (3.75,.25) {\Large .};
\node at (.19,.2) {\Large .};
\node at (3.81,.2) {\Large .};
\node at (.13,.14) {\Large .};
\node at (3.87,.14) {\Large .};
\node at (.09,.1) {\Large .};
\node at (3.91,.1) {\Large .};
\node at (.06,.06) {\Large .};
\node at (3.94,.06) {\Large .};
\node at (.03,.03) {\Large .};
\node at (3.97,.03) {\Large .};
\end{scope}
\end{scope}

\end{tikzpicture}
\end{center}

\begin{center}
\begin{tikzpicture}[yscale=.75]

\node at (0,0) {\scriptsize $\bullet$};
\node at (4,0) {\scriptsize $\bullet$};
\node at (8,0) {\scriptsize $\bullet$};

\node at (1.25,.8) {\scalebox{.6}{$16$}};
\node at (1.6,.8) {\scalebox{.6}{$17$}};
\node at (2,.8) {\scalebox{.6}{$18$}};
\node at (2.4,.8) {\scalebox{.6}{$19$}};
\node at (2.75,.8) {\scalebox{.6}{$20$}};
\node at (5.25,.8) {\scalebox{.6}{$21$}};
\node at (5.6,.8) {\scalebox{.6}{$22$}};
\node at (6,.8) {\scalebox{.6}{$23$}};
\node at (6.4,.8) {\scalebox{.6}{$24$}};
\node at (6.75,.8) {\scalebox{.6}{$25$}};
\node at (6.75,-.8) {\scalebox{.6}{$26$}};
\node at (6.4,-.8) {\scalebox{.6}{$27$}};
\node at (6,-.8) {\scalebox{.6}{$28$}};
\node at (5.6,-.8) {\scalebox{.6}{$29$}};
\node at (5.25,-.8) {\scalebox{.6}{$30$}};
\node at (2.75,-.8) {\scalebox{.6}{$31$}};
\node at (2.4,-.8) {\scalebox{.6}{$32$}};
\node at (2,-.8) {\scalebox{.6}{$33$}};
\node at (1.6,-.8) {\scalebox{.6}{$34$}};
\node at (1.25,-.8) {\scalebox{.6}{$35$}};

\node at (2,.6) {\Large .};
\node at (1.6,.58) {\Large .};
\node at (2.4,.58) {\Large .};
\node at (1.25,.55) {\Large .};
\node at (2.75,.55) {\Large .};
\node at (.95,.51) {\Large .};
\node at (3.05,.51) {\Large .};
\node at (.7,.46) {\Large .};
\node at (3.3,.46) {\Large .};
\node at (.5,.4) {\Large .};
\node at (3.5,.4) {\Large .};
\node at (.35,.32) {\Large .};
\node at (3.65,.32) {\Large .};
\node at (.25,.25) {\Large .};
\node at (3.75,.25) {\Large .};
\node at (.19,.2) {\Large .};
\node at (3.81,.2) {\Large .};
\node at (.13,.14) {\Large .};
\node at (3.87,.14) {\Large .};
\node at (.09,.1) {\Large .};
\node at (3.91,.1) {\Large .};
\node at (.06,.06) {\Large .};
\node at (3.94,.06) {\Large .};
\node at (.03,.03) {\Large .};
\node at (3.97,.03) {\Large .};

\begin{scope}[yscale={-1}]
\node at (2,.6) {\Large .};
\node at (1.6,.58) {\Large .};
\node at (2.4,.58) {\Large .};
\node at (1.25,.55) {\Large .};
\node at (2.75,.55) {\Large .};
\node at (.95,.51) {\Large .};
\node at (3.05,.51) {\Large .};
\node at (.7,.46) {\Large .};
\node at (3.3,.46) {\Large .};
\node at (.5,.4) {\Large .};
\node at (3.5,.4) {\Large .};
\node at (.35,.32) {\Large .};
\node at (3.65,.32) {\Large .};
\node at (.25,.25) {\Large .};
\node at (3.75,.25) {\Large .};
\node at (.19,.2) {\Large .};
\node at (3.81,.2) {\Large .};
\node at (.13,.14) {\Large .};
\node at (3.87,.14) {\Large .};
\node at (.09,.1) {\Large .};
\node at (3.91,.1) {\Large .};
\node at (.06,.06) {\Large .};
\node at (3.94,.06) {\Large .};
\node at (.03,.03) {\Large .};
\node at (3.97,.03) {\Large .};
\end{scope}

\begin{scope}[shift={(4,0)}]
\node at (2,.6) {\Large .};
\node at (1.6,.58) {\Large .};
\node at (2.4,.58) {\Large .};
\node at (1.25,.55) {\Large .};
\node at (2.75,.55) {\Large .};
\node at (.95,.51) {\Large .};
\node at (3.05,.51) {\Large .};
\node at (.7,.46) {\Large .};
\node at (3.3,.46) {\Large .};
\node at (.5,.4) {\Large .};
\node at (3.5,.4) {\Large .};
\node at (.35,.32) {\Large .};
\node at (3.65,.32) {\Large .};
\node at (.25,.25) {\Large .};
\node at (3.75,.25) {\Large .};
\node at (.19,.2) {\Large .};
\node at (3.81,.2) {\Large .};
\node at (.13,.14) {\Large .};
\node at (3.87,.14) {\Large .};
\node at (.09,.1) {\Large .};
\node at (3.91,.1) {\Large .};
\node at (.06,.06) {\Large .};
\node at (3.94,.06) {\Large .};
\node at (.03,.03) {\Large .};
\node at (3.97,.03) {\Large .};

\begin{scope}[yscale={-1}]
\node at (2,.6) {\Large .};
\node at (1.6,.58) {\Large .};
\node at (2.4,.58) {\Large .};
\node at (1.25,.55) {\Large .};
\node at (2.75,.55) {\Large .};
\node at (.95,.51) {\Large .};
\node at (3.05,.51) {\Large .};
\node at (.7,.46) {\Large .};
\node at (3.3,.46) {\Large .};
\node at (.5,.4) {\Large .};
\node at (3.5,.4) {\Large .};
\node at (.35,.32) {\Large .};
\node at (3.65,.32) {\Large .};
\node at (.25,.25) {\Large .};
\node at (3.75,.25) {\Large .};
\node at (.19,.2) {\Large .};
\node at (3.81,.2) {\Large .};
\node at (.13,.14) {\Large .};
\node at (3.87,.14) {\Large .};
\node at (.09,.1) {\Large .};
\node at (3.91,.1) {\Large .};
\node at (.06,.06) {\Large .};
\node at (3.94,.06) {\Large .};
\node at (.03,.03) {\Large .};
\node at (3.97,.03) {\Large .};
\end{scope}
\end{scope}

\end{tikzpicture}
\end{center}

More precisely, define the first $4$ elements of the sequence as:
$$a_0 = \{(0,0)\}, \quad a_1 = \{(2,0)\}, \quad a_2 = \{(3,0)\}, \quad a_3 = \{(1,0)\}.$$
Then define the next $12$ elements of the sequence by taking
$$a_5 = a_0, \quad a_8 = a_1, \quad a_{11} = a_2, \quad a_{14} = a_3,$$
and then defining $a_{i-1} = \a^{-1}(a_i)$ and $a_{i+1} = \a(a_i)$ for $i = 5,8,11,14$. 
Then define the next $20$ members of the sequence similarly: 
$$a_{18} = a_0, \quad a_{23} = a_1, \quad a_{28} = a_2, \quad a_{33} = a_3,$$
and $a_{i + j} = \a^{j}(a_i)$ for $i = 18,23,28,33$ and $j = \pm1,\pm2$. 
The remainder of the sequence is defined similarly. 

This sequence is eventually small. In fact, every member of this sequence is an atom of $\AA$, and so every member of this sequence is below every partition of $\AA$.

This sequence is eventually dense. Indeed, note that for any isolated $p \in X$, there are infinitely many $n$ such that $a_n = \{p\}$. If $A \in \AA$, then there is some isolated $p \in X$ with $\{p\} \sub A$, and therefore we have infinitely many $n$ with $a_n = \{p\} \sub A$.

Finally, $\seq{\a(a_n)}{n \in \w} \approx \seq{a_{n+1}}{n \in \w}$. To see this, let $\U$ be a partition of $\AA$. 
For many values of $n$, we defined $a_{n+1} = \a(a_n)$, and for these values of $n$ we have $a_{n+1} \approx_\U \a(a_n)$. 
For the other values of $n$, either 
\begin{itemize}
\item[$\circ$] $a_n = \{(0,k)\}$ and $a_{n+1} = \{(2,-k)\}$, or
\item[$\circ$] $a_n = \{(2,k)\}$ and $a_{n+1} = \{(3,k)\}$, or
\item[$\circ$] $a_n = \{(3,-k)\}$ and $a_{n+1} = \{(1,k)\}$, or
\item[$\circ$] $a_n = \{(1,-k)\}$ and $a_{n+1} = \{(0,-k-1)\}$.
\end{itemize}
The partition $\U$ consists of finitely many clopen subsets of $X$, including some $A \ni a$, $B \ni b$, and $C \ni c$. For large enough values of $n$, the corresponding $k$-values in the bulleted list above will also be large, large enough that $(0,k)$ and $(2,-k)$ are both in $B$, and $(2,k)$ and $(3,k)$ are both in $C$, etc. Thus for large enough $n$, we still have $a_{n+1} \approx_\U \a(a_n)$ even when we do not have $a_{n+1} = \a(a_n)$. Because this is true for any partition $\U$, this shows that $\seq{\a(a_n)}{n \in \w} \approx \seq{a_{n+1}}{n \in \w}$ as claimed.

By Theorem~\ref{thm:ESED}, the sequence $\seq{a_n}{n \in \w}$ induces an embedding of $\< \AA,\a \>$ into $\< \pwmff,\s \>$. Let $\eta$ denote this embedding.

This completes our description of the counterexample: $(\<\AA,\a\>,\<\BB,\b\>,\iota,\eta)$ is an instance of the lifting problem for $\< \pwmff,\s \>$. It remains to prove that this instance of the lifting problem has no solution.

Suppose, aiming for a contradiction, that $\bar \eta$ is a lifting of $\eta$ to $\<\BB,\b\>$, so that $\eta = \bar \eta \circ \iota$. 
By Theorem~\ref{thm:ESED}, there is some eventually small, eventually dense sequence $\seq{b_n}{n \in \w}$ in $\BB$ that induces the embedding $\bar \eta$, with the property that $\b(b_n) \wedge b_{n+1} \neq \mathbf{0}$ for all $n$. 

Roughly, the problem here is that to be eventually dense, the sequence $\seq{b_n}{n \in \w}$ must get close to $d$ infinitely often. The only way to do this is to enter the region near $d$ from the region near $c$, and then to exit again into the region near $c$. Using the fact that $\eta = \bar \eta \circ \iota$, this means that the sequence $\seq{a_n}{n \in \w}$ must infinitely often be going near $c$, and then near $b = \pi(d)$, and then near $c$ again. But the sequence of $a_n$'s does not behave this way: it was chosen so that whenever it goes from near $c$ to near $b$, it will, every time, continue on in the direction of $a$, not back towards $c$. 

Let us make this idea more precise.
Let $\U = \{ A,B,C \}$ be a partition of $\AA$ into three elements, with $a \in A$ and $b \in B$, and $c \in C$. 
Note that $\iota(B) = \pi^{-1}(B) \in \BB$, and $b,d \in \iota(B)$. Fix some $\bar B,\bar D \in \BB$ such that $\bar B \cup \bar D = \iota(B)$, with $b \in \bar B$ and $d \in \bar D$. Let $\bar A = \iota(A)$ and $\bar C = \iota(C)$, and let $\V = \{ \bar A,\bar B,\bar C,\bar D \}$. Thus $\V$ is a partition of $\BB$ into four elements, with $a \in \bar A$, $b \in \bar B$, $c \in \bar C$, and $d \in \bar D$.

Because $\seq{a_n}{n \in \w}$ and $\seq{b_n}{n \in \w}$ are eventually small, $a_n$ is below $\U$ and $b_n$ is below $\V$ for sufficiently large $n$. 
Furthermore, because $\eta = \bar \eta \circ \iota$, 
\begin{align*}
[\set{n \in \w}{a_n \leq A}] &\,=\, \eta(A) \,=\, \bar \eta \circ \iota(A) \,=\, \bar \eta(\bar A)
\,=\, [\set{n \in \w}{b_n \leq \bar A}], 
\end{align*}
which means that $\set{n \in \w}{a_n \in A} =^* \set{n \in \w}{b_n \in \bar A}$, and so
$$a_n \leq A \quad \text{if and only if} \quad b_n \leq \bar A$$
for sufficiently large $n$. Because $\iota$ is an embedding $\AA \to \BB$, $a_n \leq A$ if and only if $\iota(a_n) \leq \iota(A) = \bar A$. Hence 
$$\iota(a_n) \leq \bar A \quad \text{if and only if} \quad b_n \leq \bar A$$
for sufficiently large $n$. 
Similarly,
$$\iota(a_n) \leq \bar B \cup \bar D \quad \text{if and only if} \quad b_n \leq \bar B \cup \bar D, \text{ and}$$
$$\iota(a_n) \leq \bar C \quad \text{if and only if} \quad b_n \leq \bar C$$
for sufficiently large $n$. 
Because $a_n$ is below $\U$ and $b_n$ below $\V$ for sufficiently large $n$, this means each of $\iota(a_n)$ and $b_n$ is $\leq$ exactly one of $\bar A$, $\bar B \cup \bar D$, and $\bar C$ for large enough $n$. Combined with the three displayed assertions above, this implies $\iota(a_n) \approx_\V b_n$ for sufficiently large $n$. Fix $N \in \w$ such that $n \geq N$ implies $\iota(a_n) \approx_\V b_n$. 

Using the fact that $\seq{b_n}{n \in \w}$ is eventually dense, fix some $i \geq N$ such that $b_i \not\leq \bar D$ and some $k > i$ such that $b_k \leq \bar D$. 

Let $j$ denote the maximum integer $<\! k$ such that $b_j \not\leq \bar D$. (Note that $N \leq i \leq j < k$.) Let $\ell$ denote the minimum integer $>\! k$ such that $b_\ell \not\leq \bar D$. (Some such $k$ must exist because $\seq{b_n}{n \in \w}$ is eventually dense, and so infinitely often $b_n \leq \mathbf{1} - \bar D$.)
Thus $b_j \not\leq \bar D$ and $b_\ell \not\leq \bar D$, but $b_n \leq \bar D$ for every $n \in (j,\ell)$. 
In fact, because $b_j$ and $b_\ell$ are below $\V$, we have $b_j \wedge \bar D = b_\ell \wedge \bar D = \mathbf{0}$.

Recall that $\b(b_n) \wedge b_{n+1} \neq \mathbf{0}$ for all $n$. Because $b_{\ell-1} \leq \bar D$, we must have $b_\ell \wedge \b(\bar D) \neq \mathbf{0}$. But $\b(\bar D) \leq \bar D \vee \bar C$ and $b_\ell \wedge \bar D = \mathbf{0}$, so this implies $b_\ell \wedge \bar C \neq \mathbf{0}$. Because $b_\ell$ is below $\V$, we see that $b_\ell \leq \bar C$. 
Similarly, because $b_{j+1} \leq \bar D$ we must have $b_j \wedge \b^{-1}(\bar D) \neq \mathbf{0}$. But $\b^{-1}(\bar D) \leq \bar D \vee \bar C$ and $b_j \wedge \bar D = \mathbf{0}$, so this implies $b_j \wedge \bar C \neq \mathbf{0}$. Because $b_j$ is below $\V$, we get $b_j \leq \bar C$.

Hence $b_j,b_\ell \leq \bar C$, and $b_n \leq \bar D$ for all $n \in (j,\ell)$. Because $\iota(a_n) \approx_\V b_n$ for all $n \geq N$, this means that $\iota(a_j) \wedge \bar C \neq \mathbf{0}$, 
$\iota(a_\ell) \wedge \bar C \neq \mathbf{0}$, and
$\iota(a_n) \wedge \bar D \neq \mathbf{0}$ for all $n \in (j,\ell)$.
But $a_n$ is below $\U$ for all $n \in [j,\ell]$, so the only way for this to happen is if 
$a_j \leq C$, then $a_n \leq B$ for all $n \in (j,\ell)$, and then $a_\ell \leq C$.

We have arrived at our sought-after contradiction, because the sequence $\seq{a_n}{n \in \w}$ never exhibits this behavior. If $a_j \leq C$ and $a_n \leq B$ for all $n \in (j,\ell)$, then (by our definition of the $a_n$'s), we must have either $a_\ell \leq B$ or $a_\ell \leq A$.
\end{proof}

As one would probably guess, a ``reverse'' version of Theorem~\ref{thm:NoGo} is also true: not every instance of the lifting problem for $\< \pwmff,\s^{-1} \>$ has a solution. This can of course be proved by suitably modifying the proof of Theorem~\ref{thm:NoGo}. 
Alternatively, one can deduce this directly from Theorem~\ref{thm:NoGo} by showing that if $(\<\AA,\a\>,\<\BB,\b\>,\iota,\eta)$ is an instance of the lifting problem for $\< \pwmff,\s \>$, then $(\<\AA,\a^{-1}\>,\<\BB,\b^{-1}\>,\iota,\eta)$ is an instance of the lifting problem for $\< \pwmff,\s^{-1} \>$, and 
$\bar \eta$ is a solution to one of these problems if and only if it is also a solution to the other. 
This idea (or one nearly identical to it) is laid out in more detail in the proof of Lemma~\ref{lem:niam} below.

\section{The Lifting Lemma, part 1: a back-and-forth argument}\label{sec:back&forth}

Despite the counterexample described in the previous section, we still aim to prove $\<\pwmff,\s\>$ and $\<\pwmff,\s^{-1}\>$ are conjugate under \ch via a back-and-forth argument like the proof of Theorem~\ref{thm:AlgebraicB&F}. 
The idea, as mentioned at the end of Section~\ref{sec:Incompressibility}, is to show that all ``sufficiently nice'' instances of the lifting problem have a solution. 
We can then do our back-and-forth argument carefully, so that we never encounter any not-nice instances of the lifting problem, and the argument goes through. 
What precisely ``sufficiently nice'' means in this context involves the model-theoretic concept of elementarity. 

By the \emph{language of dynamical systems} we mean the formal logical language that, in addition to the usual things like the logical connectives ``and'', ``or'', and $\Rightarrow$, the negation symbol $\neg$, the quantifiers $\forall$ and $\exists$, etc., includes two special constant terms $\mathbf 0$ and $\mathbf 1$, three binary operations $+$, $-$, and $\cdot$, the binary relation $\leq$, and unary function symbols representing an automorphism of a Boolean algebra and its inverse. 
This language can be used to write down properties of a dynamical system $\< \AA,\a \>$. 
For example, $\< \AA,\a \>$ is incompressible if and only if it satisfies the sentence 
$$\forall x \in \AA \, (\neg(x = \mathbf 0) \text{ and } \neg(x=\mathbf 1)) \Rightarrow \neg(\a(x) \leq x).$$
Following the usual practice, all the standard abbreviations are allowed; for example, the sentence above can be written ``$\forall x \in \AA \!\setminus\! \{\mathbf 0,\mathbf 1\} \  \a(x) \not\leq x$" with no ambiguity. 
A \emph{first-order} sentence in the language of dynamical systems is one where the quantifiers $\forall$ and $\exists$ range only over members of a Boolean algebra (not, for example, subsets of that algebra). For example, the sentence displayed above is first-order. 

\begin{definition}
An \emph{elementary embedding} of $\< \AA,\a \>$ into $\< \BB,\b \>$ is a function $\eta: \AA \to \BB$ that preserves the truth value of all first-order sentences in the language of dynamical systems, in the sense that 
$$\< \AA,\a \> \models \varphi(a_1,\dots,a_n,\a) \  \text{ if and only if } \, \< \BB,\b \> \models \varphi(\eta(a_1),\dots,\eta(a_n),\b)$$ 
whenever $a_1,\dots,a_n \in \AA$. 
If $\AA \sub \BB$, then $\< \AA,\b \rest \AA \>$ is called an \emph{elementary substructure} of $\< \BB,\b \>$ if the inclusion map $\AA \xhookrightarrow{} \BB$ is an elementary embedding. 
\hfill
\Coffeecup
\end{definition}

Roughly, $\< \AA,\b \rest \AA\>$ is an elementary substructure of $\< \BB,\b \>$ if there is no difference between them detectable by a first-order sentence in the language of dynamical systems, even when that sentence is allowed to contain parameters from $\AA$. 
Nevertheless, a structure can be quite different from one of its elementary substructures in their second-order properties. For example, the (downward) L\"{o}wenheim-Skolem Theorem tells us that every uncountable dynamical system has countable elementary substructures. Among other things, this means uncountability is not expressible in a first-order sentence. 

We are now in a position to state the paper's primary lemma, mentioned in the introduction, which states that every ``sufficiently nice'' instance of the lifting problem for $\< \pwmff,\s \>$ has a solution. 

\begin{lemma}[Lifting Lemma]\label{lem:main}
Let $(\<\AA,\a\>,\<\BB,\b\>,\iota,\eta)$ be an instance of the lifting problem for $\< \pwmff,\s \>$.
If the image of $\eta$ is an elementary substructure of $\< \pwmff,\s \>$, then this instance of the lifting problem has a solution.
\end{lemma}

The proof of this lemma occupies most of the rest of the paper. In the remainder of this section, we show how to deduce the main theorem of the paper from the lemma. 
The first step is to deduce a ``backwards'' version of the Lifting Lemma from the version stated above. 

\begin{lemma}\label{lem:niam}
Let $(\<\AA,\a\>,\<\BB,\b\>,\iota,\eta)$ be an instance of the lifting problem for $\< \pwmff,\s^{-1} \>$.
If the image of $\eta$ is an elementary substructure of $\< \pwmff,\s^{-1} \>$, then this instance of the lifting problem has a solution.
\end{lemma}
\begin{proof}
If $(\<\AA,\a\>,\<\BB,\b\>,\iota,\eta)$ is an instance of the lifting problem for \linebreak $\< \pwmff,\s^{-1} \>$, then 
$(\< \AA,\a^{-1} \>,\< \BB,\b^{-1} \>,\iota,\eta)$ is an instance of the lifting problem for $\< \pwmff,\s \>$. 
Furthermore, if the image of $\< \AA,\a \>$ under $\eta$ is an elementary substructure of $\< \pwmff,\s \>$, then 
the image of $\< \AA,\a^{-1} \>$ under $\eta$ is an elementary substructure of $\< \pwmff,\s^{-1} \>$. (In fact the images coincide, and this assertion really just amounts to the fact that if $\eta$ transforms true sentences about $\< \AA,\a \>$ into true sentences about $\< \pwmff,\s \>$, then it must also transform true sentences about $\< \AA,\a^{-1} \>$ into true sentences about $\< \pwmff,\s^{-1} \>$. This is obvious, because these two classes of sentences are the same.) 
So assuming Lemma~\ref{lem:main}, there is an embedding $\bar \eta$ of $\< \AA,\a^{-1} \>$ into $\< \pwmff,\s \>$ that solves the instance $(\<\AA,\a^{-1}\>,\<\BB,\b^{-1}\>,\iota,\eta)$ of the lifting problem for $\< \pwmff,\s \>$. 
But then $\bar \eta$ is also a solution to the instance $(\<\AA,\a\>,\<\BB,\b\>,\iota,\eta)$ of the lifting problem for $\< \pwmff,\s^{-1} \>$.
\end{proof}

\begin{theorem}[The Main Theorem]\label{thm:main}
The Continuum Hypothesis implies $\< \pwmff,\s \>$ and $\< \pwmff,\s^{-1} \>$ are conjugate.
\end{theorem}
\begin{proof}
We repeat the proof of Theorem~\ref{thm:AlgebraicB&F} with a few important differences. 
Given $X \sub \pwmff$ with $X$ countable, let $\<\hspace{-1mm}\< X \>\hspace{-1mm}\>$ 
denote not merely the subalgebra generated by $X$, but some countable subalgebra of $\pwmff$ with $X \sub \<\hspace{-1mm}\< X \>\hspace{-1mm}\>$ such that 
$\big\langle \hspace{-.2mm} \<\hspace{-1mm}\< X \>\hspace{-1mm}\>,\s \!\rest\! \<\hspace{-1mm}\< X \>\hspace{-1mm}\> \hspace{-.2mm} \big\rangle$ 
is an elementary substructure of $\< \pwmff,\s \>$. 
(Clearly this contains the subalgebra generated by $X$, but it may be larger.) 
More precisely, let us fix for the remainder of the proof a map $X \mapsto \<\hspace{-1mm}\< X \>\hspace{-1mm}\>$ that sends a countable $X \sub \pwmf$ to a set $\<\hspace{-1mm}\< X \>\hspace{-1mm}\>$ as described above. 
This map is well-defined by the downward L\"{o}wenheim-Skolem theorem, which asserts that some such set exists, and the Axiom of Choice. 
(It is possible to avoid simply ``choosing''  $\<\hspace{-1mm}\< X \>\hspace{-1mm}\>$ by describing a canonical way of obtaining $\<\hspace{-1mm}\< X \>\hspace{-1mm}\>$ from $X$, analogous to the closure of $X$ under algebraic operations to obtain the subalgebra generated by $X$. This process is called \emph{Skolemization}, but defining it precisely is not necessary to our proof.) 
In what follows, we will write 
$\big\langle \hspace{-.2mm} \<\hspace{-1mm}\< X \>\hspace{-1mm}\>,\s \hspace{-.2mm} \big\rangle$ 
and $\big\langle \hspace{-.2mm} \<\hspace{-1mm}\< X \>\hspace{-1mm}\>,\s^{-1} \hspace{-.2mm} \big\rangle$ 
instead of $\big\langle \hspace{-.2mm} \<\hspace{-1mm}\< X \>\hspace{-1mm}\>,\s \!\rest\! \<\hspace{-1mm}\< X \>\hspace{-1mm}\> \hspace{-.2mm} \big\rangle$ and $\big\langle \hspace{-.2mm} \<\hspace{-1mm}\< X \>\hspace{-1mm}\>,\s^{-1} \!\rest\! \<\hspace{-1mm}\< X \>\hspace{-1mm}\> \hspace{-.2mm} \big\rangle$.

Generally $\card{\pwmff} = \continuum$, so $\card{\pwmff} = \aleph_1$ because we are assuming \ch. 
Let $\seq{a_\xi}{\xi < \w_1}$ be an enumeration of $\pwmff$ in order type $\w_1$. 
 
Using transfinite recursion, we now build a sequence $\seq{\phi_\a}{\a < \w_1}$ of maps such that each $\phi_\a$ is a conjugacy from its domain (which we denote $\AA_\a$) to its codomain (which we denote $\BB_\a$), or more precisely from $\< \AA_\a,\s \>$ to $\< \BB_\a,\s^{-1} \>$. 
Furthermore, the following will hold at every stage $\a$ of the recursion:
\begin{itemize}
\item[$\circ$] $\AA_\xi \sub \AA_\a$ and $\BB_\xi \sub \BB_\a$ for all $\xi \leq \a$, 
\item[$\circ$] $\AA_\a = \bigcup_{\xi < \a}\AA_\xi$ and $\BB_\a = \bigcup_{\xi < \a}\BB_\xi$ for limit $\a$, 
\item[$\circ$] the $\phi_\xi$'s are coherent, i.e., $\phi_\xi = \phi_\z \rest \AA_\xi$ whenever $\xi \leq \z \leq \a$, 
\item[$\circ$] $\< \AA_\a,\s \>$ is an elementary substructure of $\< \pwmff,\s \>$.
\end{itemize}

For the base case of the recursion, let $\AA_0 = \<\hspace{-1mm}\< \0 \>\hspace{-1mm}\>$. (In other words, $\AA_0$ is some countable subalgebra of $\pwmff$ such that $\< \AA_0,\s \>$ is an elementary substructure of $\< \pwmff,\s \>$.) 
Because incompressibility is expressible by a first-order sentence and $\< \pwmff,\s \>$ is incompressible, elementarity implies $\< \AA_0,\s \>$ is also incompressible. 
Recall that every countable incompressible dynamical system embeds in $\< \pwmff,\s^{-1} \>$. (This was mentioned in Section~\ref{sec:Incompressibility}. It is proved, in dual form, in \cite{Brian1}; it is also proved below as Theorem~\ref{thm:CountableIncompressibleEmbeds}.)  Let $\phi_0$ be an embedding from $\< \AA_0,\s \>$ into $\< \pwmff,\s^{-1} \>$. Let $\BB_0$ denote the image of $\AA_0$ under $\phi_0$. Thus 
$\phi_0$ is a conjugacy from $\< \AA_0,\s \>$ to $\< \BB_0,\s^{-1} \>$, and 
$\phi_0^{-1}$ is a conjugacy from $\< \BB_0,\s^{-1} \>$ to $\< \AA_0,\s \>$.

At limit stages of the recursion there is nothing to do: take $\AA_\a = \bigcup_{\xi < \a}\AA_\xi$, $\BB_\a = \bigcup_{\xi < \a}\BB_\xi$, and $\phi_\a = \bigcup_{\xi < \a}\phi_\xi$, which is well-defined because the $\phi_\xi$ are coherent. Recall that an increasing union of elementary substructures is itself an elementary substructure, so taking unions like this preserves the hypotheses of the recursion through limit stages. 

At successor steps we obtain $\phi_{\a+1}$ from $\phi_\a$ using Lemmas \ref{lem:main} and \ref{lem:niam}. 
First, let $\BB_{\a+1}^0 = \<\hspace{-1mm}\< \BB_\a \cup \{a_\a\} \>\hspace{-1mm}\>$ and observe that 
$$(\<\BB_\a,\s^{-1}\>,\<\BB^0_{\a+1},\s^{-1}\>,\iota,\phi_\a^{-1})$$ 
is an instance of the lifting problem for $\< \pwmff,\s \>$, where $\iota$ denotes the inclusion $\BB_\a \xhookrightarrow{} \BB_{\a+1}^0$. Furthermore, the image of $\phi^{-1}_\a$, namely $\< \AA_\a,\s \>$, is an elementary substructure of $\< \pwmff,\s \>$. 
By Lemma~\ref{lem:main}, there is an embedding $\phi^0_{\a+1}$ from $\< \BB_{\a+1}^0,\s^{-1} \>$ into $\< \pwmff,\s \>$. 
Let $\AA^0_{\a+1}$ denote the image of this embedding. 
(Note that there is no obvious reason why $\< \AA^0_{\a+1},\s \>$ should be an elementary substructure of $\< \pwmff,\s\>$. Instead, we now have elementarity on the other side of our mappings, with $\BB_{\a+1}^0 = \<\hspace{-1mm}\< \BB_\a \cup \{a_\a\} \>\hspace{-1mm}\>$.) 
Because $\phi^0_{\a+1}$ is a conjugacy from $\< \BB^0_{\a+1},\s^{-1} \>$ to $\< \AA^0_{\a+1},\s \>$, 
$(\phi^0_{\a+1})^{-1}$ is a conjugacy from $\< \AA^0_{\a+1},\s \>$ to $\< \BB^0_{\a+1},\s^{-1} \>$.

\begin{center}
\begin{tikzpicture}[scale=.8]

\node at (-.01,0) {\small $\AA_\a$};
\node at (3.99,0) {\small $\BB_\a$};
\node at (-.21,2) {\small $\AA_{\a+1}^0$};
\node at (5.7,2) {\small $\BB_{\a+1}^0 = \<\hspace{-1mm}\< \BB_\a \cup \{a_\a\} \>\hspace{-1mm}\>$};
\node at (-.21,3.5) {\small $\AA_{\a+1}^0$};
\node at (4.2,3.5) {\small $\BB_{\a+1}^0$};
\node at (-.21,5.5) {\small $\AA_{\a+1}$};
\node at (4.21,5.5) {\small $\BB_{\a+1}$};
\node at (6.28,5.5) {\small $= \mathrm{Image}(\phi_{\a+1})$};
\node at (-2.35,2) {\small $\mathrm{Image}(\phi_{\a+1}^0) =$};
\node at (-2.53,5.5) {\small $\<\hspace{-1mm}\< \AA_{\a+1}^0 \cup \{a_\a\} \>\hspace{-1mm}\> =$};

\draw[<-] (.6,0) -- (3.4,0);
\draw[<-] (.6,2) -- (3.4,2);
\draw[->] (.6,3.5) -- (3.4,3.5);
\draw[->] (.6,5.5) -- (3.4,5.5);

\node at (3.95,1) {\footnotesize \rotatebox{90}{$\sub$}};
\node at (-.3,2.75) {\footnotesize \rotatebox{90}{$=$}};
\node at (4,2.75) {\footnotesize \rotatebox{90}{$=$}};
\node at (-.3,4.5) {\footnotesize \rotatebox{90}{$\sub$}};

\node at (2,.3) {\footnotesize $\phi_\a^{-1}$};
\node at (2,2.3) {\footnotesize $\phi_{\a+1}^0$};
\node at (2,3.8) {\footnotesize $(\phi_{\a+1}^0)^{-1}$};
\node at (2,5.8) {\footnotesize $\phi_{\a+1}$};

\end{tikzpicture}
\end{center}

Second, let $\AA_{\a+1} = \<\hspace{-1mm}\< \AA_\a^0 \cup \{a_\a\} \>\hspace{-1mm}\>$ and observe that 
$$(\<\AA^0_{\a+1},\s\>,\<\AA_{\a+1},\s\>,\iota,(\phi_\a^0)^{-1})$$ 
is an instance of the lifting problem for $\< \pwmff,\s^{-1} \>$, where $\iota$ denotes the inclusion $\AA^0_{\a+1} \xhookrightarrow{} \AA_{\a+1}$. Furthermore, the image of $(\phi^0_{\a+1})^{-1}$, namely $\< \BB_{\a+1}^0,\s^{-1} \>$, is an elementary substructure of $\< \pwmff,\s^{-1} \>$. 
By Lemma~\ref{lem:main}, there is an embedding $\phi_{\a+1}$ from $\< \AA_{\a+1},\s \>$ into $\< \pwmff,\s^{-1} \>$. 
Let $\BB_{\a+1}$ denote the image of this embedding, so that 
$\phi_{\a+1}$ is a conjugacy from $\< \AA_{\a+1},\s \>$ to $\< \BB_{\a+1},\s^{-1} \>$. 

This completes the recursion. Because the $\phi_\a$ are coherent conjugacies from substructures of $\< \pwmff,\s \>$ to substructures of $\< \pwmff,\s^{-1} \>$, $\phi = \bigcup_{\a < \w_1}\phi_\a$ is a conjugacy from $\big\< \bigcup_{\a < \w_1} \AA_\a,\s \big\>$ to $\big\< \bigcup_{\a < \w_1} \BB_\a,\s^{-1} \big\>$. 
But $\bigcup_{\a < \w_1} \AA_\a = \pwmff$ and $\bigcup_{\a < \w_1} \BB_\a = \pwmff$, because we ensured at stage $\a+1$ of the recursion that $a_\a \in \AA_{\a+1}$ and $a_\a \in \BB_{\a+1}$. 
\end{proof}

\section{The Lifting Lemma, part 2: directed graphs}\label{sec:Digraphs}

We now set about proving the Lifting Lemma. 
This entails a close look at the ``finite substructures'' of $\< \pwmff,\s \>$ and $\< \pwmff,\s^{-1} \>$, and how they all fit together. 
Following \cite{Darji,GM,Brian2}, we use digraphs to encode the finitary fragments of an infinite dynamical system. 

\begin{definition}
A \emph{directed graph}, or \emph{digraph}, is a pair $\< \V,\to \>$, where $\V$ is a finite set (the vertices) and $\to$ is a relation on $\V$ (possibly asymmetric, and possibly with ``loops'', i.e., vertices $v \in \V$ with $v \to v$).
\hfill{\Coffeecup}
\end{definition}

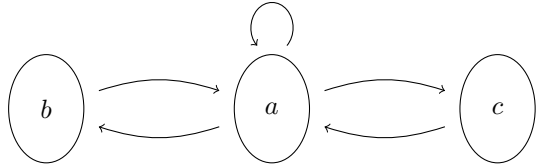
\begin{center}
\begin{tikzpicture}[yscale=1.2]

\draw (0,0) ellipse (5mm and 6mm);
\draw (3,0) ellipse (5mm and 6mm);
\draw (6,0) ellipse (5mm and 6mm);

\node at (0,0) {\small $b$};
\node at (3,0) {\small $a$};
\node at (6,0) {\small $c$};

\draw[->] (3.2,.7) arc (-45:225:2.8mm);
\draw [->] (.7,.2) to [out=15,in=165] (2.3,.2);
\draw [<-] (.7,-.2) to [out=-15,in=-165] (2.3,-.2);
\draw [->] (3.7,.2) to [out=15,in=165] (5.3,.2);
\draw [<-] (3.7,-.2) to [out=-15,in=-165] (5.3,-.2);

\end{tikzpicture}
\end{center}

\vspace{3mm}

\noindent When convenient, to keep the clutter in our pictures to a minimum, we will sometimes use a two-headed arrow between $x$ and $y$ when $x \to y$ and $y \to x$ (unlike in the picture above).

\begin{definition}\label{def:digraphs}
Suppose $\< \U,\toU \>$ and $\< \V,\toV \>$ are directed graphs. An \emph{epimorphism} from $\< \V,\toV \>$ onto $\< \U,\toU \>$ is a surjection $\phi: \V \to \U$ such that
for any $u,u' \in \U$, we have $u \toU u'$ if and only if there are some $v \in \phi^{-1}(u)$ and $v' \in \phi^{-1}(u')$ such that $v \toV v'$. 

In other words, arrows in $\V$ map to arrows in $\U$, and every arrow in $\U$ is the image of some arrow in $\V$.

In this case, we say that $\< \V,\toV \>$ is a \emph{refinement} of $\< \U,\toU \>$ (via $\phi$).
\hfill \Coffeecup
\end{definition}

For example, the following directed graph is a refinement of the one drawn above, and the epimorphism $\phi$ witnessing this fact is indicated by where the vertices of the new graph are drawn.

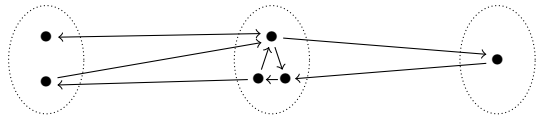
\vspace{2mm}

\begin{center}
\begin{tikzpicture}[yscale=1.2]

\draw[densely dotted] (0,0) ellipse (5mm and 6mm);
\draw[densely dotted] (3,0) ellipse (5mm and 6mm);
\draw[densely dotted] (6,0) ellipse (5mm and 6mm);

\node at (0,.25) {\small $\bullet$};
\node at (0,-.25) {\small $\bullet$};
\node at (3,.25) {\small $\bullet$};
\node at (3.18,-.22) {\small $\bullet$};
\node at (2.82,-.22) {\small $\bullet$};
\node at (6,0) {\small $\bullet$};

\draw[<->] (.16,.25)--(2.85,.29);
\draw[->] (.15,-.2)--(2.86,.19);
\draw[<-] (.15,-.28)--(2.68,-.22);
\draw[->] (3.04,.14)--(3.14,-.11);
\draw[<-] (2.96,.14)--(2.86,-.11);
\draw[<-] (2.92,-.22)--(3.08,-.22);
\draw[<-] (3.31,-.21)--(5.85,-.04);
\draw[->] (3.15,.24)--(5.85,.06);

\end{tikzpicture}
\end{center}

\vspace{1mm}

\begin{definition}
Given a dynamical system $\< \AA,\a \>$, we associate to every partition $\A$ of $\AA$ a digraph $\< \A,\toa \>$, where $\A$ itself is the vertex set, and we define $\toa$ as the ``hitting relation'' for $\a$: that is, 
$$x \toa y \quad \text{ if and only if } \quad \a(x) \wedge y \neq \mathbf{0}$$
for all $x,y \in \A$.
\hfill{\Coffeecup}
\end{definition}

\begin{lemma}\label{lem:Epimorphism}
Let $\< \AA,\a \>$ be a dynamical system. Suppose $\A$ and $\B$ are partitions of $\AA$, and $\B$ refines $\A$. Then the natural map $\pi: \B \to \A$ is an epimorphism from $\< \B,\toa \>$ to $\< \A,\toa \>$.
\end{lemma}
\begin{proof}
Because $\B$ is a partition refining $\A$, for every $a \in \A$ there is at least one $b \in \B$ with $b \leq a$, i.e., with $\pi(b) = a$. Thus $\pi$ is a surjection $\B \to \A$. 

Suppose $b,b' \in \B$ and $b \toa b'$. This means $\a(b) \wedge b' \neq \mathbf{0}$, which implies $\a(\pi(b)) \wedge \pi(b') \neq \mathbf{0}$ (because $b \leq \pi(b)$ and $b' \leq \pi(b')$), i.e. $\pi(b) \toa \pi(b')$. 

On the other hand, suppose that $a,a' \in \A$ and $a \toa a'$. This means $\a(a) \wedge a' \neq \mathbf{0}$, or equivalently, $a \wedge \a^{-1}(a') \neq \mathbf{0}$. This implies that there is some $b \in \B$ with $b \wedge a \wedge \a^{-1}(a') \neq \mathbf{0}$. 
Hence $b \wedge a \neq \mathbf{0}$ and $\a(b) \wedge a' \neq \mathbf{0}$. 
The fact that $b \wedge a \neq \mathbf{0}$ implies $\pi(b) = a$. 
The fact that $\a(b) \wedge a' \neq \mathbf{0}$ implies there is some $b' \in \B$ with $b' \wedge \a(b) \wedge a' \neq \mathbf{0}$. 
But $b' \wedge a' \neq \mathbf{0}$ implies $\pi(b') = a'$, and $b' \wedge \a(b) \neq \mathbf{0}$ means $b \toa b'$. 
Thus there are $b,b' \in \B$ such that $\pi(b) = a$, $\pi(b') = a'$, and $b \toa b'$.
\end{proof}

This lemma explains our choice of terminology in the last part of Definition~\ref{def:digraphs}: $\< \B,\toa \>$ is a (digraph) refinement of $\< \A,\toa \>$, via the natural map, whenever $\B$ is a (partition) refinement of $\A$.

\begin{definition}
A \emph{walk} in/through a digraph $\< \V,\to \>$ is a finite sequence $\< a_0,a_1,a_2,\dots,a_n \>$ of members of $\V$ such that $a_i \to a_{i+1}$ for all $i < n$. We say that this walk has length $n$ (counting the arrows, not the vertices), and that it is a walk from $a_0$ to $a_n$.

A digraph $\< \V,\to \>$ is \emph{strongly connected} if for any $a,b \in \V$, there is a walk in $\< \V,\to \>$ from $a$ to $b$.
\hfill{\Coffeecup}
\end{definition}

The next two theorems show roughly that the strongly connected digraphs capture the property of incompressibility in a dynamical system.

\begin{theorem}\label{thm:Incompressible=StronglyConnected}
A dynamical system $\< \AA,\a \>$ is incompressible if and only if $\< \A,\toa \>$ is strongly connected for every partition $\A$ of $\AA$.
\end{theorem}
\begin{proof}
For the forward implication, suppose $\< \AA,\a \>$ is an incompressible dynamical system, and let $\A$ be a partition of $\AA$. Fix $a \in \A$, and let 
$$C_a = \set{b \in \A}{\text{there is a walk in $\< \A,\toa \>$ from $a$ to $b$}}.$$ 
If $b \in C_a$ and $b \toa c$, then any walk from $a$ to $b$ can be extended to a walk from $a$ to $c$, and it follows that $c \in C_a$. Thus $\a(b) \leq \bigvee C_a$ for every $b \in C_a$. 
It follows that $\a(\bigvee C_a) = \bigvee \set{\a(b)}{b \in C_a} \leq \bigvee C_a$. 
Because $\< \AA,\a \>$ is incompressible (and $C_a \neq \0$), this implies $\bigvee C_a = \mathbf{1}$. Thus $C_a = \A$, which means there is a walk in $\< \A,\toa \>$ from $a$ to any other member of $\A$. But $a$ was arbitrary, so this means $\< \A,\toa \>$ is strongly connected.

For the reverse implication, suppose $\< \AA,\a \>$ is not incompressible, which means there is some $x \in \AA \setminus \{\mathbf{0},\mathbf{1}\}$ such that $\a(x) \leq x$. Let $\A = \{ x,\mathbf{1}-x \}$. 

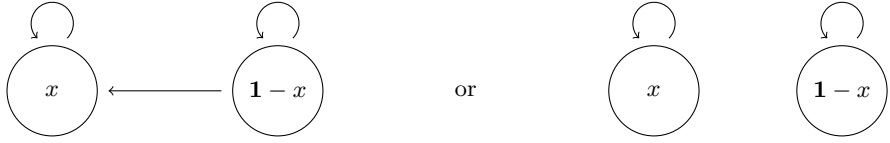
\begin{center}
\begin{tikzpicture}[scale=1]

\draw (0,0) ellipse (6mm and 6mm);
\draw (3,0) ellipse (6mm and 6mm);

\node at (0,0) {\footnotesize $x$};
\node at (3,0) {\footnotesize $\mathbf{1}-x$};

\draw[->] (.2,.7) arc (-45:225:2.8mm);
\draw[->] (3.2,.7) arc (-45:225:2.8mm);
\draw [<-] (.75,0) -- (2.25,0);

\begin{scope}[shift={(8,0)}]
\draw (0,0) ellipse (6mm and 6mm);
\draw (2.5,0) ellipse (6mm and 6mm);

\node at (0,0) {\footnotesize $x$};
\node at (2.5,0) {\footnotesize $\mathbf{1}-x$};

\draw[->] (.2,.7) arc (-45:225:2.8mm);
\draw[->] (2.7,.7) arc (-45:225:2.8mm);
\end{scope}

\node at (5.5,0) {\footnotesize or};

\end{tikzpicture}
\end{center}

\noindent The digraph $\< \A,\toa \>$ is not strongly connected: because $x \leq \a(x)$, we do not have $x \toa \mathbf{1}-x$, and there is no walk from $x$ to $\mathbf{1}-x$.
\end{proof}

\begin{lemma}\label{lem:PathExistence}
Let $\< \V,\to \>$ be a strongly connected digraph, and let $v,w \in \V$. 
There is a walk $\< v_0,v_1,\dots,v_n \>$ through $\< \V,\to \>$ such that $v_0 = v$ and $v_n = w$, and such that for any $x,y \in \V$ with $x \to y$, there is some $i < n$ such that $v_i = x$ and $v_{i+1} = y$. 
(In other words, every arrow in the digraph is represented in the walk.)
\end{lemma}
\begin{proof}
Fix an enumeration $\set{(a_i,b_i)}{i \leq k}$ of (the ordered pairs in) the edge relation $\to$, such that $a_0 = v$ and $b_k = w$.  
To get the desired walk, first let $v_0 = a_0$ and $v_1 = b_0$. Because $\< \V,\to \>$ is strongly connected, there is a walk from $b_0$ to $a_1$; so we may define $v_2,v_3, \dots, v_{\ell_1}$ (for some $\ell_1$) so that $\< b_0,v_2,v_3,\dots,v_{\ell_1},a_1 \>$ is a walk from $b_0$ to $a_1$. Take $v_{\ell_1+1} = a_1$ and $v_{\ell_1+2} = b_1$, and then define $v_{\ell_1+3},v_{\ell_1+4}, \dots, v_{\ell_2}$ (for some $\ell_2 \geq \ell_1+3$) so that $\< b_1,v_{\ell_1+3},v_{\ell_1+4},\dots,v_{\ell_2},a_2 \>$ is a walk from $b_1$ to $a_2$. Take $v_{\ell_2+1} = a_2$ and $v_{\ell_2+2} = b_2$, and then choose the next few $v_i$ to form a walk to $a_3$. Continuing in this way, we find a walk from $a_0 = v$ to $b_k = w$ such that every arrow in the digraph is represented in the walk.
\end{proof}

\begin{theorem}\label{thm:PartitionsOfPWMFF}
Given a digraph $\< \V,\to \>$, the following are equivalent:
\begin{enumerate}
\item $\< \V,\to \>$ is strongly connected.
\item There is an incompressible dynamical system $\< \AA,\a \>$ and a partition $\A$ of $\AA$ such that $\< \V,\to \> \iso \< \A,\toa \>$. 
\item There is a partition $\A$ of $\pwmff$ such that $\< \V,\to \> \iso \< \A,\tosi \>$.
\end{enumerate}
\end{theorem}
\begin{proof}
$(3) \Rightarrow (2)$ because $\< \pwmff,\s \>$ is incompressible (by Theorem~\ref{thm:It'sIncompressible}). 
$(2) \Rightarrow (1)$ by Theorem~\ref{thm:Incompressible=StronglyConnected}. 

To prove $(1) \Rightarrow (3)$, 
fix some $w \in \V$ and then, applying the previous lemma, fix a walk $\seq{v_i}{i \leq n}$ in $\< \V,\to \>$ such that $v_0 = v_n = w$ and such that every arrow in the digraph is represented in the walk. 
For each $v \in \V$, define 
$$A_v = \set{nk+i}{k \geq 0 \text{ and } v_i = v}$$
and let $\A = \set{[A_v]}{v \in \V}$. We claim that $\A = \set{[A_v]}{v \in \V}$ is a partition of $\pwmff$ and $\< \A,\tosi \> \iso \< \V,\to \>$.

Because $\< \V,\to \>$ is strongly connected, every $v \in \V$ has in-degree and out-degree $>\!0$, and this implies (as every arrow is represented in our path) that $v = v_i$ for at least one $i \leq n$. Thus $A_v$ is infinite. 
Furthermore, it is clear that $A_u \cap A_v = \0$ whenever $u \neq v$, and that $\bigcup_{v \in \V}A_v = \w$. 
Hence $\A = \set{[A_v]}{v \in \V}$ is a partition of $\pwmff$. 

To finish the proof, we show that the map $v \mapsto [A_v]$ is an isomorphism from $\< \V,\to \>$ to $\< \A,\tosi \>$. 
This map is clearly bijective, so we must show that it preserves the edge relation. If $u,v \in \V$ and $u \to v$, then there is some $i < n$ such that $v_i = u$ and $v_{i+1} = v$. So $(A_u+1) \cap A_v \supseteq \set{nk+i+1}{k \geq 0},$ 
and it follows that $\s([A_u]) \wedge [A_v] = [(A_u+1) \cap A_v] \geq [\set{nk+i+1}{k \geq 0}] \neq \mathbf{0}$, which means $[A_u] \tosi [A_v]$. 
Conversely, if $u,v \in \V$ and $[A_u] \tosi [A_v]$, then we have $j+1 \in A_v$ for infinitely many $j \in A_u$. By our definition of $A_u$ and $A_v$, this means that $u = v_i$ and $v = v_{i+1}$ for some $i < n$. Because $\seq{v_j}{j \leq n}$ is a walk, this implies $u \to v$.
\end{proof}

We end this section with the following important observation, which was used in the proof of Theorem~\ref{thm:main} above. This theorem, in its dual topological form, already appeared in \cite{Brian1}, but we include a proof here for the sake of completeness.

\begin{theorem}\label{thm:CountableIncompressibleEmbeds}
A countable dynamical system $\<\AA,\a\>$ embeds in $\< \pwmff,\s \>$ if and only if it is incompressible.
\end{theorem}
\begin{proof}
The forward direction was proved already: every dynamical system that embeds in $\< \pwmff,\s \>$ is incompressible by Theorem~\ref{thm:Incompressibility}.

For the reverse direction, fix a sequence $\seq{\A_n}{n \in \w}$ of partitions of $\AA$ satisfying the conclusion of Lemma~\ref{lem:NicePartitionSequence}: that is, $\A_n$ refines $\A_m$ whenever $m \leq n$, and for every partition $\V$ of $\AA$, $\A_n$ refines $\V$ for all sufficiently large $n$. 
Because $\< \AA,\a \>$ is incompressible, Theorem~\ref{thm:Incompressible=StronglyConnected} tells us that $\< \A_n,\toa \>$ is strongly connected for every $n$.

Let $\U$ be an ultrafilter on $\AA$, and for each $n \in \w$ let $a^n_0$ be the (unique) member of $\A_n$ in $\U$. (This implies $a^0_0 \geq a^1_0 \geq a^2_0 \geq \dots$.) 
Applying Lemma~\ref{lem:PathExistence}, for each $n$ let $\seq{a^n_i}{i \leq \ell_n}$ be a walk in $\< \A_n,\toa \>$ such that, for every $n$, $a^n_0 = a^n_{\ell_n} = a^n_0$ and every arrow in $\< \A_n,\toa \>$ is represented in the walk. 

Let $\seq{b_i}{i \in \w}$ be the infinite sequence of members of $\AA$ obtained by concatenating the walks $\seq{a^0_i}{i < \ell_0}$, $\seq{a^1_i}{i < \ell_1}$, $\seq{a^2_i}{i < \ell_2}$, etc., in that order. In other words, $b_i = a^n_j$ if and only if $i = \ell_0+\ell_1+\dots+\ell_{n-1}+j$ for some $j < \ell_n$. 

If $\V$ is a partition of $\AA$, then $\A_n$ refines $\V$ for sufficiently large $n$, which means $b_i$ is below $\V$ for sufficiently large $i$. Hence $\seq{b_i}{i \in \w}$ is eventually small. 
If $a \in \AA \setminus \{\mathbf 0\}$ then for all sufficiently large $n$ there is some $a' \in \A_n$ with $a' \leq a$, which implies there is some $j < \ell_n$ with $a^n_j = a' \leq a$. Hence there are infinitely many $i$ with $b_i \leq a$, and as $a$ was arbitrary, this shows $\seq{b_i}{i \in \w}$ is eventually dense. 

Finally, fix a partition $\V$ of $\AA$, and suppose $n$ is large enough that $\A_n$ refines $\V$. 
If $b_i = a^n_j$ and $b_{i+1} = a^n_{j+1}$, then $b_i \toa b_{i+1}$ because $\< a^n_j :\, j \leq \ell_n \>$ is a walk. 
Because $\A_n$ refines $\V$, this implies $\a(b_i) \approx_\V b_{i+1}$. 
If instead $b_i = a^n_{\ell_n-1}$ and $b_{i+1} = a^{n+1}_0$, then $\a(b_i) \approx_{\A_n} b_{i+1}$ because 
$$b_i \,=\, a^n_{\ell_n-1} \,\toa\, a^n_{\ell_n} \,=\, a^n_0 \,\geq\, a^{n+1}_0 \,=\, b_{i+1}.$$  
Because $\A_n$ refines $\V$, we have $\a(b_i) \approx_\V b_{i+1}$ in this case as well. 
In either case, $b_i \approx_\V b_{i+1}$ for all sufficiently large $n$. 
As $\V$ was an arbitrary partition of $\AA$, this means $\seq{\a(b_i)}{i \in \w} \approx \seq{b_{i+1}}{i \in \w}$. 
By Theorem~\ref{thm:CharacterizingGoodSequences}, $\seq{b_i}{i \in \w}$ induces an embedding of $\< \AA,\a \>$ into $\< \pwmff,\s \>$.
\end{proof}

\section{The Lifting Lemma, part 3: a diagonal argument}\label{sec:Diagonalization}

In Theorem~\ref{thm:Dagger} below, we give a necessary and sufficient condition for the existence of a solution to a given instance of the lifting problem for $\< \pwmff,\s \>$. 
Roughly, the condition is that we can solve an instance of the lifting problem if and only if we can solve arbitrarily large ``finitary fragments'' of the problem (where these fragments are encoded via digraphs as described in the previous section). 
The proof of this theorem is essentially just a diagonalization argument, albeit a somewhat intricate one.  
First we need several definitions and lemmas. 

\begin{definition}\label{def:Walks}
An \emph{infinite walk} through a digraph $\< \V,\to \>$ is an infinite sequence $\seq{v_n}{n \in \w}$ in $\V$ such that
$v_n \to v_{n+1}$ for all $n \in \w$. 

A \emph{diligent walk} in/through $\<\V,\to\>$ is an infinite walk with the property that every arrow is traversed infinitely often, or more precisely: 
\begin{itemize}
\item[$\circ$] for each $v,w \in \V$ with $v \to w$, there are infinitely many $n \in \w$ such that $v_n = v$ and $v_{n+1} = w$. 
\end{itemize}

Given an infinite walk $\seq{v_n}{n \in \w}$ (diligent or not) in a digraph $\< \V,\to \>$, for each $v \in \V$ define $A_v = [\set{n \in \w}{v_n = v}]$. 
It is not difficult to see that $\set{A_v}{v \in \V} \setminus \{[\0]\}$ is a partition of $\pwmff$; we call this the \emph{partition of $\pwmff$ associated to the walk $\seq{v_n}{n \in \w}$}.
\hfill{\Coffeecup}
\end{definition}

\begin{lemma}\label{lem:Walks0}
A digraph $\< \V,\to \>$ is strongly connected if and only if it has no isolated points and there is a diligent walk through $\< \V,\to \>$.
\end{lemma}
\begin{proof}
If $\< \V,\to \>$ is strongly connected, then clearly it has no isolated points. Furthermore, by Lemma~\ref{lem:PathExistence} there is a walk $\< v_0,v_1,\dots,v_k\>$ through $\< \V,\to \>$ such that $v_k = v_0$, and for each $v,w \in \V$ with $v \to w$, there is some $i < k$ such that $v_i = v$ and $v_{i+1} = w$. 
For each $n \in \w$, let $v_n = v_i$ if and only if $i \equiv n$ modulo $k$. Then $\seq{v_n}{n \in \w}$ is a diligent walk through $\< \V,\to \>$. 

Conversely, suppose $\< \V,\to \>$ has no isolated points, and $\seq{v_n}{n \in \w}$ is a diligent walk through $\< \V,\to \>$. 
Let $u,v \in \V$. 
Because $u$ and $v$ are not isolated in $\< \V,\to \>$, the definition of ``diligent'' implies $u$ and $v$ both appear infinitely often in $\seq{v_n}{n \in \w}$. 
Consequently, there is some $m \in \w$ with $v_m = u$ and some $n > m$ with $v_n = v$. But then $\<v_m,v_{m+1},\dots,v_n\>$ is a walk from $u$ to $v$ in $\< \V,\to \>$. As $u$ and $v$ were arbitrary, $\< \V,\to \>$ is strongly connected.
\end{proof}

The next two lemmas express the most important aspect of diligent walks: a diligent walk through a digraph $\< \V,\to \>$ is equivalent to an identification of $\< \V,\to \>$ with a digraph of the form $\< \A,\tosi \>$, where $\A$ is a partition of $\pwmff$. 

\begin{lemma}\label{lem:Walks1}
Let $\< \V,\to \>$ be a strongly connected digraph and let $\seq{v_n}{n \in \w}$ be a diligent walk through $\< \V,\to \>$. 
Let $\A = \set{A_v}{v \in \V}$ be the partition of $\pwmff$ associated to this walk. Then the map $v \mapsto A_v$ is an isomorphism from $\< \V,\to \>$ to $\< \A,\tosi \>$.
\end{lemma}
\begin{proof}
Let $\< \V,\to \>$ be a strongly connected digraph, let $\seq{v_n}{n \in \w}$ a diligent walk through $\< \V,\to \>$, and let $\A = \set{A_v}{v \in \V}$ be the partition of $\pwmff$ associated to this walk. 
The map $v \mapsto A_v$ is clearly a bijection. 
If $v \to w$, then (by the definition of a diligent walk) there are infinitely many $n$ such that $v_n = v$ and $v_{n+1} = w$, and this implies $A_v \tosi A_w$. 
Conversely, if $A_v \tosi A_w$, then this means that 
\begin{align*}
[\set{n}{v_n = v \text{ and }v_{n+1}=w}] & \,=\,  [\set{n}{v_n=v}] \wedge [\set{n}{v_{n+1}=w}] \\
&\,=\,  [\set{n}{v_n=v}] \wedge [\set{n-1}{v_n=w}] \\
&\,=\, A_v \wedge \s^{-1}(A_w) \,\neq\, [\0].
\end{align*}
In other words, there are infinitely many $n$ such that $v_n = v$ and $v_{n+1} = w$, which implies $v \to w$. 
Hence $v \to w$ if and only if $A_v \tosi A_w$.
\end{proof}

\begin{definition}
Suppose that $\< \V,\to \>$ is a strongly connected digraph, and $\seq{v_n}{n \in \w}$ is a diligent walk through $\< \V,\to \>$. The map $v \mapsto A_v$ is called the \emph{natural isomorphism} from $\< \V,\to \>$ to $\< \A,\tosi \>$ or, more precisely, the natural isomorphism \emph{associated to the walk $\seq{v_n}{n \in \w}$}.
\hfill
\Coffeecup
\end{definition}

\begin{lemma}\label{lem:Walks2}
Suppose $\A$ is a partition of $\pwmff$, and suppose that there is a digraph $\< \V,\to \>$ and an isomorphism $\phi$ from $\< \V,\to \>$ to $\< \A,\tosi \>$. Then there is a diligent walk through $\< \V,\to \>$ such that $\phi$ is the natural isomorphism associated to the walk. 
\end{lemma}

\begin{proof}
Let $\A$ be a partition of $\pwmff$, let $\< \V,\to \>$ be a digraph, and let $\phi$ be an isomorphism from $\< \V,\to \>$ to $\< \A,\tosi \>$. Note that this implies $\< \V,\to \>$ is strongly connected by Theorem~\ref{thm:PartitionsOfPWMFF}.

For each $a \in \A$ fix some $X_a \sub \w$ such that $a = [X_a]$. 
If $a \neq b$ then $a \wedge b = [\0]$, which means that $X_a \cap X_b$ is finite. 
Also, because $\bigvee \A = [\w]$, $\set{n \in \w}{n \notin \bigcup_{a \in \A}X_a}$ is finite.
Similarly, if it is not the case that $a \tosi b$ then $\set{n}{n \in X_a \text{ and } n+1 \in X_b}$ is finite, but if it is the case that $a \tosi b$ then $\set{n}{n \in X_a \text{ and } n+1 \in X_b}$ is infinite. 

Consequently, there is some sufficiently large value of $N \in \w$ such that
\begin{itemize}
\item[$\circ$] If $n \geq N$ then there is exactly one $a \in \A$ such that $n \in X_a$, and 
\item[$\circ$] The following are equivalent for any $a,b \in \A$:
\begin{itemize} 
\item[$\cdot$] $n \in X_a$ and $n+1 \in X_b$ for some $n \geq N$
\item[$\cdot$] $n \in X_a$ and $n+1 \in X_b$ for infinitely many $n \geq N$
\item[$\cdot$] $a \tosi b$. 
\end{itemize}
\end{itemize}
By modifying each of the $X_a$ below $N$ (which does not change the equality $[X_a]=a$), we may (and do) assume that these two properties hold with $N = 0$. 
(More precisely, to see how these modifications can be accomplished, fix a length-$N$ walk $\seq{a_i}{i \leq N}$ in $\< \A,\tosi \>$ with $a_N = [X_{a_N}]$. 
Some such walk exists because $\< \A,\tosi \>$ is strongly connected, by Theorem~\ref{thm:PartitionsOfPWMFF}. 
Then modify the $X_a$ by putting $i \in X_a$ if and only if $a = a_i$ for all $i \leq N$.)

For each $k \in \w$, let $a_k$ denote the unique $a \in \A$ such that $k \in X_a$. 
For each $n \in \w$, let 
$v_n = \phi^{-1}(a_n)$. 
We claim $\seq{v_n}{n \in \w}$ is a diligent walk through $\<\V,\to\>$, and $\phi$ is the natural isomorphism associated to this walk. 

To see that this is a diligent walk, first note that $a_n \tosi a_{n+1}$ for all $n$ (because $n \in X_a$ and $n+1 \in X_b$ is equivalent to $a \tosi b$). 
Because $\phi$ is an isomorphism, this implies $\phi^{-1}(a_n) \to \phi^{-1}(a_{n+1})$ for all $n$. 
Thus $\seq{v_n}{n \in \w}$ is an infinite walk. 
Furthermore, if $v,w \in \V$ with $v \to w$, then $\phi(v) \tosi \phi(w)$, which means there are infinitely many $n$ with $n \in X_{\phi(v)}$ and $n+1 \in X_{\phi(w)}$. Thus there are infinitely many $n$ such that $v_n = v$ and $v_{n+1} = w$. 
Thus $\seq{v_n}{n \in \w}$ is a diligent walk.  

Finally, we claim that $\phi$ is the natural isomorphism associated to this diligent walk. 
Our definition of the walk implies that for all $v \in \V$, 
$$A_v \,=\, [\set{n}{v_n = v}] \,=\, [\set{n}{a_{n} = \phi(v)}] \,=\, [X_{\phi(v)}] \,=\, \phi(v).$$
Thus $A_v = \phi(v)$ for all $v \in \V$.
\end{proof}

Recall that $[X]^2$ denotes the set of all unordered pairs of elements of $X$. 

\begin{lemma}\label{lem:Rainbow}
Suppose $\chi: [\w]^2 \to \w$, i.e., $\chi$ is an $\w$-coloring of the complete graph $K_\w$, and suppose that all monochromatic cliques for $\chi$ are finite. Then 
there is an increasing sequence $\seq{n_k}{k \in \w}$ of natural numbers such that $\chi(n_k,n_{k+1}) > k$ for all $k$.
\end{lemma}
\begin{proof}
Fix $\chi: [\w]^2 \to \w$, an $\w$-coloring  of $K_\w$, and assume that all monochromatic cliques are finite. 
Now suppose $X \sub \w$ is infinite. 
If only finitely many colors are taken by $\chi$ in the graph $[X]^2$, then Ramsey's Theorem (the infinitary version) implies that there is an infinite monochromatic clique for $\chi$ in $[X]^2$, contradicting our assumption about $\chi$. 
Thus if $X \sub \w$ is infinite, then the image of $[X]^2$ under $\chi$ is infinite.

We now construct an $\w$-sequence of infinite subsets of $\w$ by recursion. 
Let $X_0 = \w$. 
Given $X_n$, define a $2$-coloring $\chi_n$ of $[X_n]^2$ by setting
$$\chi_n(a,b) \,=\, \begin{cases}
0 \ \text{ if } \chi(a,b) < n+1, \\ 
1 \ \text{ if } \chi(a,b) \geq n+1.
\end{cases}$$
By the previous paragraph, $[X_n]^2$ does not have any infinite monochromatic cliques with color $0$. But by Ramsey's Theorem, $[X_n]^2$ does have an infinite monochromatic clique. Let $X_{n+1}$ be any infinite subset of $X_n$ such that $\chi_n(a,b) \geq n+1$ for all $a,b \in X_{n+1}$. This completes the recursion.

For each $k \in \w$, let $n_k \in X_k \setminus \max \{n_0,n_1,\dots,n_{k-1}\}$. 
Then $n_k,n_{k+1} \in X_k$, which implies $\chi(n_k,n_{k+1}) > k$, for all $k$.
\end{proof}

\begin{lemma}\label{lem:ImageOfAPartition}
Let $\< \AA,\a \>$ and $\< \BB,\b \>$ be dynamical systems, and let $\A$ be a partition of $\AA$. 
If $\eta$ is an embedding from $\< \AA,\a \>$ into $\< \BB,\b \>$, then $\eta$ is an isomorphism from $\< \A,\toa \>$ to $\< \eta[\A],\tob \>$.
\end{lemma}
\begin{proof}
Clearly $\eta$ is a bijection $\A \to \eta[\A]$. 
Furthermore, because $\eta$ is an embedding, given any $a,b \in \A$, we have $a \toa b$ if and only if $\eta(a) \tob \eta(b)$. 
It follows that $\eta$ is a digraph isomorphism from $\< \A,\toa \>$ to $\< \eta[\A],\tob \>$.
\end{proof}

Before stating the next theorem, which is the main result of this section, let us establish a new (abuse of) notation. 
Suppose $\BB$ and $\CC$ are Boolean algebras, $\B$ and $\C$ are partitions of $\BB$ and $\CC$, respectively, and $\eta$ is a bijection $\B \to \C$. Suppose also that $a$ is in the (finite) subalgebra of $\BB$ generated by $\B$, i.e., $a = \bigvee \set{b \in \B}{b \leq a}$. In this case, we define $\eta(a) = \bigvee \set{\eta(b)}{b \in \B \text{ and } b \leq a}$. In other words, a bijection $\B \to \C$ extends naturally to an isomorphism from the finite subalgebra of $\BB$ generated by $\B$ to the finite subalgebra of $\CC$ generated by $\C$, and we abuse notation by denoting this extension also by $\eta$. This abuse occurs in the statement of $(\dagger)$ in the following theorem when speaking of the map $\tilde \eta \circ \iota$. In general we may not have $\iota(a) \in \B$ (although we do have $\iota(a) = \bigvee \set{b \in \B}{b \leq \iota(a)}$), so $\tilde \eta(\iota(a))$ is only well-defined via this notational abuse. 

\begin{theorem}\label{thm:Dagger}
Suppose $(\< \AA,\a \>,\< \BB,\b \>,\iota,\eta)$ is an instance of the lifting problem for $\< \pwmff,\s \>$. 
This instance of the lifting problem has a solution if and only if the following condition holds:
\begin{itemize}
\item[$(\dagger)$] For every partition $\A$ of $\AA$ and every partition $\B$ of $\BB$ that refines $\iota[\A]$, there is a partition $\C$ of $\pwmff$ refining $\eta[\A]$, and an isomorphism $\tilde \eta$ from $\< \B,\tob \>$ to $\< \C,\tosi \>$ such that $\tilde \eta \circ \iota = \eta$
\end{itemize}
\end{theorem}

\begin{center}
\begin{tikzpicture}[scale=1.2]

\node at (-.05,0) {$\<\A,\toa\>$};
\node at (-.05,2) {$\<\B,\tob\>$};
\node at (3.15,2) {$\<\C,\tosi\>$};

\draw[dashed,->] (.65,2) -- (2.4,2);
\draw[->] (.55,.35) -- (2.5,1.65);
\draw[<-] (-.05,1.65) -- (-.05,.35);

\node at (-.2,1) {\footnotesize $\iota$};
\node at (1.82,.85) {\footnotesize $\eta$};
\node at (1.5,2.24) {\footnotesize $\tilde \eta$};

\end{tikzpicture}
\end{center}

\noindent Roughly, this theorem asserts that $(\< \AA,\a \>,\< \BB,\b \>,\iota,\eta)$ has a solution if and only if every finitary fragment $(\< \A,\toa \>,\< \B,\tob \>,\iota,\eta)$ of it has a solution.

\begin{proof}
The ``only if'' direction of the theorem is relatively easy. Suppose that $\bar \eta$ is a solution to $(\< \AA,\a \>,\< \BB,\b \>,\iota,\eta)$. 
To show that $(\dagger)$ holds, fix a partition $\A$ of $\AA$ and a partition $\B$ of $\BB$ that refines $\iota[\A]$. 
Let $\C = \bar \eta[\B]$. Taking $\tilde \eta = \bar \eta$ gives an isomorphism from $\< \B,\tob \>$ to $\< \C,\tosi \>$ by Lemma~\ref{lem:ImageOfAPartition}, and because $\bar \eta$ is a lifting of $\eta$, we have $\tilde \eta \circ \iota = \bar \eta \circ \iota = \eta$. 

For the ``if'' direction, suppose $(\dagger)$ holds. We aim to find a solution $\bar \eta$ to $(\< \AA,\a \>,\< \BB,\b \>,\iota,\eta)$. The plan is to obtain $\bar \eta$ by finding an eventually small, eventually dense sequence in $\BB$ and then applying Theorem~\ref{thm:CharacterizingGoodSequences}.

Let $\seq{\A_n}{n \in \w}$ be a sequence of partitions of $\AA$ satisfying the conclusion of Lemma~\ref{lem:NicePartitionSequence}, and likewise let $\seq{\B_n}{n \in \w}$ be a sequence of partitions of $\BB$ satisfying the conclusion of Lemma~\ref{lem:NicePartitionSequence}. 
Furthermore, we may assume for convenience that $\B_0 = \{\mathbf{1}\}$ and, by thinning out the sequence $\seq{\B_n}{n \in \w}$ if necessary, we may (and do) assume that $\B_n$ refines $\iota[\A_n]$ for all $n$. 

For each $n$, property $(\dagger)$ gives us a partition $\C_n$ of $\pwmff$, with $\C_n \supseteq \eta[\A_n]$, and an isomorphism $\tilde \eta_n$ from $\< \B_n,\tob \>$ to $\< \C_n,\tosi \>$ such that $\tilde \eta_n \circ \iota = \eta$. 
Applying Lemma~\ref{lem:Walks2}, for each $n \in \w$ there is a diligent walk $\seq{b^n_i}{i \in \w}$ through $\< \B_n,\tob \>$ such that $\tilde \eta_n$ is the natural isomorphism associated with the walk $\seq{b^n_i}{i \in \w}$.

The plan of the proof is to diagonalize across the sequences $\seq{b^n_i}{i \in \w}$ in order to obtain a single sequence $\seq{d_n}{n \in \w}$ in $\BB$. This diagonal sequence will be eventually small and eventually dense, so that it induces an embedding $\BB \to \pwmff$, and the embedding it induces will be the desired solution $\bar \eta$ to the lifting problem. 

Not all the sequences $\seq{b^n_i}{i \in \w}$ are used in our diagonalization, but only a selection of them. 
In order to decide which sequences to use, we employ Lemma~\ref{lem:Rainbow}. 
Recall that if $\B$ is a partition of $\BB$ and $b,b' \in \BB$, then $b \approx_\B b'$ means $b \wedge u \neq \mathbf{0}$ if and only if $b' \wedge u \neq \mathbf{0}$ for all $u \in \B$. 
Define a function $\chi: [\w]^2 \to \w$ as follows. 
For every $k,\ell \in \w$ with $k < \ell$, define
\begin{align*}
\chi(k,\ell) \,=\, \max \set{j \leq k}{ \b(b^k_n) \approx_{\B_j} b^\ell_{n+1} \text{ for infinitely many }n}.
\end{align*}
Our assumption that $\B_0 = \{\mathbf{1}\}$ means the set in this definition is nonempty, as it contains $0$; hence $\chi$ is well-defined. 

Roughly, the idea motivating the definition of $\chi$ is as follows. 
In our diagonal sequence $\seq{d_i}{i \in \w}$ that we are planning to construct, we will have long stretches of indices $i$ where $d_i = b^k_i$, i.e., where the diagonal sequence just copies $\seq{b^k_i}{i \in \w}$, for some $k$. 
Then at some point we stop copying $\seq{b^k_i}{i \in \w}$ and ``jump'' to copying another sequence $\seq{b^\ell_i}{i \in \w}$ with $\ell > k$. 
In particular, for some $i \in \w$ (the jump time) our diagonal sequence will have $d_i = b^k_i$ and then $d_{i+1} = b^\ell_{i+1}$. 
In the end, Theorem~\ref{thm:CharacterizingGoodSequences} requires that we have $\seq{\b(d_i)}{i \in \w} \approx \seq{d_{i+1}}{i \in \w}$. 
Thus we need $\b(d_i) \approx_{\B_j} d_{i+1}$ for larger and larger $j$ as we go along. 
In some sense, the $j$ in this expression is measuring the degree of error in the jump: the bigger the $j$, the less the error between $\b(d_i)$ and $d_{i+1}$. 
Because we cannot control these jump times very well in the diagonalization, we would like to have infinitely many opportunities to jump from one sequence to another with small error (large $j$). 
The number $\chi(k,\ell)$ tells us what error we can expect to see when jumping from $\seq{b^k_i}{i \in \w}$ to $\seq{b^\ell_i}{i \in \w}$. 

In the parlance of Ramsey theory, $\chi$ is a coloring $[\w]^2 \to \w$. 
We claim that all the monochromatic cliques for this coloring are finite. 
To see this, fix $j \in \w$, and suppose $C \sub \w$ is a clique with color $j$, i.e., $\chi(k,\ell) = j$ for all $k,\ell \in C$. 
Let $m$ be large enough that $\B_{m}$ refines the partition generated by $\B_{j+1} \cup \set{\b^{-1}(b)}{b \in \B_{j+1}}$. 
We claim $|C| \leq |\B_m|+m+1$.
Aiming for a contradiction, suppose that $|C| > |\B_m|+m+1$.
For every $b \in \BB$ that is below $\B_m$, let $\pi_m(b)$ denote the (unique) member of $\B_m$ with $b \leq \pi_m(b)$. 
Fix $k_0,k_1,\dots,k_{|\B_m|} \in C$ such that $m \leq k_0 < k_1 < \dots < k_{|\B_m|}$. 
(Note that some such members of $C$ exist by our assumption $|C| > |\B_m|+m+1$.) 
For any given $k \geq m$ and $n \in \w$, $b^k_n$ is below $\B_m$, so that $\pi_m(b^k_n)$ is well-defined, and there are only $|\B_m|$ possible values for $\pi_m(b^k_n)$. 
In particular, by the pigeonhole principle, for every $n \in \w$ at least two of 
$$\pi_m(b^{k_0}_n),\, \pi_m(b^{k_1}_n),\, \pi_m(b^{k_2}_n),\, \dots,\, \pi_m(b^{k_{|\B_m|}}_n)$$
are the same. More precisely, for each $n \in \w$ there is a pair $i_n,i'_n$ with $i_n < i'_n \leq |\B_m|$ such that $\pi_m(b^{k_{i_n}}_n) = \pi_m(b^{k_{i'_n}}_n)$. 
Applying the pigeonhole principle again, there are some particular $i < i' \leq |\B_m|$ such that $i = i_n$ and $i' = i'_n$ for infinitely many $n$. 
But $\pi_m(b^{k_{i'}}_n) \tob \pi_m(b^{k_{i'}}_{n+1})$ for any $n$ (because $b^{k_{i'}}_n \tob b^{k_{i'}}_{n+1}$), 
and and therefore $\pi_m(b^{k_i}_n) \tob \pi_m(b^{k_{i'}}_{n+1})$ whenever $\pi_m(b^{k_i}_{n}) = \pi_m(b^{k_{i'}}_{n})$. 
Hence $\pi_m(b^{k_i}_n) \tob \pi_m(b^{k_{i'}}_{n+1})$ for infinitely many $n$. 
Now recall that $m \leq k_i < k_{i'}$, which means all members of the sequences $\seq{b^{k_i}_n}{n \in \w}$ and $\big \langle b^{k_{i'}}_n :\, n \in \w \big \rangle$ are below $\B_m$. 
By our choice of $m$, if $b,b'$ are below $\B_m$ then both $\b(b)$ and $b'$ are below $\B_{j+1}$; and if $b \tob b'$, then this means $\b(b)$ and $b'$ are below the same member of $\B_{j+1}$ (that is, $\b(b),b' \leq u$ for some $u \in \B_{j+1}$), which implies $\b(b) \approx_{\B_{j+1}} b'$. 
In particular, because $\pi_m(b^{k_i}_n) \tob \pi_m(b^{k_{i'}}_{n+1})$ for infinitely many $n$, this means $\b(b^{k_i}_n) \approx_{\B_{j+1}} b^{k_{i'}}_{n+1}$ for infinitely many $n \in \w$. Thus $\chi(k_i,k_{i'}) \geq j+1$. Contradiction!

As all the monochromatic cliques for $\chi$ are finite, Lemma~\ref{lem:Rainbow} gives us an increasing sequence $\seq{n_k}{k \in \w}$ such that $\chi(n_k,n_{k+1}) > k$ for all $k$.

Next we define a sequence $J_0,J_1,J_2,\dots$ of integers (our jump times) by recursion. 
For each $a \in \AA$, fix some $X_a \sub \w$ such that $\eta(a) = [X_a]$. 
Let $J_0 = 0$, and given some $J_k$, $k \in \w$, choose $J_{k+1}$ such that
\begin{itemize}
\item[$\circ$] for every $b \in \B_{n_k}$, there is some $i \in [J_k,J_{k+1})$ such that $b^{n_k}_i = b$,
\item[$\circ$] $\b \big( b^{n_k}_{J_k-1} \big) \approx_{\B_k} b^{n_{k+1}}_{J_k}$, and 
\item[$\circ$] if $a$ is a member of the finite subalgebra of $\AA$ generated by $\A_{n_{k+1}}$, then $X_a \setminus J_{k+1} = \set{i}{b^{n_{k+1}}_i \leq \iota(a)} \setminus J_{k+1}$.
\end{itemize}
To see that some such $J_{k+1}$ exists: the first condition is satisfied by all sufficiently large choices of $J_{k+1}$, because the sequence $\seq{b^{n_k}_i}{i \in \w}$ is a diligent walk and therefore contains each member of $\B_{n_k}$ infinitely many times; the second condition is satisfied infinitely often, because $\chi(n_k,n_{k+1}) > k$; the third condition is satisfied by all sufficiently large choices of $J_{k+1}$ because, by the fact that $\tilde \eta_{n_{k+1}} \circ \iota = \eta$ (and by our notational abuse that extends $\tilde \eta_{n_{k+1}}$ to the subalgebra of $\AA$ generated by $\A_{n_{k+1}}$),
$$[\{ i :\, b^{n_{k+1}}_i \leq \iota(a) \}] \,=\, \tilde \eta_{n_{k+1}}(\iota(a)) \,=\, (\tilde \eta_{n_{k+1}} \circ \iota)(a) \,=\, \eta(a) \,=\, [X_a]$$
for every $a$ in the finite subalgebra of $\AA$ generated by $\A_{n_{k+1}}$.

Define the diagonal sequence $\seq{d_n}{n \in \w}$ in $\BB$ by setting 
$$d_i \,=\, b^{n_k}_i \quad \text{ whenever $J_k \leq i < J_{k+1}$}.$$
The sequence $\seq{d_i}{i \in \w}$ is eventually small in $\BB$ by our choice of the $\B_n$'s, and because the $n_k$ are increasing and each $b^{n_k}_i$ is in $\B_{n_k}$.
The sequence is eventually dense by the condition that for every $b \in \B_{n_k}$ there is some $i \in [J_k,J_{k+1})$ such that $b^{n_k}_i = b$. By our choice of the $\B_n$'s, this implies that for any given $x \in \BB$, if $k$ is sufficiently large then some $b \in \B_{n_k}$ has $b \leq x$, and so there is some $d_i = b \leq x$ with $J_k \leq i < J_{k+1}$. 

Furthermore, we claim $\seq{\b(d_i)}{i \in \w} \approx \seq{d_{i+1}}{i \in \w}$. 
Let $\V$ be a partition of $\BB$, and fix $K$ large enough that $\B_K$ refines the partition generated by $\V \cup \set{\b^{-1}(b)}{b \in \V}$. 
If $i \neq J_\ell-1$ for any $\ell$, then there are some $k,j$ such that $d_i = b^k_j$ and $d_{i+1} = b^k_{j+1}$. 
But $\< b^k_j :\, j \in \w \>$ is is a diligent walk through $\< \B_k,\tob \>$, 
so $b^k_j \tob b^k_{j+1}$. 
If $k \geq K$, then $\b(b^k_j)$ and $b^k_{j+1}$ are both below $\V$, and because $b^k_j \tob b^k_{j+1}$, they are both below the same member of $\V$. 
Thus $\b(d_i) = \b(b^k_j) \approx_{\V} b^k_{j+1} = d_{i+1}$. 
Hence $\b(d_i) \approx_{\V} d_{i+1}$ for sufficiently large $i$, provided that $i \neq J_\ell-1$ for any $\ell$. 
On the other hand, suppose $i = J_\ell-1$ for some $\ell \geq K$. Then $\b(d_i) \approx_{\B_\ell} d_{i+1}$ by the second condition listed above in our definition of the $J_k$. 
Because $\ell \geq K$, $\B_\ell$ refines $\V$, so $\b(d_i) \approx_{\B_\ell} d_{i+1}$ implies $\b(d_i) \approx_{\V} d_{i+1}$. 
Thus, in either case, $\b(d_i) \approx_{\V} d_{i+1}$ for all sufficiently large $i$. 
Because $\V$ was an arbitrary partition of $\BB$, this shows $\seq{\b(d_i)}{i \in \w} \approx \seq{d_{i+1}}{i \in \w}$ as claimed. 

Applying Theorem~\ref{thm:CharacterizingGoodSequences}, $\seq{d_i}{i \in \w}$ induces an embedding of $\< \BB,\b \>$ into $\< \pwmff,\s \>$. Call this embedding $\bar \eta$. We claim that $\bar \eta$ is a lifting of $\eta$. 

To see this, fix $a \in \AA$. There is some $K$ large enough that $a$ is in the finite subalgebra of $\AA$ generated by $\A_k$ for all $k \geq K$. 
Given $k \geq K$, $\B_k$ refines $\iota[\A_k]$, and thus $\iota(a) = \bigvee \set{d \in \B_k}{d \leq \iota(a)}$. 
Combining this with the third bullet point in our choice of the $J_k$, if $J_k \leq i < J_{k+1}$ then 
$d_i = b^{n_k}_i \leq \iota(a)$ if and only if $i \in X_a$. Consequently, 
$$(\bar \eta \circ \iota)(a) \,=\, \bar \eta(\iota(a)) \,=\, [\set{i}{d_i \leq \iota(a)}] \,=\, [X_a] \,=\, \eta(a).$$
As $a$ was arbitrary, it follows that $\eta = \bar \eta \circ \iota$ as claimed.
\end{proof}

Of course one could prove an analogous result concerning the lifting problem for $\< \pwmff,\s^{-1} \>$. 
But there is no need. 
The purpose of Theorem~\ref{thm:Dagger} is to bring us one step closer to proving the Lifting Lemma. We already proved that a ``reverse'' Lifting Lemma (Lemma~\ref{lem:niam}) follows from the original version, so there is no need for a reversed version of Theorem~\ref{thm:Dagger}. 

The rest of this section is not part of the proof of Lemma~\ref{lem:main}, and the reader who wishes to may skip it without losing the thread of the argument. 

The main theorem of \cite{Brian1}, translated from the topological to the algebraic category, states that every incompressible dynamical system of size $\leq\!\aleph_1$ embeds in $\< \pwmff,\s \>$ and in $\< \pwmff,\s^{-1} \>$ (the conclusion of the vacuously true Theorem~\ref{thm:AU2}.) 
To end this section, we derive this theorem as a relatively straightforward consequence of Theorem~\ref{thm:Dagger}. 

The proof shows how elementarity can be used, in combination with Theorem~\ref{thm:Dagger}, to solve certain instances of the lifting problem, just as in Lemma~\ref{lem:main}. 
The difference is that the following lemma requires $\iota$ to be elementary, whereas Lemma~\ref{lem:main} places the burden on $\eta$ instead.

\begin{lemma}\label{lem:IotaElementary}
Suppose $(\< \AA,\a \>,\< \BB,\b \>,\iota,\eta)$ is an instance of the lifting problem for $\< \pwmff,\s \>$. 
If $\iota$ is an elementary embedding of $\< \AA,\a \>$ into $\< \BB,\b \>$, then condition $(\dagger)$ from Theorem~\ref{thm:Dagger} is satisfied, and this instance of the lifting problem has a solution.
\end{lemma}
\begin{proof}
Suppose $\iota$ is an elementary embedding from $\< \AA,\a \>$ into $\< \BB,\b \>$. 
Let $\A$ be a partition of $\AA$ and let $\B$ be a partition of $\BB$ that refines $\iota[\A]$. 
We aim to show that condition $(\dagger)$ from Theorem~\ref{thm:Dagger} holds.
 
Let $\varphi$ be the first-order sentence about $\< \BB,\b \>$ obtained as the conjunction of all true assertions of the following form:
\begin{itemize}
\item[$\circ$] $\bigvee \B = \mathbf{1}$ and $b \wedge b' = \mathbf{0}$ whenever $b,b' \in \B$ and $b \neq b'$,
\item[$\circ$] $\bigvee \iota[\A] = \mathbf{1}$ and $b \wedge b' = \mathbf{0}$ whenever $b,b' \in \iota[\A]$ and $b \neq b'$,
\item[$\circ$] $b \tob b'$ for some $b,b' \in \B$ or some $b,b' \in \iota[\A]$,
\item[$\circ$] $\neg(b \tob b')$ for some $b,b' \in \B$ or some $b,b' \in \iota[\A]$,
\item[$\circ$] $b \leq b'$ for some $b \in \B$ and $b' \in \iota[\A]$.
\end{itemize}
In other words, $\varphi$ asserts that $\iota[\A]$ and $\B$ are partitions of $\BB$, it describes precisely the digraphs $\< \B,\tob \>$ and $\< \iota[\A],\tob \>$, and it records which members of $\B$ are below which members of $\iota[\A]$. 

Fix an enumeration $\A = \{a_1,a_2,\dots,a_m\}$ of $\A$ and an enumeration $\B = \{b_1,b_2,\dots,b_n\}$ of $\B$. 
Let $\psi(\iota(a_1),\dots,\iota(a_m),x_1,\dots,x_n,\b)$ be the formula about $\< \BB,\b \>$ obtained from $\varphi$ by replacing every instance of $b_i$ with the variable $x_i$, and let $\bar \psi$ be the formula
$$\exists x_1 \exists x_2 \dots \exists x_n \psi(\iota(a_1),\dots,\iota(a_m),x_1,\dots,x_n,\b).$$
This is a true assertion about $\< \BB,\b \>$, because $\psi(\iota(a_1),\dots,\iota(a_m),b_1,\dots,b_n,\b)$ is true.

Now consider the formula  
$$\exists x_1 \exists x_2 \dots \exists x_n \psi(a_1,\dots,a_n,x_1,\dots,x_n,\a)$$
about $\< \AA,\a \>$ obtained from $\bar \psi$ by replacing every instance of $\iota(a_i)$ with $a_i$ and every instance of $\b$ with $\a$. 
Because $\iota$ is an elementary emedding, this formula must be true in $\< \AA,\a \>$, as the corresponding formula (the original $\bar \psi$) is true in $\< \BB,\b \>$. 
Hence there exist some $\tilde b_1, \tilde b_2, \dots, \tilde b_n \in \AA$ such that 
$$\< \AA,\a \> \models \psi(a_1,\dots,a_m,\tilde b_1, \dots, \tilde b_n,\a).$$
Define $\tilde \eta: \B \to \pwmff$ by setting $\tilde \eta(b_i) = \eta(\tilde b_i)$ for all $i \leq n$.

Let $\tilde \B = \{ \tilde b_i :\, i \leq n \}$ and let $\C = \tilde \eta[\tilde \B]$.
Observe that $\tilde \B$ is a partition of $\AA$, because 
$\psi(a_1,\dots,a_m,\tilde b_1, \dots, \tilde b_n,\a)$ asserts that it is. 
Similarly, $\< \tilde \B, \toa \>$ is isomorphic to $\< \B,\tob \>$, and $\tilde \B$ refines $\A$, and the natural map $\tilde \B \to \A$ matches the natural map $\B \to \iota[\A]$ (in the sense that the former sends $\tilde b_i$ to $a_j$ if and only if the latter sends $b_i$ to $\iota(a_j)$). 
Because $\tilde \B$ is a partition of $\AA$, $\C = \eta[\tilde \B]$ is a partition of $\pwmff$, and $\eta$ is an isomorphism from $\< \tilde B,\toa \>$ to $\< \eta[\tilde \B],\tosi \>$ by Lemma~\ref{lem:ImageOfAPartition}. It follows that $\tilde \eta$ is an isomorphism from $\< \B,\tob \>$ to $\< \eta[\tilde \B],\tosi \>$. 
Furthermore, because the natural map $\tilde \B \to \A$ matches the natural map $\B \to \iota[\A]$, we have $\tilde \eta \circ \iota(a) = \eta(a)$ for all $a \in \A$. 

In other words, this map $\tilde \eta$ witnesses that $(\dagger)$ is satisfied for $\A$ and $\B$. 
As $\A$ and $\B$ were arbitrary partitions, $(\dagger)$ holds.
\end{proof}

\begin{theorem}\label{thm:BigEmbeddings}
Every incompressible dynamical system of size $\aleph_1$ embeds in $\< \pwmff,\s \>$ and in $\< \pwmff,\s^{-1} \>$. 
\end{theorem}
\begin{proof}
We prove the theorem first for $\< \pwmff,\s \>$. 
Every countable incompressible dynamical system embeds in $\< \pwmff,\s \>$ by Theorem~\ref{thm:CountableIncompressibleEmbeds}, so it suffices to consider systems of size $\aleph_1$. 

Let $\< \AA,\a \>$ be an incompressible dynamical system with $|\AA| = \aleph_1$. 
Applying the downward L\"{o}wenheim-Skolem theorem $\w_1$ times, there is an increasing sequence $\seq{\AA_\b}{\b < \w_1}$ of countable subalgebras of $\AA$ such that
\begin{itemize}
\item[$\circ$] $\< \AA_\b,\a \>$ is an elementary substructure of $\< \AA,\a \>$ for all $\b < \w_1$, 
\item[$\circ$] $\AA_\b = \bigcup_{\xi < \b} \AA_\xi$ whenever $\b$ is a limit ordinal, and 
\item[$\circ$] $\bigcup_{\b < \w_1} \AA_\b = \AA$.
\end{itemize} 

We now construct, via transfinite recursion, a sequence $\seq{\eta_\b}{\b < \w_1}$ of maps such that each $\eta_\b$ is an embedding from $\< \AA_\b,\a \>$ into $\< \pwmff,\s \>$, and the $\eta_\b$ are coherent in the sense that $\eta_\b = \eta_\g \rest \AA_\b$ whenever $\b < \g$.

To begin, fix an embedding $\eta_0$ from $\< \AA_0,\a \>$ into $\< \pwmff,\s \>$. (Some such embedding exists by Theorem~\ref{thm:CountableIncompressibleEmbeds}, and because elementarity implies $\< \AA_0,\a \>$ is incompressible.) At a limit stage $\b$ of the recursion, let $\eta_\b = \bigcup_{\xi < \b}\eta_\xi$. This is an embedding from $\< \AA_\b,\a \>$ into $\< \pwmff,\s \>$, because the $\eta_\xi$ are coherent embeddings and $\AA_\b = \bigcup_{\xi < \b}\AA_\xi$. 

For the successor step, suppose $\g = \b+1$ and suppose that we already have an embedding $\eta_\b$ from $\< \AA_\b,\a \>$ into $\< \pwmff,\s \>$. Let $\iota$ denote the inclusion map $\AA_\b \xhookrightarrow{} \AA_\g$. 
Then $(\<\AA_\b,\a\>,\<\AA_\g,\a\>,\iota,\eta_\b)$ is an instance of the lifting problem for $\< \pwmff,\s \>$, and furthermore $\iota$ is an elementary embedding. By Lemma~\ref{lem:IotaElementary}, there is an embedding $\eta_\g$ from $\< \AA_\g,\a \>$ into $\< \pwmff,\s \>$ such that $\eta_\g \circ \iota = \eta_\b$. Because $\iota$ is the inclusion map $\AA_\b \xhookrightarrow{} \AA_\g$, this last condition just means that $\eta_\b = \eta_\g \rest \AA_\b$, i.e., the coherence of the maps is preserved.

In the end, let $\eta = \bigcup_{\b < \w_1}\eta_\b$. Because the $\eta_\b$ are coherent embeddings and $\AA = \bigcup_{\b < \w_1}\AA_\b$, $\eta$ is an embedding from $\< \AA,\a \>$ to $\< \pwmff,\s \>$. 

It remains to show that we can also obtain an embedding from $\< \AA,\a \>$ to $\< \pwmff,\s^{-1} \>$. 
Of course one could repeat the above proof, \emph{mutatis mutandis}, but this is not necessary. 
Recall that $\< \AA,\a \>$ is incompressible if and only if $\< \AA,\a^{-1} \>$ is. Thus, by the preceding argument, there is an embedding $\eta'$ from $\< \AA,\a^{-1} \>$ into $\< \pwmff,\s \>$. 
But then the same map $\eta'$ is also an embedding from $\< \AA,\a \>$ into $\< \pwmff,\s^{-1} \>$.
\end{proof}

\begin{corollary}\label{cor:BigEmbeddings}
Assuming the Continuum Hypothesis, $\< \pwmff,\s \>$ embeds in $\< \pwmff,\s^{-1} \>$ and vice versa. 
\end{corollary}
\begin{proof}
This follows from Theorem~\ref{thm:BigEmbeddings}, because \ch implies $\card{\pwmff} = \aleph_1$, and both $\< \pwmff,\s \>$ and $\< \pwmff,\s^{-1} \>$ are incompressible.
\end{proof}

As noted in the introduction, the conclusion of this corollary is independent of \zfc: assuming \ocama, there is no embedding from $\< \pwmff,\s \>$ into $\< \pwmff,\s^{-1} \>$ or from $\< \pwmff,\s^{-1} \>$ into $\< \pwmff,\s \>$. This was proved in \cite{Brian1} as a relatively straightforward corollary to the extensive work of Farah in \cite{FarahA} on how \ocama restricts the continuous self-maps of $\w^*$.

\section{The Lifting Lemma, part 4: polarization}\label{sec:Elementarity}

The Lifting Lemma guarantees that an instance $(\<\AA,\a\>,\<\BB,\b\>,\iota,\eta)$ of the lifting problem has a solution, assuming the image of $\< \AA,\a \>$ under $\eta$ is an elementary substructure of $\< \pwmff,\s \>$. Note that, unlike Lemma~\ref{lem:IotaElementary}, this puts no conditions whatsoever on $\iota$, only on $\eta$ and $\< \AA,\a \>$. 

In the previous section we formulated a condition $(\dagger)$ that is both necessary and sufficient for the existence of a solution to an instance of the lifting problem. 
In this section we use $(\dagger)$ to provide a sufficient condition for a solution. 
Unlike $(\dagger)$, this condition is a property of $\< \AA,\a \>$ only, called ``polarization'', and in particular it makes no mention of $\iota$. 

After this section, our plan for completing the proof of Lemma~\ref{lem:main} is as follows. In Section 10 we prove a combinatorial statement more or less equivalent to the fact that $\< \pwmff,\s \>$ is polarized. Then in Section 11, we show that polarization passes to elementary substructures, because it is expressible in a sufficiently first-order way. 
This means that the hypotheses of the Lifting Lemma imply $\< \AA,\a \>$ is polarized, which, by the results proved in this section, suffices to get a solution to the lifting problem.

\begin{definition}
Let $\< \AA,\a \>$ be an incompressible dynamical system and let $\A$ be a partition of $\AA$. A \emph{virtual refinement} of $\< \A,\toa \>$ is a pair $(\phi,\< \V,\to \>)$, where $\< \V,\to \>$ is a strongly connected digraph and $\phi$ is an epimorphism from $\< \V,\to \>$ to $\< \A,\toa \>$. 
A partition $\A'$ of $\AA$ \emph{realizes} $(\phi,\< \V,\to \>)$ if $\A'$ refines $\A$ and there is an isomorphism $\psi$ from $\< \V,\to \>$ to $\< \A',\toa \>$ such that $\pi \circ \psi = \phi$, where $\pi$ is the natural map from $\A'$ to $\A$.
\hfill{\Coffeecup}
\end{definition}

The following lemma says virtual refinements generalize actual refinements. 
This lemma has Lemma~\ref{lem:Epimorphism} as a special case (by taking $\iota = \mathrm{id}_\AA$).

\begin{lemma}\label{lem:Natural}
Suppose $\< \AA,\a \>$ and $\< \BB,\b \>$ are dynamical systems, and $\iota$ is an embedding of $\< \AA,\a \>$ into $\< \BB,\b \>$. Given a partition $\A$ of $\AA$ and a partition $\B$ of $\BB$ that refines $\iota[\A]$, define the map
$$\phi(b) \,=\, \text{the unique member of $\A$ with } b \leq \iota(a).$$
Then $(\phi,\< \B,\tob \>)$ is a virtual refinement of $\< \A,\toa \>$.
\end{lemma}
\begin{proof}
Note that $\phi$ is well-defined, because $\B$ refines $\iota[\A]$, and is clearly a surjection $\B \to \A$. 
If $a,a' \in \A$ and $a \toa a'$, then $\iota(a) \tob \iota(a')$ (because $\iota$ is an embedding), and this implies that $b \tob b'$ for some $b \leq \iota(a)$ and $b' \leq \iota(a')$. In other words, there are $b,b' \in \B$ such that $\phi(b) = a$, $\phi(b') = a'$, and $b \tob b'$. 
On the other hand, it is clear that $b \tob b'$ implies $\phi(b) \toa \phi(b')$.
\end{proof}

\noindent Henceforth, if the conditions of this lemma hold then the function
$$\phi(b) \,=\, \text{the unique member of $\A$ with } b \leq \iota(a)$$
is called the \emph{natural epimorphism} from $\< \B,\tob \>$ onto $\< \A,\toa \>$.

\begin{definition}\label{def:Incompatible}
Let $\< \A,\toA \>$, $\< \U,\toU \>$, and 
$\< \V,\toV \>$ be strongly connected digraphs, and suppose that
$\phi$ and $\psi$ are epimorphisms from $\< \U,\toU \>$ onto $\< \A,\toA \>$ and from $\< \V,\toV \>$ onto $\< \A,\toA \>$, respectively. 
These two epimorphisms are \emph{compatible} if there is a strongly connected digraph
$\< \W,\toW \>$, and there are epimorphisms $\bar \phi$ and $\bar \psi$ from 
$\< \W,\toW \>$ onto 
$\< \U,\toU \>$ and $\< \V,\toV \>$, respectively, such that $\phi \circ \bar \phi = \psi \circ \bar \psi$.
\end{definition}

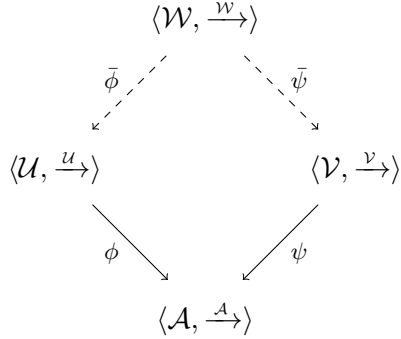
\begin{center}
\begin{tikzpicture}[scale=1]

\node at (0,0) {$\< \A,\toA \>$};
\node at (-2,2) {$\< \U,\toU \>$};
\node at (2,2) {$\< \V,\toV \>$};
\node at (0,4) {$\< \W,\toW \>$};

\draw[<-] (.5,.5)--(1.5,1.5);
\draw[<-] (-.5,.5)--(-1.5,1.5);
\draw[<-,dashed] (1.5,2.5)--(.5,3.5);
\draw[<-,dashed] (-1.5,2.5)--(-.5,3.5);

\node at (1.25,.85) {\footnotesize $\psi$};
\node at (-1.25,.85) {\footnotesize $\phi$};
\node at (1.25,3.15) {\footnotesize $\bar \psi$};
\node at (-1.25,3.15) {\footnotesize $\bar \phi$};

\end{tikzpicture}
\end{center}

\noindent Otherwise $\phi$ and $\psi$ are \emph{incompatible}. 
By extension, if $\< \AA,\a \>$ is a dynamical system and $\A$ is a partition of $\AA$, then two virtual refinements $(\phi,\<\U,\toU\>)$ and $(\psi,\<\V,\toV\>)$ of $\< \A,\toa \>$ are (in)compatible if the epimorphisms $\phi$ and $\psi$ are (in)compatible.
\hfill {\Coffeecup} 

\vspace{.5mm}

Note that this definition requires $\< \W,\toW \>$ to be strongly connected. 
In fact, without the requirement that $\< \W,\toW \>$ be strongly connected, any two digraph epimorphisms would be compatible. The reason is that the ``pullback'' of $\phi$ and $\psi$ always makes the above diagram commute, although the pullback of two strongly connected digraphs may not be strongly connected.
(The \emph{pullback} of $\phi$ and $\psi$ is the digraph $\<\W,\toW\>$ where
$$\W \,=\, \set{(u,v) \in \U \times \V}{\phi(u) = \psi(v)}, \quad \text{and}$$
$$(u,v) \toW (u',v') \ \text{ if and only if } \ u \toU u' \text{ and }v \toV v'.$$
The \emph{standard projections} from $\<\W,\toW\>$ to $\<\U,\toU\>$ and $\<\V,\toV\>$, respectively, are the maps  $\bar \phi(u,v) = u$ and $\bar \psi(u,v) = v$. It is not difficult to see that these maps are epimorphisms, and that $\phi \circ \bar \phi = \psi \circ \bar \psi$. For further details about the pullback of digraphs, also called the fiber product, see \cite{BNW} or \cite{CN}.)
%

In light of this, one should expect that any proof of the incompatibility of two epimorphisms must rely in some essential way on the strong connectedness condition for $\< \W,\toW \>$.  
For an example of incompatible epimorphisms, consider the following maps reminiscent of the proof of Theorem~\ref{thm:NoGo}:

\vspace{2mm}

\begin{center}
\begin{tikzpicture}[xscale=.46,yscale=.5]

\draw[->,thick] (7,-2) -- (10,-5);
\draw[->,thick] (18.5,-2) -- (15.5,-5);


\draw[densely dotted] (0,0) ellipse (8mm and 11mm);
\draw[densely dotted] (5,0) ellipse (9mm and 13mm);
\draw[densely dotted] (10,0) ellipse (8mm and 11mm);

\node at (0,0) {\tiny $\bullet$};
\node at (5,-.5) {\tiny $\bullet$};
\node at (5,.5) {\tiny $\bullet$};
\node at (10,0) {\tiny $\bullet$};
\draw[<->] (9.85,.1) -- (5.15,.475);
\draw[<->] (.15,0) -- (4.85,-.475);
\draw[<->] (5.15,-.475) -- (9.85,-.1);

\draw[->] (-.12,.14) arc (35:325:2.2mm);
\draw[->] (4.88,.39) arc (-35:-325:2.2mm);
\draw[->] (4.87,-.62) arc (125:415:2.2mm);
\draw[->] (10.12,-.1) arc (-145:145:2.2mm);

\begin{scope}[shift={(7.75,-7)}]

\draw (0,0) ellipse (8mm and 11mm);
\draw (5,0) ellipse (9mm and 13mm);
\draw (10,0) ellipse (8mm and 11mm);

\draw [->] (1,.2) to [out=10,in=170] (3.9,.2);
\draw [<-] (1,-.2) to [out=-10,in=-170] (3.9,-.2);
\draw [->] (6.1,.2) to [out=10,in=170] (9,.2);
\draw [<-] (6.1,-.2) to [out=-10,in=-170] (9,-.2);

\draw[->] (-1,.27) arc (50:310:4mm);
\draw[->] (11,-.27) arc (230:490:4mm);
\draw[->] (4.73,-1.45) arc (140:400:4mm);
\end{scope}

\begin{scope}[shift={(15.5,0)}]

\draw[densely dotted] (0,0) ellipse (8mm and 11mm);
\draw[densely dotted] (5,0) ellipse (9mm and 13mm);
\draw[densely dotted] (10,0) ellipse (8mm and 11mm);

\node at (0,0) {\tiny $\bullet$};
\node at (5,-.5) {\tiny $\bullet$};
\node at (5,.5) {\tiny $\bullet$};
\node at (10,0) {\tiny $\bullet$};
\draw[->] (.18,.11) -- (4.85,.475);
\draw[<-] (.15,-.11) -- (4.82,-.475);
\draw[<-] (5.15,-.475) -- (9.82,-.11);
\draw[->] (5.18,.475) -- (9.85,.11);

\draw[->] (-.12,.14) arc (35:325:2.2mm);
\draw[->] (4.87,-.62) arc (125:415:2.2mm);
\draw[->] (5.13,.62) arc (-415:-125:2.2mm);
\draw[->] (10.12,-.1) arc (-145:145:2.2mm);
\end{scope}

\end{tikzpicture}
\end{center}

\noindent Note that the pullback of these epimorphisms is not strongly connected: 

\vspace{2mm}
\begin{center}
\begin{tikzpicture}[xscale=.62,yscale=.72]

\draw[densely dotted] (0,0) ellipse (8mm and 11mm);
\draw[densely dotted] (5,0) ellipse (12mm and 16mm);
\draw[densely dotted] (10,0) ellipse (8mm and 11mm);

\node at (0,0) {\tiny $\bullet$};
\node at (5.5,-.7) {\tiny $\bullet$};
\node at (5.5,.7) {\tiny $\bullet$};
\node at (4.5,-.4) {\tiny $\bullet$};
\node at (4.5,.4) {\tiny $\bullet$};
\node at (10,0) {\tiny $\bullet$};

\draw[<-] (9.8,.16) -- (5.65,.695);
\draw[->] (9.8,-.16) -- (5.65,-.695);
\draw[<-] (9.8,.06) -- (4.65,.4);
\draw[->] (9.8,-.06) -- (4.65,-.4);

\draw[<-] (.15,-.08) -- (4.35,-.4);
\draw[->] (.15,.08) -- (4.35,.4);

\draw[->] (-.12,.14) arc (35:325:2.2mm);
\draw[->] (4.37,.52) arc (-125:-415:2.2mm);
\draw[->] (5.37,.82) arc (-125:-415:2.2mm);
\draw[->] (4.63,-.52) arc (55:-235:2.2mm);
\draw[->] (5.63,-.82) arc (55:-235:2.2mm);
\draw[->] (10.12,-.13) arc (215:505:2.2mm);

\end{tikzpicture}
\end{center}
\vspace{2mm}

\noindent We leave it as an exercise to show that these two digraph epimorphisms are incompatible. Interestingly, one way of doing this is strikingly similar to the latter part of the proof of Theorem~\ref{thm:NoGo}.

Although we do not provide a proof of the fact here (as it is not needed), let us remark that two epimorphisms of strongly connected digraphs are compatible if and only if their pullback is strongly connected. 

\begin{lemma}\label{lem:Compatability}
Suppose $\< \AA,\a \>$ and $\< \BB,\b \>$ are incompressible dynamical systems, and $\iota$ is an embedding of $\< \AA,\a \>$ into $\< \BB,\b \>$. Let $\A$ be a partition of $\AA$, and let $\B$ be a partition of $\BB$ refining $\iota[\A]$. 
Then for any partition $\A'$ of $\AA$ refining $\A$, the natural map from $\< \A',\toa \>$ to $\< \A,\toa \>$
and the natural epimorphism from $\< \B,\tob \>$ to $\< \A,\toa \>$ 
are compatible.
\end{lemma}
\begin{proof}
Let $\B'$ be any partition of $\BB$ refining both $\B$ and $\iota[\A']$. 
We claim that the digraph $\< \B',\tob \>$, along with the appropriate natural epimorphisms, witnesses the compatibility of the two virtual refinements in question. 
In short, the proof works because the composition of two natural maps is again a natural map (with the appropriate domain and co-domain). 

More precisely, let $\phi$ denote the natural map from $\A'$ to $\A$, and 
let $\psi$ denote the natural epimorphism from $\< \B,\tob \>$ onto $\< \A,\toa \>$. 
Let $\bar \phi$ denote the natural epimorphism from $\< \B',\tob \>$ onto $\< \A',\toa \>$ and 
let $\bar \psi$ denote the natural map from $\B'$ to $\B$. 
The definitions of these various maps give us that both $\phi \circ \bar \phi$ and $\psi \circ \bar \psi$ are equal to the function $\B' \to \A$ given by
$$b \,\mapsto\, \text{the unique member of $\A$ with } b \leq \iota(a),$$
or in other words, the natural epimorphism from $\< \B',\tob \>$ onto $\< \A,\toa \>$. 
\end{proof}

In other words, this lemma states that a virtual refinement of $\< \A,\toa \>$ that comes from an embedding (in the way that $\< \B,\tob \>$ does in the statement of the lemma) is always compatible with actual refinements of $\A$. 


\begin{definition}\label{def:Polarized}
Let $\< \AA,\a \>$ be an incompressible dynamical system, and let $\A$ be a partition of $\AA$. 
A virtual refinement $(\phi,\<\V,\to\>)$ of $\<\A,\toa\>$ is \emph{polarized} if either
\begin{enumerate}
\item there is a partition $\A'$ of $\AA$ that realizes $(\phi,\<\V,\to\>)$, or
\item there is a partition $\A'$ of $\AA$ refining $\A$ such that $(\phi,\< \V,\to \>)$ and $(\pi,\< \A',\toa \>)$ are incompatible, where $\pi$ denotes the natural projection from $\A'$ onto $\A$. 
\end{enumerate} 
By extension, the dynamical system $\< \AA,\a \>$ is \emph{polarized} if every virtual refinement of every partition of $\AA$ polarized. 
\hfill{\Coffeecup}
\end{definition}

For example, finite dynamical systems are not polarized. To see this, suppose $\AA$ is finite and $\< \AA,\a \>$ is an incompressible dynamical system. Then $\AA$ has a maximal partition $\A$ (the set of atoms of $\AA$). If $(\phi,\<\V,\to\>)$ is any virtual refinement of $\< \A,\toa \>$ with $|\V| > |\A|$, then both options from Definition~\ref{def:Polarized} are clearly impossible. 


For a less trivial example of a non-polarized dynamical system, consider the system $\< \AA,\a \>$ from the proof of Theorem~\ref{thm:NoGo}. 
The fact that $\< \AA,\a \>$ is not polarized follows from our next theorem. 

\begin{theorem}\label{thm:FinalStroke}
Suppose $(\< \AA,\a \>,\< \BB,\b \>,\iota,\eta)$ is an instance of the lifting problem for $\< \pwmff,\s \>$. 
If $\<\AA,\a\>$ is polarized, then this instance of the lifting problem has a solution.
\end{theorem}
\begin{proof}
In order to prove that $(\< \AA,\a \>,\< \BB,\b \>,\iota,\eta)$ has a solution, it suffices to prove that condition $(\dagger)$ from Theorem~\ref{thm:Dagger} holds.
Let $\A$ be a partition of $\AA$, and let $\B$ be a partition of $\B$ that refines $\iota[\A]$. 

Let $\phi$ denote the natural epimorphism from $\< \B,\tob \>$ onto $\< \A,\toa \>$. By Lemma~\ref{lem:Natural},  
$(\phi,\< \B,\tob \>)$ is a virtual refinement of $\< \A,\toa \>$. 

Let $\tilde \A$ be any partition of $\AA$ refining $\A$, and let $\tilde \pi$ denote the natural map $\tilde \A \to \A$. Then $(\tilde \pi,\< \tilde \A,\toa \>)$ is a virtual refinement of $\< \A,\toa \>$ by Lemma~\ref{lem:Epimorphism}, and this virtual refinement is compatible with $(\phi,\<\B,\tob\>)$, by Lemma~\ref{lem:Compatability}. 
Consequently, case $(2)$ of Definition~\ref{def:Polarized} is not true of $\A$ and $(\phi,\<\B,\tob\>)$. 
Because $\A$ is a polarized partition, this means that case $(1)$ must be true: there is a partition $\A'$ of $\AA$ that realizes $(\phi,\<\B,\tob\>)$. 
Letting $\pi$ denote the natural map from $\A'$ onto $\A$, this means that there is an isomorphism $\phi'$ from $\< \B,\tob \>$ onto $\< \A',\toa \>$ such that $\pi \circ \phi' = \phi$.

The remainder of the argument imitates the proof of Lemma~\ref{lem:IotaElementary}, but with the members of $\A'$ playing the part that the $\tilde b_i$ played in that proof. 

Define $\tilde \eta(b) = \eta \circ \phi'(b)$ for all $b \in \B$, and let $\C = \tilde \eta[\B]$. 
Because $\phi'$ is a digraph isomorphism from $\< \B,\tob \>$ to $\< \A',\toa \>$ and $\eta$ is a digraph isomorphism from $\< \A',\toa \>$ to $\< \C,\tosi \>$ (by Lemma~\ref{lem:ImageOfAPartition}), $\tilde \eta$ is an isomorphism from $\< \B,\tob \>$ to $\< \C,\tosi \>$. 
Also, $\C = \tilde \eta[\B] = \eta[\phi'[\B]] = \eta[\A']$; because $\A'$ refines $\A$, this implies that $\C$ refines $\eta[\A]$. 
It remains to check that $\eta = \tilde \eta \circ \iota$.

Let $a \in \A$. Recall that $\iota(a)$ may not be in $\B$ (the domain of $\tilde \eta$), but extending (or abusing) our notation in the natural way, we define 
$\tilde \eta \circ \iota(a) = \bigvee \set{\tilde \eta(b)}{b \in \B \text{ and }b \leq \iota(a)}$. 
Similarly, (with the same notational abuse) we define $\phi' \circ \iota(a) = \bigvee \set{\phi'(b)}{b \in \B \text{ and }b \leq \iota(a)}$. 
Because $\pi \circ \phi' = \phi$ and $\pi(\phi'(b)) = a$ if and only if $\phi'(b) \leq a$, 
\begin{align*}
\phi' \circ \iota(a) &\,=\, \textstyle \bigvee \set{\phi'(b)}{b \in \B \text{ and }b \leq \iota(a)} \\
& \textstyle \,=\, \bigvee \set{\phi'(b)}{b \in \B \text{ and } \phi'(b) \leq a} \\
& \textstyle \,=\, \bigvee \set{c}{c \in \A' \text{ and } c \leq a} \,=\, a.
\end{align*}
Consequently, $\eta(a) = \eta \circ \phi' \circ \iota(a) = \tilde \eta \circ \iota(a).$
As $a$ was an arbitrary member of $\A$, this shows that $\eta = \tilde \eta \circ \iota$ as required.
\end{proof}

\section{The Lifting Lemma, part 5: the heart of the lemma}\label{sec:heart}

In this section, we prove a graph-theoretic result that is the combinatorial heart and soul of the Lifting Lemma, and indeed, of the entire paper. 

\begin{definition}
Let $\< \A,\toA \>$ and $\< \B,\toB \>$ be strongly connected digraphs, and let $\phi$ be a function $\B \to \A$ (for instance, an epimorphism). 

A length-$k$ walk $\<b_0,b_1,\dots,b_k\>$ in $\< \B,\toB \>$ \emph{projects} onto a length-$k$ walk $\<a_0,a_1,\dots,a_k\>$ in $\< \A,\toA \>$ (via $\phi$) if $\phi(b_i) = a_i$ for all $i \leq k$. 
In the same way, an infinite walk $\seq{b_n}{n \in \w}$ in $\< \B,\toB \>$ \emph{projects} (via $\phi$) onto an infinite walk $\seq{a_n}{n \in \w}$ in $\< \A,\toA \>$ if $\phi(b_n) = a_n$ for all $n$. 

Likewise, a walk $\seq{b_n}{n \in \w}$ \emph{almost projects} (via $\phi$) onto $\seq{a_n}{n \in \w}$ if $\phi(b_n) = a_n$ for all sufficiently large values of $n$.
\hfill{\Coffeecup}
\end{definition}

The main theorem of this section states: 

\begin{theorem}\label{thm:Heart&Soul}
Let $\< \A,\toA \>$ and $\< \B,\toB \>$ be strongly connected digraphs, and let $\phi$ be an epimorphism from $\< \B,\toB \>$ onto $\< \A,\toA \>$. 
If $\seq{a_n}{n \in \w}$ is a diligent walk through $\< \A,\toA \>$, then either
\begin{enumerate}
\item there is a diligent walk through $\< \B,\toB \>$ that almost projects onto $\seq{a_n}{n \in \w}$, or
\item there is a strongly connected digraph $\< \C,\toC \>$, and an epimorphism $\psi$ from $\< \C,\toC \>$ to $\< \A,\toA \>$, and a diligent walk through $\< \C,\toC \>$ that almost projects onto $\seq{a_n}{n \in \w}$, such that $\psi$ and $\phi$ are incompatible.
\end{enumerate}
\end{theorem}

\noindent These two alternatives are mutually exclusive (though not obviously so), and they correspond roughly to the two alternatives in Definition~\ref{def:Polarized}. 

The proof of Theorem~\ref{thm:Heart&Soul} naturally divides itself into two major stages. In the first stage, we prove a weaker version of the theorem, where the ``diligent walk'' from alternative $(1)$ is replaced by the weaker ``infinite walk'' instead. 
This is stated as Theorem~\ref{thm:DigraphPathLifting} below. 
The second stage of the proof extends this weaker version of the theorem to the full version. 

\begin{lemma}\label{lem:EpisForFree}
Suppose that $\< \A,\toA \>$ is a strongly connected digraph, and let $\seq{a_n}{n \in \w}$ be a diligent walk through $\< \A,\toA \>$. 
Suppose that $\< \C,\toC \>$ is a digraph with no isolated points, and let $\psi$ be a function $\C \to \A$. 
If there is a diligent walk $\seq{c_n}{n \in \w}$ through $\< \C,\toC \>$ that almost projects via $\psi$ onto $\seq{a_n}{n \in \w}$, then $\< \C,\toC \>$ is strongly connected and $\psi$ is an epimorphism.
\end{lemma}
\begin{proof}
Suppose $\seq{c_n}{n \in \w}$ is a diligent walk through $\< \C,\toC \>$ that almost projects onto $\seq{a_n}{n \in \w}$. The existence of a diligent walk through $\< \C,\toC \>$ implies $\< \C,\toC \>$ is strongly connected, by Lemma~\ref{lem:Walks0}. To prove the lemma, we must show $\psi$ is an epimorphism. 
Fix some $N \in \w$ sufficiently large that if $n \geq N$ then $\psi(c_n) = a_n$. 

Let $a \in \A$. Because $\seq{a_n}{n \in \w}$ is diligent, there are infinitely many $n \in \w$ such that $a_n = a$; in particular, $a_n = a$ for some $n \geq N$. Hence $\psi(c_n) = a_n = a$, and as $a$ was arbitrary, it follows that $\psi$ is surjective. 

Fix $c,c' \in \C$ such that $c \toC c'$. Because $\seq{c_n}{n \in \w}$ is diligent, there is some $n \geq N$ such that $c = c_n$ and $c' = c_{n+1}$. Because $n \geq N$, this implies $\psi(c) = a_n \toA a_{n+1} = \psi(c')$. 

Lastly, fix $a,a' \in \A$ such that $a \toA a'$. Because $\seq{a_n}{n \in \w}$ is diligent, there is some $n \geq N$ such that $a = a_n$ and $a' = a_{n+1}$. Because $n \geq N$, this implies $\psi(c_n) = a$ and $\psi(c_{n+1}) = a'$. But $c_n \toC c_{n+1}$ (because $\seq{c_n}{n \in \w}$ is a walk), so this shows there are $c,c' \in \C$ such that $\psi(c) = a$, $\psi(c') = a'$, and $c \toC c'$.
\end{proof}

\begin{lemma}\label{lem:Konig}
Let $\< \A,\toA \>$ and $\< \B,\toB \>$ be strongly connected digraphs, let $\phi$ be an epimorphism from $\< \B,\toB \>$ to $\< \A,\toA \>$, and suppose $\seq{a_n}{n \in \w}$ is an infinite walk in $\< \A,\toA \>$. 
There is an infinite walk in $\< \B,\toB \>$ that projects onto $\seq{a_n}{n \in \w}$ if and only if for every $k \in \w$ there is a length-$k$ walk $\<b_0,b_1,\dots,b_k\>$ in $\< \B,\toB \>$ that projects onto $\<a_0,a_1,\dots,a_k\>$.
\end{lemma}
\begin{proof}
The ``only if'' direction is obvious: if $\seq{b_n}{n \in \w}$ is an infinite walk in $\< \B,\toB \>$ that projects onto $\seq{a_n}{n \in \w}$, then $\<b_0,b_1,\dots,b_k\>$ projects onto $\<a_0,a_1,\dots,a_k\>$ for every $k \in \w$. 

The ``if'' direction is a consequence of K\H{o}nig's Tree Lemma. To see this, consider the set $\mathcal T$ of all finite walks $\<b_0,b_1,\dots,b_k\>$ in $\< \B,\toB \>$ that project onto the initial segment $\<a_0,a_1,\dots,a_k\>$ of $\seq{a_n}{n \in \w}$ with the same length. 
When ordered in the natural way, by end-extension, this set $\mathcal T$ is a tree. Because $\B$ is finite, this tree is finitely branching (specifically, each member of $\mathcal T$ has at most $|\B|$ immediate successors). 

Now suppose that for every $k \in \w$ there is a length-$k$ walk $\<b_0,b_1,\dots,b_k\>$ in $\< \B,\toB \>$ that projects onto $\<a_0,a_1,\dots,a_k\>$. 
This implies $\mathcal T$ is infinite. 
Consequently, K\H{o}nig's Tree Lemma implies that $\mathcal T$ has an infinite branch 
$\big\< \0,\<b_0\>,\<b_0,b_1\>,\<b_0,b_1,b_2\>,\dots \big\>.$ 
The union of this branch is an infinite sequence $\<b_0,b_1,b_2,\dots\>$ in $\B$, and this sequence is an infinite walk in $\< \B,\toB \>$ that projects onto $\seq{a_n}{n \in \w}$.
\end{proof}

\begin{definition}\label{def:StateSpace}
Let $\< \A,\toA \>$ and $\< \B,\toB \>$ be strongly connected digraphs, and let $\phi$ be an epimorphism from $\< \B,\toB \>$ to $\< \A,\toA \>$. 
The \emph{state space} of $\phi$ is a digraph $\< \St^\phi,\toSp \>$, where 
$$\St^\phi \,=\, \set{\fS \sub \B}{\fS \sub \phi^{-1}(a) \text{ for some }a \in \A},$$
and given $\fS_0,\fS_1 \in \St^\phi$, we define $\fS_0 \toSp \fS_1$ if and only if there is some edge $(a_0,a_1)$ in $\toA$ such that $\fS_0 \sub \phi^{-1}(a_0)$ and $\fS_1 \sub \phi^{-1}(a_1)$, and 
$$\fS_1 \,=\, \set{y \in \phi^{-1}(a_1)}{\exists x \in \fS_0 \text{ such that } x \toB y}.$$ 
In this case we say that the edge $(a_0,a_1)$ in $\toA$ \emph{induces} the edge $(\fS_0,\fS_1)$ in $\toSp$. 
The \emph{natural map} $\Pi_\phi: \St^\phi \setminus \{\0\} \to \A$ sends each $\fS \in \St^\phi$ to the unique $a \in \A$ such that $\fS \sub \phi^{-1}(a)$.
\hfill{\Coffeecup}
\end{definition}
%

Drawn below are two examples of the state space $\< \St^\phi,\toSp \>$ associated to an epimorphism $\phi$ from some $\< \B,\toB \>$ onto some $\< \A,\toA \>$. 
To clarify the construction, we have labelled the members of $\toA$, and labelled each member of $\toSp$ with the member of $\toA$ that induces it (except for the arrow $\0 \toSp \0$, which is induced by every edge in $\toA$). 

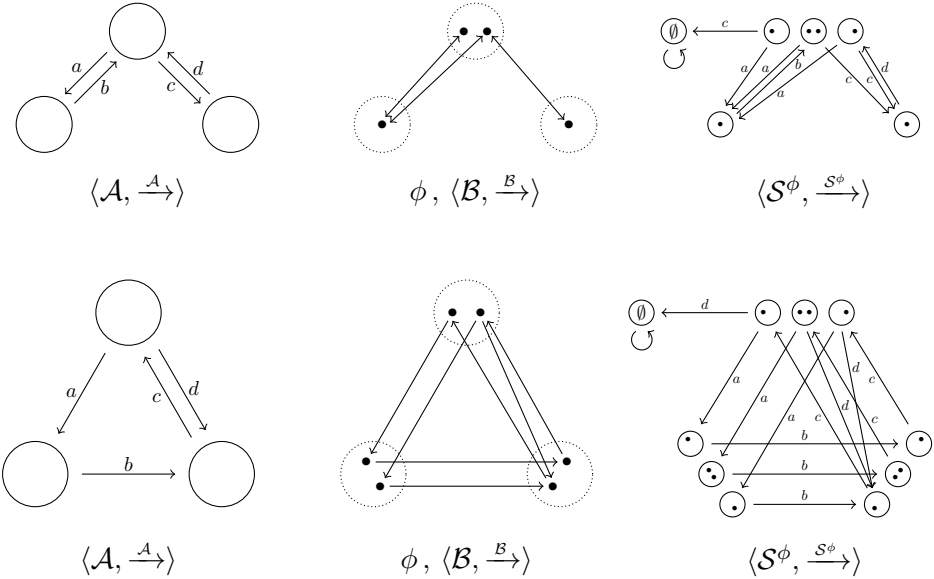
\vspace{4mm}

\begin{center}
\begin{tikzpicture}[xscale=.62,yscale=.62]

\draw (0,0) circle (6mm);
\draw (4,0) circle (6mm);
\draw (2,2) circle (6mm);

\node at (3.3,1.22) {\scalebox{.7}{$d$}};
\node at (2.72,.8) {\scalebox{.7}{$c$}};
\node at (.7,1.22) {\scalebox{.7}{$a$}};
\node at (1.3,.8) {\scalebox{.7}{$b$}};

\draw[<-] (3.35,.45) -- (2.45,1.35);
\draw[->] (3.55,.65) -- (2.65,1.55);
\draw[->] (.65,.45) -- (1.55,1.35);
\draw[<-] (.45,.65) -- (1.35,1.55);

\node at (2,-1.4) {$\< \A,\toA \>$};
\node at (9.25,-1.4) {$\phi\,,\,\< \B,\toB \>$};
\node at (16.5,-1.4) {$\< \St^\phi,\toSp \>$};

\begin{scope}[shift={(7.25,0)}]
\draw[densely dotted] (0,0) circle (6mm);
\draw[densely dotted] (4,0) circle (6mm);
\draw[densely dotted] (2,2) circle (6mm);

\node at (0,0) {\tiny $\bullet$};
\node at (4,0) {\tiny $\bullet$};
\node at (1.75,2) {\tiny $\bullet$};
\node at (2.25,2) {\tiny $\bullet$};

\draw[<->] (.07,.16) -- (1.65,1.9);
\draw[<->] (.17,.04) -- (2.15,1.9);
\draw[<->] (3.9,.1) -- (2.35,1.9);
\end{scope}

\begin{scope}[shift={(14.5,0)}]

\draw (2,2) circle (2.7mm);
\node at (2.1,2) {\scalebox{.4}{$\bullet$}};
\node at (1.9,2) {\scalebox{.4}{$\bullet$}};
\draw (1.21,2) circle (2.7mm);
\node at (1.11,2) {\scalebox{.4}{$\bullet$}};
\draw (2.79,2) circle (2.7mm);
\node at (2.89,2) {\scalebox{.4}{$\bullet$}};
\draw (0,0) circle (2.7mm);
\node at (0,0) {\scalebox{.4}{$\bullet$}};
\draw (4,0) circle (2.7mm);
\node at (4,0) {\scalebox{.4}{$\bullet$}};

\draw (-1,2) circle (2.7mm);
\node at (-1,2) {\scalebox{.65}{$\emptyset$}};

\draw[->] (3.83,.43)--(3.07,1.68);
\draw[<-] (3.73,.33)--(2.97,1.58);
\draw[<-] (3.61,.24)--(2.25,1.65);
\draw[->] (.35,.23)--(1.8,1.6);
\draw[<-] (.25,.33)--(1.7,1.7);
\draw[<-] (.12,.4)--(1,1.65);
\draw[<-] (.4,.12)--(2.5,1.7);
\draw[<-] (-.59,2)--(.8,2);

\node at (.1,2.15) {\scalebox{.5}{$c$}};
\node at (.52,1.2) {\scalebox{.5}{$a$}};
\node at (.97,1.2) {\scalebox{.5}{$a$}};
\node at (1.3,.65) {\scalebox{.5}{$a$}};
\node at (1.67,1.3) {\scalebox{.5}{$b$}};
\node at (2.75,.95) {\scalebox{.5}{$c$}};
\node at (3.2,.95) {\scalebox{.5}{$c$}};
\node at (3.53,1.23) {\scalebox{.5}{$d$}};

\draw[->] (-1.12,1.6) arc (-235:55:2.2mm);

\end{scope}

\end{tikzpicture}
\end{center}

\vspace{6mm}

\begin{center}
\begin{tikzpicture}[xscale=.62,yscale=.62]

\draw (0,0) circle (7mm);
\draw (4,0) circle (7mm);
\draw (2,3.464) circle (7mm);

\node at (.78,1.74) {\scalebox{.7}{$a$}};
\node at (3.4,1.84) {\scalebox{.7}{$d$}};
\node at (2.6,1.64) {\scalebox{.7}{$c$}};
\node at (2,.22) {\scalebox{.7}{$b$}};

\draw[->] (1,0) -- (3,0);
\draw[<-] (.5,.866) -- (1.5,2.598);
\draw[<-] (3.65,.946) -- (2.65,2.678);
\draw[->] (3.35,.786) -- (2.35,2.518);

\node at (2,-1.8) {$\< \A,\toA \>$};
\node at (9.25,-1.8) {$\phi\,,\,\< \B,\toB \>$};
\node at (16.5,-1.8) {$\< \St^\phi,\toSp \>$};

\begin{scope}[shift={(7.25,0)}]
\draw[densely dotted] (0,0) circle (7mm);
\draw[densely dotted] (4,0) circle (7mm);
\draw[densely dotted] (2,3.464) circle (7mm);

\node at (.15,-.26) {\tiny $\bullet$};
\node at (-.15,.26) {\tiny $\bullet$};
\node at (3.85,-.26) {\tiny $\bullet$};
\node at (4.15,.26) {\tiny $\bullet$};
\node at (1.7,3.464) {\tiny $\bullet$};
\node at (2.3,3.464) {\tiny $\bullet$};

\draw[->] (.35,-.26) -- (3.65,-.26);
\draw[->] (.05,.26) -- (3.95,.26);
\draw[<-] (-.05,.433) -- (1.6,3.291);
\draw[<-] (.25,-.087) -- (2.2,3.291);
\draw[->] (3.7,-.087) -- (1.7,3.291);
\draw[->] (4.05,.433) -- (2.45,3.291);
\draw[<-] (3.82,-.08) -- (2.34,3.28);
\end{scope}

\begin{scope}[shift={(14.5,0)}]
\draw (0,0) circle (2.7mm);
\node at (.055,-.08) {\scalebox{.4}{$\bullet$}};
\node at (-.055,.08) {\scalebox{.4}{$\bullet$}};
\draw (-.45,.65) circle (2.7mm);
\node at (-.5,.74) {\scalebox{.4}{$\bullet$}};
\draw (.45,-.65) circle (2.7mm);
\node at (.5,-.74) {\scalebox{.4}{$\bullet$}};
\begin{scope}[xscale=-1,shift={(-4,0)}]
\draw (0,0) circle (2.7mm);
\node at (.055,-.08) {\scalebox{.4}{$\bullet$}};
\node at (-.055,.08) {\scalebox{.4}{$\bullet$}};
\draw (-.45,.65) circle (2.7mm);
\node at (-.5,.74) {\scalebox{.4}{$\bullet$}};
\draw (.45,-.65) circle (2.7mm);
\node at (.5,-.74) {\scalebox{.4}{$\bullet$}};
\end{scope}
\draw (2,3.464) circle (2.7mm);
\node at (2.1,3.464) {\scalebox{.4}{$\bullet$}};
\node at (1.9,3.464) {\scalebox{.4}{$\bullet$}};
\draw (1.21,3.464) circle (2.7mm);
\node at (1.11,3.464) {\scalebox{.4}{$\bullet$}};
\draw (2.79,3.464) circle (2.7mm);
\node at (2.89,3.464) {\scalebox{.4}{$\bullet$}};

\draw (-1.5,3.464) circle (2.7mm);
\node at (-1.5,3.464) {\scalebox{.65}{$\emptyset$}};

\draw[->] (.78,3.464)--(-1.07,3.464);

\draw[->] (.45,0)--(3.55,0);
\draw[->] (0,.65)--(4,.65);
\draw[->] (.9,-.65)--(3.1,-.65);

\draw[->] (1,3.06)--(-.22,1.07);
\draw[<-] (3,3.06)--(4.22,1.07);
\draw[->] (1.81,3.06)--(.22,.38);
\draw[<-] (2.19,3.06)--(3.78,.38);
\draw[->] (2.62,3.06)--(.66,-.3);
\draw[<-] (1.38,3.06)--(3.34,-.3);
\draw[->] (2.8,3.06)--(3.45,-.3);
\draw[->] (2.05,3.06)--(3.45,-.3);

\node at (-.14,3.64) {\scalebox{.5}{$d$}};
\node at (2,.2) {\scalebox{.5}{$b$}};
\node at (2,.85) {\scalebox{.5}{$b$}};
\node at (2,-.45) {\scalebox{.5}{$b$}};
\node at (1.72,1.21) {\scalebox{.5}{$a$}};
\node at (2.28,1.21) {\scalebox{.5}{$c$}};
\node at (1.13,1.62) {\scalebox{.5}{$a$}};
\node at (.54,2.03) {\scalebox{.5}{$a$}};
\node at (3.44,2.03) {\scalebox{.5}{$c$}};
\node at (3.5,1.17) {\scalebox{.5}{$c$}};
\node at (3.1,2.3) {\scalebox{.5}{$d$}};
\node at (2.87,1.48) {\scalebox{.5}{$d$}};

\draw[->] (-1.62,3.064) arc (-235:55:2.2mm);

\end{scope}

\end{tikzpicture}
\end{center}

\vspace{2mm}

The natural map $\Pi_\phi: \St^\phi \setminus \{\0\} \to \A$ is indicated by the positions of the vertices of $\St^\phi$ in the picture (although we have omitted the dotted circles as in the drawings of $\<\B,\toB\>$). 

Observe that $\Pi_\phi$ is an epimorphism from the digraph $\< \St^\phi \setminus \{\0\},\toSp \>$ to $\< \A,\toA \>$. 
Observe also that $\< \St^\phi,\toSp \>$ is not strongly connected, because there is no walk from the empty state $\0$ to any other vertex of $\St^\phi$. 
In fact, the sub-digraph $\< \St^\phi \setminus \{\0\},\toSp \>$ need not be strongly connected either: for instance, in the second example illustrated above, there is no path from a size-$1$ member of $\St^\phi$ to a size-$2$ member. 

\begin{definition}\label{def:AssociatedWalk}
Let $\< \A,\toA \>$ and $\< \B,\toB \>$ be strongly connected digraphs, and let $\phi$ be an epimorphism from $\< \B,\toB \>$ to $\< \A,\toA \>$. 
Given an infinite walk $\bar a = \seq{a_n}{n \in \w}$ in $\< \A,\toA \>$, we associate to $\bar a$ an infinite sequence $\seq{\fS_n^{\bar a}}{n \in \w}$ in $\St^\phi$, defined by taking
\begin{align*}
\fS^{\bar a}_k \,=\, \Big\{ b \in \B :\ \text{$b$ is the last member }&\text{of a length-$k$ walk in $\< \B,\toB \>$} \\
&\text{ that projects onto }\<a_1,a_2,\dots,a_k\> \Big\}
\end{align*}
for all $k \in \w$. 
Each $\fS^{\bar a}_k$ is well-defined (i.e., $\fS^{\bar a}_k \in \St^\phi$) because $\fS^{\bar a}_k \sub \phi^{-1}(a_k)$. 
This sequence is the \emph{state space sequence associated to $\bar a$}.
\hfill{\Coffeecup}
\end{definition}

To motivate Definition~\ref{def:AssociatedWalk}, suppose we have an epimorphism $\phi$ from some strongly connected digraph $\< \A,\toA \>$ onto some other strongly connected digraph $\< \B,\toB \>$. We wish to understand the following 

\vspace{1mm}

\noindent \emph{Question}: \hspace{1mm} Given an infinite walk $\seq{a_n}{n \in \w}$ in $\< \A,\toA \>$, is there an

\hspace{15.5mm} infinite walk in $\< \B,\toB \>$ that projects onto it?

\vspace{1mm}

\noindent By Lemma~\ref{lem:Konig}, this question is equivalent to asking whether there are arbitrarily long finite walks in $\< \B,\toB \>$ that project onto initial segments of $\bar a = \seq{a_n}{n \in \w}$. The state space sequence is simply a convenient way of keeping track of what such walks might look like: $\fS^{\bar a}_k$ tells us the possibilities for $b_k$ in such an infinite walk $\seq{b_n}{n \in \w}$ 
Or more precisely, it tells us the possibilities not yet ruled out by looking at $\<a_0,a_1,\dots,a_k\>$ (the ``past'' part of $\bar a$ at time $k$); some possibilities may also be ruled out by considering $\< a_k,a_{k+1},a_{k+2},\dots\>$ (the future part of $\bar a$), although $\fS^{\bar a}_k$ will not be able to detect it. 

According to Lemma~\ref{lem:Konig}, there is an infinite walk in $\< \B,\toB \>$ that projects onto a walk $\seq{a_n}{n \in \w}$ if and only if the state space sequence associated to $\seq{a_n}{n \in \w}$ never reaches the empty state $\0$. 

\begin{lemma}\label{lem:Recursion}
Let $\< \A,\toA \>$ and $\< \B,\toB \>$ be strongly connected digraphs, and let $\phi$ be an epimorphism from $\< \B,\toB \>$ to $\< \A,\toA \>$. 
If $\bar a = \seq{a_n}{n \in \w}$ is an infinite walk in $\< \A,\toA \>$, then the state space sequence associated to $\bar a$ satisfies the following recursion:
\begin{align*}
\fS^{\bar a}_0 &\,=\, \phi^{-1}(a_0), \\
\fS^{\bar a}_{k+1} &\,=\, \set{y \in \phi^{-1}(a_{k+1})}{\exists x \in \fS^{\bar a}_k \text{ such that } x \toB y}.
\end{align*}
\end{lemma}
\begin{proof}
It is clear that $\fS^{\bar a}_0 = \phi^{-1}(a_0)$, because $\<b\>$ is a length-$0$ walk projecting onto $\<a_0\>$ if and only if $\phi(b) = a_0$. 
For the general step of the recursion, observe that $\< b_0,b_1,\dots,b_k,b_{k+1} \>$ is a length-$(k+1)$ walk projecting onto $\< a_0,a_1,\dots,a_k,a_{k+1} \>$ if and only if $(1)$ $\< b_0,b_1,\dots,b_k \>$ is a length-$k$ walk projecting onto $\< a_0,a_1,\dots,a_k \>$, $(2)$ $b_k \toB b_{k+1}$, and $(3)$ $\phi(b_{k+1}) = a_{k+1}$. Thus $b \in \fS^{\bar a}_{k+1}$ if and only if $\phi(b) = a_{k+1}$ and there is some length-$k$ walk $\< b_0,b_1,\dots,b_k \>$ projecting onto $\< a_0,a_1,\dots,a_k \>$ such that $b_k \toB b$. 
\end{proof}

One immediate consequence of this lemma is that 
$\fS^{\bar a}_k \toSp \fS^{\bar a}_{k+1}$ for every $k$, and this edge in $\toSp$ is induced by the edge $(a_k,a_{k+1})$ in $\toA$. 
Consequently, $\seq{\fS^{\bar a}_n}{n \in \w}$ is an infinite walk through $\< \St^\phi,\toSp \>$. This walk may or may not include the empty state $\0$; if it does then it has nowhere else to go, and $\fS^{\bar a}_n = \0$ for all sufficiently large $n$. 

\vspace{2mm}

The following theorem is a weakening of Theorem~\ref{thm:Heart&Soul}, which has ``infinite walk''  instead of ``diligent walk" in alternative $(1)$. The proof of this theorem completes the first major stage of our proof of Theorem~\ref{thm:Heart&Soul}.

\begin{theorem}\label{thm:DigraphPathLifting}
Let $\< \A,\toA \>$ and $\< \B,\toB \>$ be strongly connected digraphs, and let $\phi$ be an epimorphism from $\< \B,\toB \>$ onto $\< \A,\toA \>$. 
If $\seq{a_n}{n \in \w}$ is a diligent walk through $\< \A,\toA \>$, then either
\begin{enumerate}
\item there is an infinite walk through $\< \B,\toB \>$ that almost projects onto $\seq{a_n}{n \in \w}$, or
\item there is a strongly connected digraph $\< \C,\toC \>$ and an epimorphism $\psi$ from $\< \C,\toC \>$ to $\< \A,\toA \>$ such that $\psi$ is incompatible with $\phi$, and there is a diligent walk through $\< \C,\toC \>$ that projects onto $\seq{a_n}{n \in \w}$.
\end{enumerate}
\end{theorem}

\begin{proof}
Let $\< \A,\toA \>$ and $\< \B,\toB \>$ be strongly connected digraphs, and let $\phi$ be an epimorphism from $\< \B,\toB \>$ onto $\< \A,\toA \>$. 
Fix an infinite walk $\bar a = \seq{a_n}{n \in \w}$ through $\< \A,\toA \>$. 
For each $m \in \w$, let $\bar a_m = \seq{a_{n+m}}{n \in \w}$ (i.e., $\bar a_m$ is the sequence $\bar a$ with the first $m$ entries removed).

To begin, suppose there is some $m \in \w$ such that $\fS^{\bar a_m}_k \neq \0$ for all $k$. 
Applying Lemma~\ref{lem:Konig}, this implies there is an infinite walk $\seq{b_n}{n \in \w}$ through $\<\B,\toB\>$ that projects onto $\bar a_m$. 
Using the fact that $\< \B,\toB \>$ is strongly connected, fix some length-$m$ walk $\< b'_0,b'_1,\dots,b'_m\>$ in $\< \B,\toB \>$ such that $b'_m = b_0$. Define a new walk $\< \tilde b_0,\tilde b_1,\tilde b_2,\dots \>$ by concatenating these two walks: 
$$\tilde b_i \,=\, \begin{cases}
b_i' & \text{ if } i \leq m, \\
b_{i-m} & \text{ if } i \geq m.
\end{cases}$$
Then $\< \tilde b_0,\tilde b_1,\tilde b_2,\dots \>$ is an infinite walk through $\< \B,\toB \>$ that almost projects onto $\bar a$ via $\phi$, because $\phi(b_n) = a_n$ whenever $n \geq m$. 

Therefore, if there is some $m \in \w$ such that $\fS^{\bar a_m}_k \neq \0$ for all $k$, then alternative $(1)$ from the statement of the theorem holds. 
To prove the theorem, it suffices to show that 
if this is not the case then alternative $(2)$ holds. 

Let us suppose, then, that for every $m \in \w$ there is some $k \in \w$ such that $\fS^{\bar a_m}_{k} = \0$. 
For each $m$, let $k_m$ denote the least number with this property, so that $\fS^{\bar a_m}_i \neq \0$ for all $i < k_m$, but $\fS^{\bar a_m}_{k_m} = \0$. 
Note that $k_m > 0$ for all $m$, because $\fS^{\bar a_m}_0 = \phi^{-1}(a_m) \neq \0$. 

Because $\St^\phi$ and $\A$ are both finite, there are some particular $\C_0 \sub \St^\phi$, and $\E_0 \sub \ \toSp$, and  $\fS_\text{in} \in \St^\phi$, and $a_\text{in},a_\text{end},a_\text{out} \in \A$, such that 
\begin{itemize}
\item[$\circ$] $\set{\fS^{\bar a_m}_i}{i \leq k_m} \,=\, \C_0$,
\item[$\circ$] $\set{(\fS^{\bar a_m}_i,\fS^{\bar a_m}_{i+1})}{i < k_m} \,=\, \E_0$,
\item[$\circ$] $\fS^{\bar a_m}_0 \,=\, \fS_\text{in}$,  
\item[$\circ$] $a_{m-1} \,=\, a_\text{in}$, 
\item[$\circ$] $a_{k_m} \,=\, a_\text{end}$,  
\item[$\circ$] $a_{k_m+1} \,=\, a_\text{out}$
\end{itemize}
for infinitely many $m > 0$. (We exclude $m=0$ from consideration so that $a_{m-1}$ is well-defined.) 
Let $D$ denote the set of all $m > 0$ satisfying the six properties listed above. 

We are now ready to describe the digraph $\< \C,\toC \>$ and the epimorphism $\psi$ that witness alternative $(2)$ from the statement of the theorem. 
First define $\C \,=\, \A \cup \C_0$. (We are assuming, without loss of generality, that $\A \cap \C_0 = \0$. This is true automatically unless the members of $\A$ are labeled in an unreasonable way.) 
Define the edge relation on $\C$ as follows:
\begin{align*}
x \toC y \quad \ \Longleftrightarrow \quad \ & x,y \in \A \text{ and } x \toA y \text{, or} \\
& x,y \in \C_0 \text{ and } (x,y) \in \E_0 \text{, or} \\
& x=a_\text{in} \text{ and } y = \fS_\text{in} \text{, or} \\
& x = \0 \text{ and } y = a_{\mathrm{out}}. 
\end{align*}

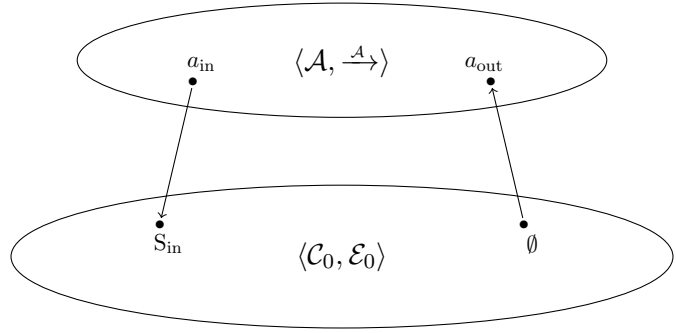
\vspace{1mm}

\begin{center}
\begin{tikzpicture}[xscale=1.1,yscale=.95]

\draw (0,2.75) ellipse (32mm and 8mm);
\node at (0,2.75) {$\< \A,\toA \>$};
\node at (-1.8,2.45) {\tiny $\bullet$};
\node at (-1.7,2.7) {\scalebox{.8}{$a_\text{in}$}};
\node at (1.8,2.45) {\tiny $\bullet$};
\node at (1.72,2.7) {\scalebox{.8}{$a_\text{out}$}};

\draw (0,0) ellipse (40mm and 10mm);
\node at (0,0) {$\< \C_0,\E_0 \>$};
\node at (-2.2,.45) {\tiny $\bullet$};
\node at (-2.1,.18) {\scalebox{.8}{$\fS_\text{in}$}};
\node at (2.2,.45) {\tiny $\bullet$};
\node at (2.3,.18) {\scalebox{.8}{$\0$}};

\draw[->] (-1.815,2.36)--(-2.185,.54);
\draw[<-] (1.815,2.36)--(2.185,.54);
\end{tikzpicture}
\end{center}

\vspace{1mm}

Define $\psi: \C \to \A$ as follows: 
$$\psi(x) = \begin{cases}
x & \text{ if } x \in \A, \\
\Pi_\phi(\fS) & \text{ if } \fS \in \C_0 \setminus \{\0\}, \\
a_\text{end} & \text{ if } x = \0.
\end{cases}$$

This completes the definition of $\< \C,\toC \>$ and $\psi$. To finish the proof, we need to show that $\< \C,\toC \>$ is strongly connected, that $\psi$ is an epimorphism from $\< \C,\toC \>$ to $\< \A,\toA \>$, that there is a diligent walk through $\< \C,\toC \>$ projecting onto $\seq{a_n}{n \in \w}$ via $\psi$, and that $\psi$ is incompatible with $\phi$. 

\vspace{2mm}

We can verify the first three of these four things all at once by applying Lemma~\ref{lem:EpisForFree}. To do this, we must define a diligent walk $\seq{c_i}{i \in \w}$ through $\< \C,\toC \>$ that projects onto $\seq{a_i}{i \in \w}$ via $\psi$. We define this walk via recursion, obtaining at each stage not a single $c_i$, but a long piece of the walk.

To begin, recall that for every edge $(a,a')$ in $\toA$, there are infinitely many $i \in \w$ such that $(a_i,a_{i+1}) = (a,a')$. Fix some $m_0 \in D$ large enough that for every $(a,a')$ in $\toA$, there is at least one $i+1 < m_0$ such that $(a_i,a_{i+1}) = (a,a')$. Then define 
\begin{align*}
c_i &\,=\, a_i \qquad \quad \ \ \, \text{for all } i < m_0 \text{, and} \\
c_i &\,=\, \fS^{\bar a_{m_0}}_{i-m_0} \qquad \text{whenever }  m_0 \leq i \leq m_0+k_{m_0}.
\end{align*}
This completes the base step of the recursion. 

Later stages of the recursion are similar. At stage $n$, suppose that there is some $m_{n-1} \in A$ such that $c_i$ is already defined for all $i \leq m_{n-1}+k_{m_{n-1}}$. 
Fix some $m_n \in D$ large enough that for every $(a,a') \in \ \toA$ there is some $i$ with $m_{n-1}+k_{m_{n-1}} \leq i < m_n-1$ such that $(a_i,a_{i+1}) = (a,a')$. Then define 
\begin{align*}
c_i &\,=\, a_i \qquad \quad \ \ \, \text{whenever } m_{n-1}+k_{m_{n-1}} \leq i < m_n \text{, and} \\
c_i &\,=\, \fS^{\bar a_{m_n}}_{i-m_n} \qquad \text{whenever }  m_n \leq i \leq m_n+k_{m_n}.
\end{align*}
This completes the recursion. 

Next, we check that $\seq{c_i}{i \in \w}$ is a diligent walk through $\< \C,\toC \>$ that projects onto $\seq{a_i}{i \in \w}$ via $\psi$. 

Every $i \in \w$ fits into one of the following four cases:
\begin{itemize}
\item[$\circ$] $c_i = a_i$ and $c_{i+1} = a_{i+1}$,
\item[$\circ$] $c_i = a_{m_n-1} = a_\text{in}$ and $c_{i+1} = \fS^{\bar a_{m_n}}_0 = \fS_\text{in}$ for some $m_n \in D$,
\item[$\circ$] $c_i = \fS^{\bar a_{m_n}}_j$ and $c_{i+1} = \fS^{\bar a_{m_n}}_{j+1}$ for some $m_n \in D$ and some $j < m_n$, or 
\item[$\circ$] $c_i = \0$ and $c_{i+1} = a_{m_n + k_{m_n}+1} = a_\text{out}$ for some $m_n \in D$. 
\end{itemize}
In any of these four cases, it follows immediately from the definition of $\toC$ that $c_i \toC c_{i+1}$. 
Hence $\seq{c_i}{i \in \w}$ is an infinite walk in $\< \C,\toC \>$. 

Furthermore, at stage $n$ of the recursion we chose the new $c_i$'s (that is, those $c_i$ for which $m_{n-1}+k_{m_{n-1}} \leq i \leq m_n+k_{m_n}$) so that: 
\begin{itemize}
\item[$\circ$] for every $(a,a')$ in $\toA$ there is some new $c_i$ with $(c_i,c_{i+1}) = (a,a')$, 
\item[$\circ$] for every $(\fS,\fS')$ in $\E_0$ there is some new $c_i$ with $(c_i,c_{i+1}) = (\fS,\fS')$,
\item[$\circ$] there is some new $c_i$ with $(c_i,c_{i+1}) = (a_\text{in},\fS_\text{in})$, and
\item[$\circ$] there is some new $c_i$ with $(c_i,c_{i+1}) = (\0,a_\text{out})$.
\end{itemize}
In other words, for every $(c,c')$ in $\toC$ there is some new $c_i$ with $(c_i,c_{i+1}) = (c,c')$. 
It follows that $\seq{c_i}{i \in \w}$ is a diligent walk through $\< \C,\toC \>$.

For each $i \in \w$, either $c_i \in \A$ or $c_i \in \C_0$. If $c_i \in \A$ then $\psi(c_i) = a_i$ by construction. 
If $c_i = \fS \in \C_0$ then either 
\begin{itemize}
\item[$\circ$] $c_i = \fS^{\bar a_{m_n}}_{i-m_n}$ for some $n$ with $m_n \leq i < m_n+k_{m_n}$, in which case we have $\fS^{\bar a_{m_n}}_{i-m_n} \neq \0$ and $\fS^{\bar a_{m_n}}_{i-m_n} \sub \phi^{-1}(a_{i-m_n})$, which implies that $\psi(c_i) = \Pi_\phi(\fS^{\bar a_{m_n}}_{i-m_n}) = a_i$; or 
\item[$\circ$] $c_i = \fS^{\bar a_{m_n}}_{k_{m_n}} = \0$ for some $m_n \in D$, in which case $\psi(c_i) = a_\text{end} = a_{m_n+k_{m_n}} = a_i$.
\end{itemize}
In either case, $\psi(c_i) = a_i$. This holds for all $i \in \w$. 

By Lemma~\ref{lem:EpisForFree}, $\< \C,\toC \>$ is strongly connected, $\psi$ is an epimorphism from $\< \C,\toC \>$ to $\< \A,\toA \>$, and $\seq{c_i}{i \in \w}$ is a diligent walk through $\< \C,\toC \>$ that projects via $\psi$ onto $\seq{a_i}{i \in \w}$.

\vspace{2mm}

It remains to show that $\psi$ is incompatible with $\phi$. Suppose, aiming for a contradiction, that $\psi$ is compatible with $\phi$. That is, suppose that there is a strongly connected digraph $\< \W,\toW \>$ and epimorphisms $\bar \phi$ and $\bar \psi$ from $\< \W,\toW \>$ to $\< \B,\toB \>$ and $\< \C,\toC \>$, respectively, such that $\phi \circ \bar \phi = \psi \circ \bar \psi$. 

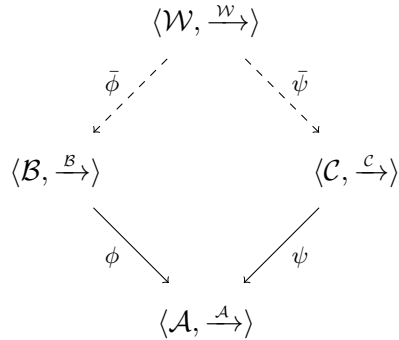
\begin{center}
\begin{tikzpicture}[scale=1]

\node at (0,0) {$\< \A,\toA \>$};
\node at (-2,2) {$\< \B,\toB \>$};
\node at (2,2) {$\< \C,\toC \>$};
\node at (0,4) {$\< \W,\toW \>$};

\draw[<-] (.5,.5)--(1.5,1.5);
\draw[<-] (-.5,.5)--(-1.5,1.5);
\draw[<-,dashed] (1.5,2.5)--(.5,3.5);
\draw[<-,dashed] (-1.5,2.5)--(-.5,3.5);

\node at (1.25,.85) {\footnotesize $\psi$};
\node at (-1.25,.85) {\footnotesize $\phi$};
\node at (1.25,3.15) {\footnotesize $\bar \psi$};
\node at (-1.25,3.15) {\footnotesize $\bar \phi$};

\end{tikzpicture}
\end{center}

Because $\bar \psi$ is a surjection $\W \to \C$, the fibers $\bar \psi^{-1}(\fS_\text{in})$ and $\bar \psi^{-1}(\0)$ are both nonempty. Using the fact that $\< \W,\toW \>$ is strongly connected, fix a walk $\< w_0,w_1,w_2,\dots,w_n \>$ in $\< \W,\toW \>$ from some $w_0 \in \bar \psi^{-1}(\fS_\text{in})$ to some $w_n \in \bar \psi^{-1}(\0)$.

Because $\bar \psi$ is an epimorphism, $\< \bar \psi(w_0),\bar \psi(w_1),\bar \psi(w_2),\dots,\bar \psi(w_n) \>$ is a walk in $\< \C,\toC \>$ from $\bar \psi(w_0) = \fS_\text{in}$ to $\bar \psi(w_n) = \0$. Without loss of generality, we may (and do) assume that $\bar \psi(w_k) \neq \0$ for any $k < n$. (If we did have $\bar \psi(w_k) = \0$ for some $k < n$, then we could just consider the shorter walk $\< w_0,w_1,w_2,\dots,w_k \>$ instead of the original walk $\< w_0,w_1,w_2,\dots,w_n \>$.) But observe that every walk in $\< \C,\toC \>$ from $\fS_\text{in}$ to a vertex in $\A$ must include the vertex $\0$, because $(\0,a_\text{out})$ is the only member of $\toC$ going from a member of $\C_0$ to a member of $\A$. Because we are assuming $\bar \psi(w_i) \neq \0$ for all $i < n$, this implies $\bar \psi(w_i) \in \C_0$ for all $i \leq n$.  

Because $\bar \phi$ is an epimorphism, $\< \bar \phi(w_0),\bar \phi(w_1),\bar \phi(w_2),\dots,\bar \phi(w_n) \>$ is a walk in $\< \B,\toB \>$. We now prove, by induction on $i$, that $\bar \phi(w_i) \in \bar \psi(w_i)$ for all $i \leq n$. 
(By the previous paragraph, $\bar \psi(w_i) \in \C_0 \in \St^\phi$ whenever $i \leq n$. Therefore the assertion $\bar \phi(w_i) \in \bar \psi(w_i)$ at least makes sense: we have $\bar \phi(w_i) \in \B$ and $\bar \psi(w_i) \sub \B$ whenever $i \leq n$.)

For the base case $i=0$, observe that 
$$\phi \circ \bar \phi(w_0) \,=\, \psi \circ \bar \psi(w_0) \,=\, \psi(\fS_\text{in}) \,=\, \Pi_\phi(\fS_\text{in}).$$ 
Consequently, $\bar \phi(w_0) \in \phi^{-1}(\Pi_\phi(\fS_\text{in}))$. 
Now fix some $m \in D$. 
Recall that $\fS_\text{in} = \fS^{\bar a_m}_0 = \phi^{-1}(a_m) \neq \0$, which means $\Pi_\phi(\fS_\text{in}) = a_m$. Consequently, $\phi^{-1}(\Pi_\phi(\fS_\text{in})) = \phi^{-1}(a_m) = \fS_\text{in}$. Hence 
$$\bar \phi(w_0) \,\in\, \phi^{-1}(\Pi_\phi(\fS_\text{in})) \,=\, \fS_\text{in} \,=\, \bar \psi(w_0).$$ 

For the successor step, fix $i < n$ and suppose (the inductive hypothesis) that $\bar \phi(w_i) \in \bar \psi(w_i)$; we wish to show $\bar \phi(w_{i+1}) \in \bar \psi(w_{i+1})$. 
Because $\bar \psi$ is an epimorphism and $w_i \toW w_{i+1}$, $\bar \psi(w_i) \toC \bar \psi(w_{i+1})$. 
Because $\bar \psi(w_i),\bar \psi(w_{i+1}) \in \C_0$ (rather than $\A$), this means $(\bar \psi(w_i),\bar \psi(w_{i+1})) \in \E_0$, which is to say that for some (any) $m \in D$ there is some $k < k_m$ such that $\bar \psi(w_i) = \fS^{\bar a_m}_k$ and $\bar \psi(w_{i+1}) = \fS^{\bar a_m}_{k+1}$. 
But then
$$\phi \circ \bar \phi(w_i) \,=\, \psi \circ \bar \psi(w_i) \,=\, \psi(\fS^{\bar a_m}_k) \,=\, a_{m+k} \text{, and}$$
$$\phi \circ \bar \phi(w_{i+1}) \,=\, \psi \circ \bar \psi(w_{i+1}) \,=\, \psi(\fS^{\bar a_m}_{k+1}) \,=\, a_{m+k+1}.$$
In particular, $\bar \phi(w_{i+1}) \in \phi^{-1}(a_{m+k+1})$. Furthermore, $\bar \phi(w_i) \toB \bar \phi(w_{i+1})$ (because $w_i \toW w_{i+1}$ and $\bar \phi$ is an epimorphism), and $\bar \phi(w_i) \in \bar \psi(w_i) = \fS^{\bar a_m}_k$ (by the inductive hypothesis). 
But by Lemma~\ref{lem:Recursion}, 
$$\fS^{\bar a_m}_{k+1} \,=\, \set{b \in \phi^{-1}(a_{m+k+1})}{\exists x \in \fS^{\bar a_m}_k \text{ such that } x \toB y},$$
so this means that $\bar \phi(w_{i+1}) \in \fS^{\bar a_m}_{k+1} = \bar \psi(w_{i+1})$. 
This completes the induction. 

This shows that $\bar \phi(w_i) \in \bar \psi(w_i)$ for all $i \leq n$. In particular, $\bar \phi(w_n) \in \bar \psi(w_n)$. But $\bar \psi(w_n) = \0$. Contradiction! 
\end{proof}

This completes the first stage of our proof of Theorem~\ref{thm:Heart&Soul}. 
The second stage is similar in its broad strokes: either some condition holds that (in a relatively easy way) implies alternative $(1)$ of the theorem, or else we must construct a digraph $\< \C,\toC \>$ witnessing alternative $(2)$. 


\begin{proof}[Proof of Theorem~\ref{thm:Heart&Soul}] 
Let $\< \A,\toA \>$ and $\< \B,\toB \>$ be strongly connected digraphs, let $\phi$ be an epimorphism from $\< \B,\toB \>$ onto $\< \A,\toA \>$, and  
let $\bar a = \seq{a_n}{n \in \w}$ be a diligent walk through $\< \A,\toA \>$. 

By Theorem~\ref{thm:DigraphPathLifting}, if there is no infinite walk in $\< \B,\toB \>$ that almost projects onto $\bar a$, then 
alternative $(2)$ from the statement of Theorem~\ref{thm:DigraphPathLifting} holds. 
But alternative $(2)$ from the statement of Theorem~\ref{thm:DigraphPathLifting} implies alternative $(2)$ from the statement of Theorem~\ref{thm:Heart&Soul}. 
Therefore, if there is no infinite walk in $\< \B,\toB \>$ that almost projects onto $\bar a$, then alternative $(2)$ holds and we are done. 
Thus, for the remainder of the proof, we may (and do) assume that there is an infinite walk in $\< \B,\toB \>$ that almost projects onto $\seq{a_n}{n \in \w}$. 

\vspace{2mm}

For each $\ell,m \in \w$ with $\ell < m$, define a relation $\mathcal R_{\ell,m} \sub \B \times \B$ as follows:
\begin{align*}
(v,w) \in \mathcal R_{\ell,m} \qquad \Leftrightarrow \qquad &\text{there is a walk $\< b_\ell,b_{\ell+1},\dots,b_m \>$ in $\< \B,\toB \>$} \\
&\text{from $v$ to $w$ that projects onto $\< a_\ell,a_{\ell+1},\dots,a_m \>$.}
\end{align*}
This relation expresses the ``walkability'' between the members of $\B$, using correctly projecting walks beginning at time $\ell$ and ending at time $m$. 

Because $\B$ is finite, the mapping $\chi: [\w]^2 \to \P(\B \times \B)$ defined by
$$\chi(\ell,m) \,=\, \mathcal R_{\min\{\ell,m\},\max\{\ell,m\}}$$
is a finite coloring of $[\w]^2$. 
By the infinite version of Ramsey's Theorem, there is some particular $\mathcal R \sub \B \times \B$ and an infinite $D \sub \w$ such that if $\ell,m \in D$ and $\ell < m$, then $\mathcal R_{\ell,m} = \mathcal R$.
We may (and do) assume $D$ does not contain two consecutive integers.  
Henceforth, let $m_k$ denote the $k^\mathrm{th}$ member of $D$. Before moving on, let us make some observations concerning this set $D$ and this relation $\mathcal R$, both of which play an important part in the remainder of the proof.

\vspace{2mm}

\noindent \textbf{Observation:} $\mathcal R$ is not the empty relation.  

\vspace{2mm}

\noindent \emph{Proof of Observation:} By the assumption at the beginning of the proof, there is an infinite walk $\seq{b^0_i}{i \in \w}$ in $\< \B,\toB \>$ that almost projects onto $\seq{a_i}{i \in \w}$. From our definition of $\mathcal R_{\ell,m}$, it follows that $(b^0_\ell,b^0_m) \in \mathcal R_{\ell,m}$ for all sufficiently large $\ell,m \in \w$ with $\ell < m$. In particular, because $D$ is infinite this is true for some $\ell,m \in D$, and consequently $\mathcal R \neq \0$.
\hfill$\dashv$

\vspace{2mm}

In what follows, let $\fB$ denote the domain of this relation $\mathcal R$, i.e., 
$\fB = \set{v \in \V}{(v,v') \in \mathcal R \text{ for some }v'}$.

\vspace{2mm}

\noindent \textbf{Observation:} $\mathcal R$ is a transitive relation.

\vspace{2mm}

\noindent \emph{Proof of Observation:} If $k,\ell,m \in D$ with $k < \ell < m$, then 
$$\qquad\qquad\qquad\qquad\, \mathcal R \,=\, \mathcal R_{k,m} \,=\, \mathcal R_{\ell,m} \circ \mathcal R_{k,\ell} \,=\, \mathcal R \circ \mathcal R. \qquad\qquad\qquad\qquad \dashv$$

\vspace{2mm}

\noindent \textbf{Observation:} There is some $v \in \B$ such that $(v,v) \in \mathcal R$. 

\vspace{2mm}

\noindent \emph{Proof of Observation:}  
By assumption, there is an infinite walk $\seq{b^0_i}{i \in \w}$ in $\< \B,\toB \>$ that almost projects onto $\seq{a_i}{i \in \w}$. 
By the pigeonhole principle, there is some $v \in \B$ such that $b_m = v$ for infinitely many $m \in D$. 
As in the proof of our first observation above, $(b_\ell,b_m) \in \mathcal R$ for sufficiently large $\ell < m$ with $\ell,m \in D$. Therefore $(v,v) \in \mathcal R$.
\hfill$\dashv$

\vspace{2mm}

\noindent \textbf{Observation:} If $\ell,m \in D$, then $a_\ell = a_m$.  

\vspace{2mm}

\noindent \emph{Proof of Observation:} Fix some $v \in \B$ such that $(v,v) \in \mathcal R$, and let $\ell,m \in D$ with $\ell < m$. Because $(v,v) \in \mathcal R = \mathcal R_{\ell,m}$, the definition of $\mathcal R_{\ell,m}$ implies there is a walk $\< b_\ell,\dots,b_m \>$ in $\< \B,\toB \>$ from $v = b_\ell$ to $v = b_m$ that projects onto $\< a_\ell,\dots,a_m \>$. But (by the definition of ``projects") this means $a_\ell = \phi(b_\ell) = \phi(v)$ and $a_m = \phi(b_m) = \phi(v)$.
\hfill$\dashv$

\vspace{2mm}

Let us write $a_D$ to denote the unique member of $\set{a_m}{m \in D}$.

\vspace{2mm}

Next, we associate to each ``time" $i \in \w \setminus \{0,1,\dots,m_0\}$ two relations $\P^i$ and $\F^i$ (for \emph{Past} and \emph{Future}) that generalize the walkability relation $\mathcal R$. 

If $i \in D$, let $\P^i = \F^i = \mathcal R$. 
If $i \notin D$ and $i > m_0 = \min D$, then define:
\begin{align*}
(v,b) \in \P^i \quad \Leftrightarrow \quad &\text{$v \in \B$ and for some $m \in D$ with $m < i$ there is} \\
&\text{a walk $\< b_{m},b_{m+1},\dots,b_i \>$ in $\< \B,\toB \>$ from $v = b_m$} \\
&\text{to $b = b_i$ that projects onto $\< a_{m},a_{m+1},\dots,a_i \>$.} \\
\vphantom{f^{f^{f^{f^f}}}} (b,v) \in \F^i \quad \Leftrightarrow \quad &\text{$v \in \B$ and for some $m \in D$ with $m > i$ there is} \\
&\text{a walk $\< b_{i},b_{i+1},\dots,b_{m} \>$ in $\< \B,\toB \>$ from $b=b_i$} \\
&\text{to $v = b_m$ that projects onto $\< a_i,a_{i+1},\dots,a_{m} \>$.}
\end{align*}

\noindent This defines $\P^i$ and $\F^i$ for $i \geq m_0$. It is not necessary to define $\P^i$ and $\F^i$ for $i < m_0$, or if desired, one may define them there arbitrarily.

\vspace{2mm}

\noindent \textbf{Observation:} $\F^i \supseteq \F^i \circ \mathcal R$ for any $i \geq m_0$.

\vspace{2mm}

\noindent \emph{Proof of Observation:} If $i \in D$ then this follows from our observation that $\mathcal R = \mathcal R \circ \mathcal R$, so fix $i \notin D$. Suppose $(b,v) \in \F^i$ and $(v,v') \in \mathcal R$. The former means that, for some $m \in D$ with $m > i$, there is a walk $\< b_i,\dots,b_m \>$ from $b = b_i$ to $v = b_m$ that projects onto $\< a_i,\dots,a_m \>$; the latter means that, for some $m' \in D$ with $m' > m$, there is a walk $\< b_m,\dots,b_{m'} \>$ from $v = b_m$ to $v' = b_{m'}$ that projects onto $\< a_m,\dots,a_{m'} \>$. Concatenating these two walks gives us a walk $\< b_i,\dots,b_{m'} \>$ from $b = b_i$ to $v' = b_{m'}$ that projects onto $\< a_i,\dots,a_{m'} \>$. But this means $(b,v') \in \F^i$. 
\hfill$\dashv$

\vspace{2mm}

\noindent \textbf{Observation:} If $i \geq m_2$, then $\P^i \supseteq \mathcal R \circ \P^i$.

\vspace{2mm}

\noindent \emph{Proof of Observation:} If $i \in D$ then this follows from our observation that $\mathcal R = \mathcal R \circ \mathcal R$, so fix $i \notin D$ with $i > m_2$. Suppose $(v,b) \in \P^i$ and $(v',v) \in \mathcal R$. 

Because $(v,b) \in \P^i$, for some $m \in D$ with $m < i$ there is a walk $\< b_m,\dots,b_i \>$ from $v = b_m$ to $b = b_i$ that projects onto $\< a_m,\dots,a_i \>$. We claim that, without loss of generality, we may take $m > m_0$. To see this, suppose instead that $m = m_0$. Recalling that $i > m_2$, the truncated walk $\< b_{m_0},\dots,b_{m_2} \>$ witnesses $(b_{m_0},b_{m_2}) \in \mathcal R_{m_0,m_2} = \mathcal R$. 
But $\mathcal R = \mathcal R_{m_1,m_2}$ as well, so there is a walk $\< b'_{m_1},\dots,b'_{m_2} \>$ from $b'_{m_1} = b_{m_0}$ to $b'_{m_2} = b_{m_2}$ that projects onto $\< a_{m_1},\dots,a_{m_2} \>$. 
Consequently, $\< b'_{m_1},\dots,b'_{m_2} = b_{m_2},b_{m_2+1},\dots,b_i \>$ is a walk from $b_{m_0} = b'_{m_1}$ to $v = b_i$ that projects onto $\< a_{m_1},\dots,a_i \>$. 

Thus for some $m \in D \setminus \{m_0\}$ with $m < i$ there is a walk $\< b_m,\dots,b_i \>$ from $v = b_m$ to $b = b_i$ that projects onto $\< a_m,\dots,a_i \>$. 
Furthermore, because $(v',v) \in \mathcal R$ and $m > m_0$, there is a walk $\< b_{m_0},\dots,b_{m} \>$ from $v' = b_{m_0}$ to $v = b_{m}$ that projects onto $\< a_{m_0},\dots,a_{m} \>$. 
Concatenating these two walks gives us a walk $\< b_{m_0},\dots,b_{i} \>$ from $v' = b_{m_0}$ to $b = b_{i}$ that projects onto $\< a_{m_0},\dots,a_{i} \>$. Therefore $(v',b) \in \P^i$. 
\hfill$\dashv$

\vspace{2mm}


For notational convenience in what follows, if $v \in \B$ then define
\begin{align*}
\P^i_v &\,=\, \set{b \in \B}{ (v,b) \in \P^i } \qquad \text{and} \qquad
\F^i_v &\,=\, \set{b \in \B}{ (b,v) \in \F^i }.
\end{align*}
Thus we may write $b \in \P^i_v$ rather than $(v,b) \in \P^i$ (this can be read as ``$b$ has $v$ in its past at time $i$''); and we may write $b \in \F^i_v$ rather than $(b,v) \in \F^i$ (which can be read as ``$b$ has $v$ in its future at time $i$''). 
By the previous observation, if $b \in \P^i_v$ and $(v',v) \in \mathcal R$, then $b \in \P^i_{v'}$ (provided $i \geq m_2$) and if $b \in \F^i_v$ and $(v,v') \in \mathcal R$, then $b \in \F^i_{v'}$


\vspace{2mm}

\noindent \textbf{Observation:} If $i \geq m_0$, then $\P^i$ and $\F^i$ are nonempty. 

\vspace{2mm}

\noindent \emph{Proof of Observation:} We have already showed $\mathcal R \neq \0$, and we defined $\F^i = \P^i = \mathcal R$ for $i \in D$. 
If $i > m_0$ and $i \notin D$, then fix some $m > i$ with $m \in D$. 
Because $\mathcal R_{m_0,m} = \mathcal R \neq \0$, there is a walk $\< b_{m_0},\dots,b_m \>$ in $\<\B,\toB\>$ that projects onto $\< a_{m_0},\dots,a_m \>$. 
Then $(b_{m_0},b_i) \in \P^i$ and $(b_i,b_m) \in \F^i$.
\hfill$\dashv$

\vspace{2mm}

\noindent \textbf{Observation:} If $i,i' \in \w$ and $\P^i = \P^{i'} \neq \0$, then $a_i = a_{i'}$. 
Similarly, if $i,i' \in \w$ and $\F^i = \F^{i'} \neq \0$, then $a_i = a_{i'}$.

\vspace{2mm}

\noindent \emph{Proof of Observation:} 
Suppose $\0 \neq \P^i = \P^{i'}$. 
If $(v,b) \in \P^i = \P^{i'}$, then the definition of $\P^i$ implies $\phi(b) = a_i$, and likewise the definition of $\F^{i'}$ implies $\phi(b) = a_{i'}$, so $a_i = a_{i'}$. 
The argument is essentially the same for $\F^i,\F^{i'}$.  
\hfill$\dashv$

\vspace{2mm}

If $\P \sub \B \times \B$ and $\P = \P^i \neq \0$ for some $i \in \w$, define $\Pi(\P) = a_i$. Similarly, if $\F \sub \B \times \B$ and $\F = \F^i \neq \0$ for some $i \in \w$, define $\Pi(\F) = a_i$. Note that this mapping $\Pi$ is well-defined by the previous observation.

\vspace{2mm}

With these preliminaries in place, we now move on to the main line of argument. Consider the following statement:

\vspace{1mm}

\begin{itemize}
\item[$(\ddagger)$] There is some $v \in \fB$ with $(v,v) \in \mathcal R$, such that for every edge $e$ in $\toB$, 
there are infinitely many $\ell \in D$ such that for some $m \in D$ with $m > \ell$ 
there is a walk $\< b_{\ell},b_{\ell+1},b_{\ell+2},\dots,b_{m} \>$ in $\< \B,\toB \>$ from $b_{\ell} = v$ to $b_{m} = v$ such that this walk projects onto $\< a_{\ell},a_{\ell+1},a_{\ell+2},\dots,a_{m} \>$, and this walk traverses $e$, in the sense that  
$(b_i,b_{i+1}) = e$ for some $i$ with $\ell \leq i < m$.
\end{itemize}

\vspace{1mm}

\noindent As we shall see in what follows, this statement $(\ddagger)$ is a sharp dividing line between alternatives $(1)$ and $(2)$ in the statement of the theorem: if $(\ddagger)$ is true then $(1)$ holds, and if not then $(2)$ holds.  


\vspace{2mm}

To see that $(\ddagger) \Rightarrow (1)$, we construct by recursion a diligent walk through $\< \B,\toB \>$ that almost projects onto $\seq{a_n}{n \in \w}$. 
At step $n$ of the recursion, we add not just one point to our walk, but a long string of points with indices between some $m_{k_{n-1}}$ and $m_{k_n}$. (These indices $k_0,k_1,\dots$ are chosen along the way, as part of the construction.) 
Fix an enumeration of the edges of $\<\B,\toB\>$, say $\toB \ = \{e_0,e_1,\dots,e_{M-1}\}$, and fix some $v \in \fB$ witnessing $(\ddagger)$. 

For the base step of the recursion, let $k_0=0$ and let $\< b_0,b_1,b_2,\dots,b_{m_0} \>$ be a walk in $\< \B,\toB \>$ ending at $v = b_{m_0} = b_{m_{k_0}}$. Some such walk exists because $\< \B,\toB \>$ is strongly connected.

At step $n > 0$ of the recursion, we begin with a walk $\< b_0,b_1,\dots,b_{m_{k_{n-1}} }\>$ in $\< \B,\toB \>$ such that $b_{m_{k_{n-1}}} = v$. 
Let $j$ denote the unique member of $\{0,1,2,\dots,M-1\}$ such that $n \equiv j$ (mod $M$). 
Because $b_{m_{k_{n-1}}} = v$, and because $v$ witnesses $(\ddagger)$, 
there is some $k_n > k_{n-1}$ such that there is a walk 
$\< b_{m_{k_{n-1}}},b_{m_{k_{n-1}}+1},b_{m_{k_{n-1}}+2},\dots,b_{m_{k_n}} \>$ 
in $\< \B,\toB \>$ from $v = b_{m_{k_{n-1}}}$ to $v = b_{m_{k_n}}$ that projects onto 
$\< a_{m_{k_{n-1}}},a_{m_{k_{n-1}}+1},a_{m_{k_{n-1}}+2},\dots,a_{m_{k_n}} \>$, and such that 
$e_j = (b_i,b_{i+1})$ for some $i$ with $m_{k_{n-1}} \leq i < m_{k_n}$. 
We take this walk to define $b_i$ when $m_{k_{n-1}} < i \leq m_{k_n}$.

We claim that the sequence $\seq{b_i}{i \in \w}$ produced by this recursion is a diligent walk through $\< \B,\toB \>$ that almost projects onto $\seq{a_i}{i \in \w}$. 

It is clear that $\seq{b_i}{i \in \w}$ is an infinite walk in $\< \B,\toB \>$. 
Furthermore, this walk is diligent because for any given $e_j$ in $\toB$, if $n \equiv j$ (mod $M$) then $(b_i,b_{i+1}) = e_j$ for some $i$ with $m_{k_{n-1}} \leq i < m_{k_n}$, so that every member of $\toB$ is equal to $(b_i,b_{i+1})$ for infinitely many $i \in \w$. 
Finally, our construction guarantees that $\phi(b_i) = a_i$ for all $i \geq m_0$, which means $\seq{b_i}{i \in \w}$ almost projects onto $\seq{a_i}{i \in \w}$. 
This completes the proof that $(\ddagger)$ implies alternative $(1)$ from the statement of the theorem. 

\vspace{2mm}

It remains to show that if $(\ddagger)$ fails then alternative $(2)$ from the statement of the theorem holds. As in the proof of Theorem~\ref{thm:DigraphPathLifting}, this involves the construction of a digraph $\< \C,\toC \>$. 
For each $k > 0$, define
\begin{align*}
\LL^k &\,=\, \set{ ( \P^i,\F^i )}{m_k < i < m_{k+1}}, \\
\E^k &\,=\, \set{\big( (\P^i,\F^i),(\P^{i+1},\F^{i+1}) \big)}{m_k < i < m_{k+1}-1}, \\
c^k_\text{in}  &\,=\, \big( \P^{m_k+1},\F^{m_k+1} \big), \\
c^k_\text{out}  &\,=\, \big( \P^{m_{k+1}-1},\F^{m_{k+1}-1} \big). 
\end{align*}


A $4$-tuple $(\LL,\E,c_\text{in},c_\text{out})$ is \emph{used infinitely often} if there are infinitely many $k > 0$ with $(\LL,\E,c_\text{in},c_\text{out}) = (\LL^k,\E^k,c^k_\text{in},c^k_\text{out}).$ 
Because there are only finitely many possibilities for the $4$-tuple $(\LL^k,\E^k,c^k_\text{in},c^k_\text{out})$, there is some $K > 0$ such that $(\LL^k,\E^k,c^k_\text{in},c^k_\text{out})$ is used infinitely often for every $k \geq K$. 

Let $N$ denote the number of $4$-tuples that are used infinitely often, and fix indices 
$k_0,k_1,\dots,k_N$ such that 
$$(\LL^{k_0},\E^{k_0},c^{k_0}_\text{in}c^{k_0}_\text{out}),\dots,(\LL^{k_N},\E^{k_N},c^{k_N}_\text{in},c^{k_N}_\text{out})$$
is a complete list of all the $4$-tuples used infinitely often. 
For each $j \leq N$ let
\begin{align*}
\C_j &\,=\, \{j\} \times \LL^{k_j}, \\
\E_j &\,=\, \set{\big( (j,t),(j,t') \big)}{(t,t') \in \E^{k_j}}, \\
c^j_\text{in} &\,=\, \big( j,c^{k_j}_\text{in} \big), \\
c^j_\text{out} &\,=\, \big( j,c^{k_j}_\text{out} \big).
\end{align*}
In other words, these are the $4$-tuples described above, only we have modified the vertex sets of the $\LL^{k_j}$ so that all the different $\C_j$'s are disjoint. 
Furthermore let us assume, without loss of generality, that $\mathrm{dom}(\mathcal R) = \fB \notin \bigcup_{j \leq N}\C_j$. (This is true automatically unless the vertices of $\B$ are labelled in an unreasonable way.) 

We are now ready to define the digraph $\< \C,\toC \>$ that will witness alternative $(2)$ from the statement of the theorem.  First, let 
$$\C \,=\, \{\fB\} \cup \C_0 \cup \C_1 \cup \dots \cup \C_N.$$
Then define the edge relation on $\C$ as follows:
\begin{align*}
x \toC y \quad \ \Longleftrightarrow \quad \ & x,y \in \C_j \text{ and } (x,y) \in \E_j \text{ for some $j \leq N$, or} \\
& x=\fB \text{ and } y = c^j_\text{in} \text{ for some $j \leq N$, or} \\
& x = c^j_{\mathrm{out}} \text{ and } y = \fB \text{ for some $j \leq N$}. 
\end{align*}

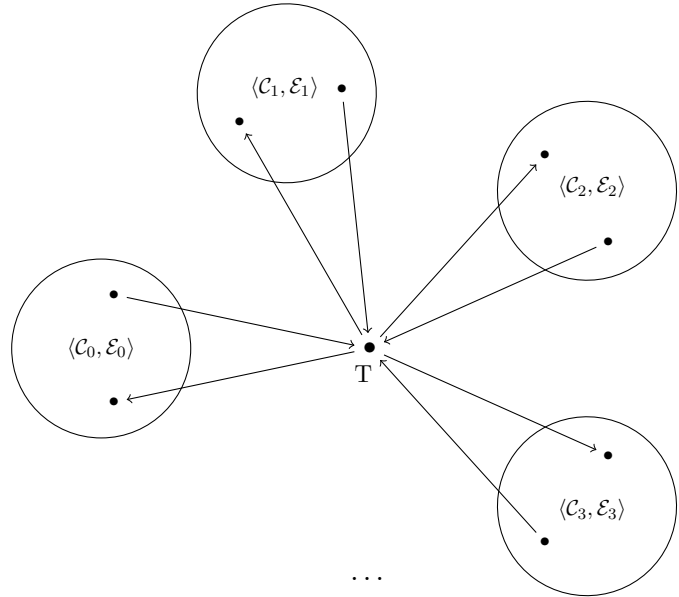
\begin{center}
\begin{tikzpicture}[scale=1.7]

\node at (0,0) {\small $\bullet$};
\node at (-.05,-.2) {\scalebox{.85}{$\fB$}};

\draw (-2.1,0) circle (7mm);
\draw[->] (-.12,-.025)--(-1.9,-.4);
\node at (-2,-.42) {\tiny $\bullet$};
\draw[<-] (-.12,.025)--(-1.9,.4);
\node at (-2,.42) {\tiny $\bullet$};
\node at (-2.1,0) {\scalebox{.75}{$\< \C_0,\E_0 \>$}};

\begin{scope}[rotate=-72]
\draw (-2.1,0) circle (7mm);
\draw[->] (-.12,-.025)--(-1.9,-.4);
\node at (-2,-.42) {\tiny $\bullet$};
\draw[<-] (-.12,.025)--(-1.9,.4);
\node at (-2,.42) {\tiny $\bullet$};
\node at (-2.15,0) {\scalebox{.75}{$\< \C_1,\E_1 \>$}};
\end{scope}

\begin{scope}[rotate=-144]
\draw (-2.1,0) circle (7mm);
\draw[->] (-.12,-.025)--(-1.9,-.4);
\node at (-2,-.42) {\tiny $\bullet$};
\draw[<-] (-.12,.025)--(-1.9,.4);
\node at (-2,.42) {\tiny $\bullet$};
\node at (-2.15,0) {\scalebox{.75}{$\< \C_2,\E_2 \>$}};
\end{scope}

\begin{scope}[rotate=-216]
\draw (-2.1,0) circle (7mm);
\draw[->] (-.12,-.025)--(-1.9,-.4);
\node at (-2,-.42) {\tiny $\bullet$};
\draw[<-] (-.12,.025)--(-1.9,.4);
\node at (-2,.42) {\tiny $\bullet$};
\node at (-2.15,0) {\scalebox{.75}{$\< \C_3,\E_3 \>$}};
\end{scope}

\node at (0,-1.8) {$\dots$};

\end{tikzpicture}
\end{center}
\vspace{2mm}

Define $\psi: \C \to \A$ by setting $\psi(\fB) = a_D$, and $\psi(j,(\P,\F)) = \Pi(\P) = \Pi(\F)$. (To see that this last equality is true, observe that if $(j,(\P,\F)) \in \C_j$, then $(\P,\F) = (\P^i,\F^i)$ for some $i$, and this implies $\Pi(\P_i) = \Pi(\F_i) = a_i$.)


This completes the definition of $\< \C,\toC \>$ and $\psi$. 
%
To finish the proof, we need to show that $\< \C,\toC \>$ is strongly connected, that $\psi$ is an epimorphism from $\< \C,\toC \>$ to $\< \A,\toA \>$, that there is a diligent walk through $\< \C,\toC \>$ projecting onto $\seq{a_n}{n \in \w}$ via $\psi$, and that $\psi$ is incompatible with $\phi$. 

\vspace{2mm}

Like in the proof of Theorem~\ref{thm:DigraphPathLifting}, we can verify the first three of these four things all at once via Lemma~\ref{lem:EpisForFree}. To do this, we must define a diligent walk $\seq{c_i}{i \in \w}$ through $\< \C,\toC \>$ that almost projects onto $\seq{a_i}{i \in \w}$ via $\psi$. 
Recall that $K$ was chosen sufficiently large so that if $k \geq K$ then $(\LL^k,\E^k,c^k_\text{in},c^k_\text{out})$ is used infinitely often. 

Given $i \geq m_0$, let $k(i)$ be the (unique) index with $m_{k(i)} \leq i < m_{k(i)+1}$. Furthermore, if $i \geq m_K$ then define $j(i)$ to be the (unique) index $\leq\! N$ with 
$$\big( \LL^{k(i)},\E^{k(i)},c^{k(i)}_\text{in},c^{k(i)}_\text{out} \big) \,=\, \big( \LL^{k_{j(i)}},\E^{k_{j(i)}},c^{k_{j(i)}}_\text{in},c^{k_{j(i)}}_\text{out} \big).$$
For all $i \geq m_K$, define 
$$c_i \,=\, \begin{cases}
\fB &\text{ if } i \in D, \\
\big( j(i),(\P^i,\F^i) \big) &\text{ otherwise.}
\end{cases}$$

Let us check that this sequence is a walk through $\< \C,\toC \>$. 
For the first case, suppose $i,i+1 \notin D$. 
This implies $k(i) = k(i+1)$ and $j(i) = j(i+1)$, which means $c_i,c_{i+1} \in \C_{j(i)}$ and furthermore, 
$\big( (\P^i,\F^i),(\P^{i+1},\F^{i+1}) \big) \in \E^{k(i)}$ (by the definition of $\E^{k(i)}$), 
which gives 
$(c_i,c_{i+1}) \in \E_{j(i)}$ (by the definition of $\E_{j(i)}$). 
Hence $c_i \toC c_{i+1}$ in this case. 
For the second case, suppose $i \in D$ but $i+1 \notin D$. 
Then $i = m_{k(i)} = m_{k(i+1)}$ and $i+1 = m_{k(i)}+1$, which means 
$c_{i+1} \,=\, \big( j(i+1), (\P^{m_{k(i)}+1},\F^{m_{k(i)}+1} ) \big) \,=\, c^{j(i+1)}_\text{in},$
while $c_i = \fB$. 
Thus $c_i \toC c_{i+1}$ in this case also. 
For the third case, suppose $i \notin D$ but $i+1 \in D$. 
Then $i+1 = m_{k(i+1)}$ and $k(i) = k(i+1)-1$, which implies (like in the previous case) that   
$c_i = c^{j(i)}_\text{out},$
while $c_{i+1} = \fB$. 
Thus $c_i \toC c_{i+1}$ in this case also. 
These three cases are exhaustive, by our assumption near the beginning of the proof that $D$ does not contain two consecutive integers. 
Hence $\seq{c_i}{i \geq m_K}$ is a walk. 

To see that this walk is diligent, fix some edge $(c,c')$ in $\toC$. There is some $j \leq N$ such that either $(1)$ $(c,c') \in \E_j$, or $(2)$ $c = \fB$ and $c' = c^j_\text{in}$, or $(3)$ $c = c^j_\text{out}$ and $c' = \fB$. 
Because 
$\big( \LL^{k_j},\E^{k_j},c^{k_j}_\text{in}c^{k_j}_\text{out} \big)$ 
is used infinitely often, there are arbitrarily large values of $k \in \w$ such that $j(m_k) = j$. 
But then there is some $i$ with $k(i) = k$ and $m_{k(i)} \leq i \leq m_{k(i)+1}$ such that $(c_i,c_{i+1}) = (c,c')$. Thus $(c,c') = (c_i,c_{i+1})$ for arbitrarily large values of $i$, and this shows that $\seq{c_i}{i \geq m_K}$ is a diligent walk. 

Furthermore, if $i \in D$, then $a_i = a_D = \psi(c_i) = a_D$; and if $i \notin D$, then 
$\Pi_\phi(c_i) = \psi \big( j(i),(\P^i,\F^i) \big) = \Pi(\P^i) = \Pi(\F^i) = a_i.$ 
Therefore this walk projects onto $\seq{a_i}{i \geq m_K}$. 
Because $\seq{a_i}{i \in \w}$ is a diligent walk through $\< \A,\toA \>$, the tail sequence $\seq{a_i}{i \geq m_K}$ is also a diligent walk through $\< \A,\toA \>$. 
We have thus showed that $\seq{c_i}{i \geq m_K}$ is a diligent walk through $\< \C,\toC \>$ that projects onto a diligent walk $\seq{a_i}{i \geq m_K}$ through $\< \A,\toA \>$. 
By Lemma~\ref{lem:EpisForFree}, $\< \C,\toC \>$ is strongly connected, and $\psi$ is an epimorphism from $\< \C,\toC \>$ onto $\< \A,\toA \>$. 

Now that we know $\< \C,\toC \>$ is strongly connected, we can use this fact to find a length-$m_K$ walk $\< c^0,c^1,\dots,c^{m_K}\>$ in $\< \C,\toC \>$ such that $c^{m_K} = c_{m_K} = \fB$. 
And then 
$\< c^0,c^1,\dots,c^{m_K}=c_{m_K},c_{m_K+1},c_{m_K+2},\dots \>$ 
is a diligent walk through $\< \C,\toC \>$ that almost projects onto $\seq{a_i}{i \in \w}$.

\vspace{2mm}

To finish the proof, it remains only to show that if $(\ddagger)$ fails then $\psi$ is incompatible with $\phi$. 
We do this by proving the contrapositive: if $\psi$ and $\phi$ are compatible, then $(\ddagger)$ holds.

To this end, let us suppose $\psi$ is compatible with $\phi$. That is, suppose there is a strongly connected digraph $\< \W,\toW \>$ and epimorphisms $\bar \phi$ and $\bar \psi$ from $\< \W,\toW \>$ to $\< \B,\toB \>$ and $\< \C,\toC \>$, respectively, such that $\phi \circ \bar \phi = \psi \circ \bar \psi$. 

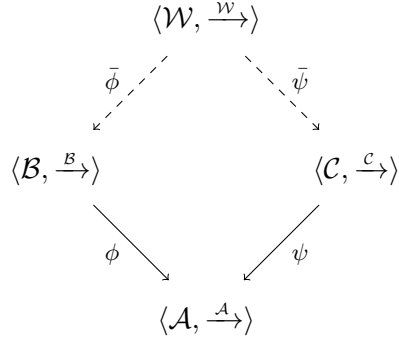
\begin{center}
\begin{tikzpicture}[scale=1]

\node at (0,0) {$\< \A,\toA \>$};
\node at (-2,2) {$\< \B,\toB \>$};
\node at (2,2) {$\< \C,\toC \>$};
\node at (0,4) {$\< \W,\toW \>$};

\draw[<-] (.5,.5)--(1.5,1.5);
\draw[<-] (-.5,.5)--(-1.5,1.5);
\draw[<-,dashed] (1.5,2.5)--(.5,3.5);
\draw[<-,dashed] (-1.5,2.5)--(-.5,3.5);

\node at (1.25,.85) {\footnotesize $\psi$};
\node at (-1.25,.85) {\footnotesize $\phi$};
\node at (1.25,3.15) {\footnotesize $\bar \psi$};
\node at (-1.25,3.15) {\footnotesize $\bar \phi$};

\end{tikzpicture}
\end{center}

For each $c \in \C$ and $v \in \B$, define 
$$\P_v(c) \,=\, \begin{cases}
\set{b \in \B}{ (v,b) \in \mathcal R} &\text{ if }c = \fB, \\
\,\P^i_v &\text{ if $c = (j(i),(\P^i,\F^i))$ for some $i \notin D$.}
\end{cases}$$ 
$$\F_v(c) \,=\, \begin{cases}
\set{b \in \B}{ (b,v) \in \mathcal R} &\text{ if }c = \fB, \\
\,\F^i_v &\text{ if $c = (j(i),(\P^i,\F^i))$ for some $i \notin D$.}
\end{cases}$$ 

\vspace{2mm}

\noindent \textbf{Claim:} Suppose $\<w_0,w_1,\dots,w_\ell\>$ is a walk in $\< \W,\toW \>$ with $\bar \psi(w_\ell) = \fB$, and let $v = \bar \phi(w_\ell)$. Then $\bar \phi(w_i) \in \F_{v}(\bar \psi(w_i))$ for all $i < \ell$.

\vspace{2mm}

\noindent \emph{Proof of Claim:} 
This claim is proved by a backwards inductive argument: we take $i = \ell-1$ as the base case, and prove for the inductive step that if $0 < i < \ell$ and $\bar \phi(w_i) \in \F_v(\bar \psi(w_i))$, then $\bar \phi(w_{i-1}) \in \F_v(\bar \psi(w_{i-1}))$.

For the base case $i = \ell-1$, recall that we are assuming $\bar \psi(w_\ell) = \fB$. Thus 
$$\phi \circ \bar \phi(w_\ell) \,=\, \psi \circ \bar \psi(w_\ell) \,=\, \psi(\fB) \,=\, a_D,$$
and therefore $\bar \phi(w_\ell) \in \phi^{-1}(a_D)$. 
Because $w_{\ell-1} \toW w_\ell$ and $\bar \psi$ is an epimorphism, $\bar \psi(w_{\ell-1}) \toC \bar \psi(w_\ell)$. 
Because $\bar \psi(w_\ell) = \fB$, this means $\bar \psi(w_{\ell-1}) = (j,c^j_{\text{out}})$ for some $j \leq N$, which means $\bar \psi(w_{\ell-1}) = \big( j,(\P^{m_k-1},\F^{m_k-1}) \big)$ for some $k \geq K$. 
Furthermore, $\bar \phi(w_{\ell-1}) \toB \bar \phi(w_\ell)$, because $w_{\ell-1} \toW w_\ell$ and $\bar \phi$ is an epimorphism. 
Thus $\< \bar \phi(w_{\ell-1}),\bar \phi(w_\ell) \>$ is a walk in $\< \B,\toB \>$ from $\bar \phi(w_{\ell-1})$ to $\bar \phi(w_\ell) = v$, and this walk projects onto $\< a_{m_k-1},a_{m_k}\>$ because 
$$\phi \circ \bar \phi(w_\ell) \,=\, \psi \circ \bar \psi(w_\ell) \,=\, \psi(\fB) \,=\, a_D \,=\, a_{m_k} \text{ and} \,$$
$$\phi \circ \bar \phi(w_{\ell-1}) \,=\, \psi \circ \bar \psi(w_{\ell-1}) \,=\, \Pi(\F^{m_k-1}) \,=\, a_{m_k-1}.$$
Consequently, by the definition of $\F$, $\bar \phi(w_{\ell-1}) \in \F_v^{m_k-1} = \F_v(\bar \psi(w_{\ell-1}))$.

For the inductive step, fix $i$ with $0 < i < \ell$ and suppose, as an inductive hypothesis, that $\bar \phi(w_i) \in \F_{v}(\bar \psi(w_i))$. 
We consider two cases. 
For the first case, suppose $\bar \psi(w_i) = \fB$. In this case, the inductive hypothesis states that $\bar \phi(w_i) \in \F_v(\fB)$. This means that $(\bar \phi(w_i),v) \in \mathcal R$, which implies in particular that $\bar \phi(w_i) \in \fB \sub \phi^{-1}(a_D)$. 
Once this is known, we may reason exactly as in the previous paragraph (replacing each instance of $\ell$ with an $i$) and deduce that $\bar \phi(w_{i-1}) \in \F_{\bar \phi(w_i)}(\bar \psi(w_{i-1}))$. 
Because $(\bar \phi(w_i),v) \in \mathcal R$, this implies (via the two observations following the definition of $\P$ and $\F$ and the comments following those observations) that $\bar \phi(w_{i-1}) \in \F_v(\bar \psi(w_{i-1}))$.

For the second case, suppose $\bar \psi(w_i) \neq \fB$. 
Because $w_{i-1} \toW w_i$ and $\bar \psi$ is an epimorphism, 
$\bar \psi(w_{i-1}) \toC \bar \psi(w_i)$. 
By our definition of $\<\C,\toC\>$, this means there is some $j \leq N$ and some $p$ with $m_{k_j} < p < m_{k_j+1}$ such that 
$\F(\bar \psi(w_i)) = \F^p$ and $\F(\bar \psi(w_{i-1})) = \F^{p-1}$. 
(The fact that $\bar \psi(w_i) \neq \fB$ rules out the possibility that $p \in D$.) 
By definition, $\bar \phi(w_i) \in \F_v(\bar \psi(w_i)) = \F_v^p$ means there is a walk $\< b_p,b_{p+1},\dots,\dots,b_{m'} \>$ in $\< \B,\toB \>$, for some $m' \in D$ with $m' \geq m_{k_{j+1}}$, that goes from $b_p = \bar \phi(w_i)$ to $v = b_{m'}$ and projects onto $\< a_p,a_{p+1},\dots,a_{m'} \>$. 
Furthermore, 
$$\phi \circ \bar \phi(w_i) \,=\, \psi \circ \bar \psi(w_i) \,=\, \Pi(\F^p) \,=\, a_p \text{, and} \,$$
$$\phi \circ \bar \phi(w_{i-1}) \,=\, \psi \circ \bar \psi(w_{i-1}) \,=\, \Pi(\F^{p-1}) \,=\, a_{p-1}.$$
In particular, $\phi \big( \bar \phi(w_{i-1}) \big) = a_{p-1}$, and of course $\bar \phi(w_{i-1}) \toB \bar \phi(w_i)$ (because $w_{i-1} \toW w_i$ and $\bar \phi$ is an epimorphism). But then this shows that 
$$\< \bar \phi(w_{i-1}),\bar \phi(w_{i})=b_p,b_{p+1},\dots,b_{m'} \>$$
is a walk in $\< \B,\toB \>$ from $\bar \phi(w_{i-1})$ to $b_{m'} = v$, and this walk projects onto $\< a_{p-1},a_p,a_{p+1},\dots,a_{m'} \>$. 
Hence $\bar \phi(w_{i-1}) \in \F_v^{p-1} =  \F_v(\bar \psi(w_{i-1}))$. 
\hfill $\dashv$

\vspace{2mm}

\noindent \textbf{Claim:} If $w \in \W$ and $\bar \psi(w) = \fB$, then $\bar \phi(w) \in \fB$ and $(\bar \phi(w),\bar \phi(w)) \in \mathcal R$.

\vspace{2mm}

\noindent \emph{Proof of Claim:}
Fix $w \in \W$ with $\bar \psi(w) = \fB$. Because $\< \W,\toW \>$ is strongly connected, there is a walk $\< w_0,w_1,\dots,w_\ell \>$ with length $\ell > 0$ from $w = w_0$ to $w = w_\ell$. 
Letting $v = \bar \phi(w_\ell)$, the previous claim gives $\bar \phi(w_0) \in \F_v(\bar \psi(w_0))$. 
By the definition of $\F_v(\bar \psi(w_0))$, this means $(\bar \phi(w_0),v) \in \mathcal R$. 
But $\bar \phi(w_0) = \bar \phi(w)$ and $v = \bar \phi(w_\ell) = \bar \phi(w)$, so $(\bar \phi(w),\bar \phi(w)) \in \mathcal R$.
\hfill $\dashv$

\vspace{2mm}

\noindent \textbf{Claim:} Suppose $\<w_0,w_1,\dots,w_\ell\>$ is a walk in $\< \W,\toW \>$ with $\bar \psi(w_0) = \fB$, and let $v = \bar \phi(w_0)$. Then $\bar \phi(w_i) \in \P_{v}(\bar \psi(w_i))$ for all $i > 0$.

\vspace{2mm}

\noindent \emph{Proof of Claim:} 
This claim is the mirror image of the one above concerning $\F$. It is proved in almost the same way, but by a regular induction (starting at $1$, not $0$) rather than a backward induction. 

For the base case $i = 1$, recall that we are assuming $\bar \psi(w_0) = \fB$. Thus 
$$\phi \circ \bar \phi(w_0) \,=\, \psi \circ \bar \psi(w_0) \,=\, \psi(\fB) \,=\, a_D,$$
and therefore $\bar \phi(w_0) \in \phi^{-1}(a_D)$. 
Because $w_0 \toW w_1$ and $\bar \psi$ is an epimorphism, $\bar \psi(w_0) \toC \bar \psi(w_1)$. 
Because $\bar \psi(w_0) = \fB$, this means $\bar \psi(w_1) = (j,c^j_{\text{in}})$ for some $j \leq N$, which means $\bar \psi(w_1) = (j,(\P^{m_k+1},\F^{m_k+1})$ for some $k \geq K$. 
Furthermore, $\bar \phi(w_0) \toB \bar \phi(w_1)$, because $w_0 \toW w_1$ and $\bar \phi$ is an epimorphism. 
Thus $\< \bar \phi(w_0),\bar \phi(w_1) \>$ is a walk in $\< \B,\toB \>$ from $v=\bar \phi(w_0)$ to $\bar \phi(w_1)$, and this walk projects onto $\< a_{m_k},a_{m_k+1}\>$ because 
$$\phi \circ \bar \phi(w_0) \,=\, \psi \circ \bar \psi(w_0) \,=\, \psi(\fB) \,=\, a_D \,=\, a_{m_k} \text{ and} \,$$
$$\phi \circ \bar \phi(w_{1}) \,=\, \psi \circ \bar \psi(w_{1}) \,=\, \Pi(\P^{m_k+1}) \,=\, a_{m_k+1}.$$
Consequently, by the definition of $\P$, $\bar \phi(w_{1}) \in \P_v^{m_k+1} = \P_v(\bar \psi(w_{1}))$.

For the inductive step, fix $i$ with $0 < i < \ell$ and suppose, as an inductive hypothesis, that $\bar \phi(w_i) \in \P_{v}(\bar \psi(w_i))$. 
We consider two cases. 
For the first case, suppose $\bar \psi(w_i) = \fB$. In this case, the inductive hypothesis states that $\bar \phi(w_i) \in \P_v(\fB)$. This means $(v,\bar \phi(w_i)) \in \mathcal R$, which implies $\bar \phi(w_i) \in \fB \sub \phi^{-1}(a_D)$. 
Once this is known, we may reason as in the previous paragraph (replacing $0$ with $i$ and $1$ with $i+1$) and deduce $\bar \phi(w_{i+1}) \in \P_{\bar \phi(w_i)}(\bar \psi(w_{i+1}))$. 
Because $(v,\bar \phi(w_i)) \in \mathcal R$, this implies $\bar \phi(w_{i-1}) \in \P_v(\bar \psi(w_{i-1}))$.

For the second case, suppose $\bar \psi(w_i) \neq \fB$. 
Because $w_i \toW w_{i+1}$ and $\bar \psi$ is an epimorphism, 
$\bar \psi(w_i) \toC \bar \psi(w_{i+1})$. 
By our definition of $\<\C,\toC\>$, this means there is some $j \leq N$ and some $p$ with $m_{k_j} < p < m_{k_j+1}$ such that 
$\P(\bar \psi(w_i)) = \P^p$ and $\P(\bar \psi(w_{i+1})) = \P^{p+1}$. 
(The fact that $\bar \psi(w_i) \neq \fB$ rules out the possibility that $p \in D$.) 
By definition, $\bar \phi(w_i) \in \P_v(\bar \psi(w_i)) = \P_v^p$ means there is a walk $\< b_{m'},b_{m'+1},\dots,b_{p} \>$ in $\< \B,\toB \>$, for some $m' \in D$ with $m' \leq m_{k_j}$, that goes from $v = b_{m'}$ to  $\bar \phi(w_i) = b_p$ and projects onto $\< a_{m'},a_{m'+1},\dots,a_{p} \>$. 
Furthermore, 
$$\phi \circ \bar \phi(w_i) \,=\, \psi \circ \bar \psi(w_i) \,=\, \Pi(\P^p) \,=\, a_p \text{, and} \,$$
$$\phi \circ \bar \phi(w_{i+1}) \,=\, \psi \circ \bar \psi(w_{i+1}) \,=\, \Pi(\P^{p+1}) \,=\, a_{p+1}.$$
In particular, $\phi \big( \bar \phi(w_{i+1}) \big) = a_{p+1}$, and of course $\bar \phi(w_i) \toB \bar \phi(w_{i+1})$ (because $w_{i} \toW w_{i+1}$ and $\bar \phi$ is an epimorphism). But then this shows that 
$$\< b_{m'},b_{m'+1},\dots,b_p = \bar \phi(w_i),\bar \phi(w_{i+1}) \>$$
is a walk in $\< \B,\toB \>$ from $b_{m'} = v$ to $\bar \phi(w_{i+1})$, and this walk projects onto $\< a_{m'},\dots,a_{p},a_{p+1} \>$. 
Hence $\bar \phi(w_{i+1}) \in \P_v^{p+1} =  \P_v(\bar \psi(w_{i+1}))$. 
\hfill $\dashv$

\vspace{2mm}

With these claims established, we are ready to finish the proof. 
Because $\bar \psi$ is an epimorphism, there is some $w \in \W$ such that $\bar \psi(w) = \fB$. 
We claim that $v = \bar \phi(w)$ witnesses $(\ddagger)$. 

From the second of the three claims above, $v \in \fB$ and $(v,v) \in \mathcal R$. 

Fix an edge $e = (b,b')$ in $\toB$. 
Because $\bar \phi$ is an epimorphism, there exist some $w,w' \in \W$ such that $w \toW w'$ and $(b,b') = (\bar \phi(w),\bar \phi(w'))$. 
Using the fact that $\< \W,\toW \>$ is strongly connected, there is a walk $\< w_0,w_1,\dots,w_\ell \>$ in $\< \W,\toW \>$ with $w_0 = w_\ell = v$ such that $(w,w') = (w_i,w_{i+1})$ for some $i < \ell$. 

Using the first and third of our three claims above, $\bar \phi(w_i) \in \P_v(\bar \psi(w_i))$ and $\bar \phi(w_{i+1}) \in \F_v(\bar \psi(w_{i+1}))$.

There is some $j \leq N$ such that either $\bar \psi(w_i) \in \C_j$ or $\bar \psi(w_{i+1}) \in \C_j$ (or both). 
Fix any $k \geq K$ with $j(m_k) = j$. 
Because $\bar \psi(w_i) \toC \bar \psi(w_{i+1})$, our definition of $\E_j$ implies there is some $p$ with $m_k \leq p < m_{k+1}$ such that 
$\bar \psi(w_i) = \big( j,(\P^p,\F^p) \big)$ and 
$\bar \psi(w_{i+1}) = \big( j,(\P^{p+1},\F^{p+1}) \big)$. 
Hence
$$b \,=\, \bar \phi(w_i) \,\in\, \P^p_v \qquad \text{ and } \qquad b' \,=\, \bar \phi(w_{i+1}) \,\in\, \F^{p+1}_v.$$
By definition, this means 
\begin{itemize}
\item[$\circ$] for some $\ell \in D$ with $\ell \leq m_k$, there is a walk $\< b_\ell,b_{\ell+1},\dots,b_p \>$ in $\< \B,\toB \>$ from $b_\ell = v$ to $b_p = \bar \phi(w_i) = b$ that projects onto $\< a_{m_k},a_{m_k+1},\dots,a_p \>$; and 
\item[$\circ$] for some $m \in D$ with $m \geq m_{k+1}$, there is a walk $\< b_{p+1},b_{p+2},\dots,b_m \>$ in $\< \B,\toB \>$ from $b_{p+1} = \bar \phi(w_{i+1}) = b'$ to $b_m = v$ that projects onto $\< a_{p+1},a_{p+2},\dots,a_m \>$.  
\end{itemize}
But then, concatenating these two walks, 
$$\< b_\ell,b_{\ell+1},\dots,b_p,b_{p+1},b_{p+2},\dots,b_m \>$$ 
is a walk in $\< \B,\toB \>$ from $v$ to $v$ that projects onto $\< a_\ell,\dots,a_m \>$ and has $(b_p,b_{p+1}) = (b,b') = e$ for some $p$ with $\ell \leq p < m$. 

This nearly proves that $v$ witnesses $(\ddagger)$ for $e$, except that we must show it is possible for the parameter $\ell$ to be arbitrarily large in such walks. 

In the above argument, we took $k \geq K$ subject only to the condition that $j(m_k) = j$. There are infinitely many such $k$, because $(\LL^{k_j},\E^{k_j},c^{k_j}_\text{in}c^{k_j}_\text{out})$ is used infinitely often. Consequently, the value of $k$ in the preceding argument may be taken arbitrarily large. To finish the proof, we show that the value of $\ell$ in the preceding argument may be taken to be either $m_k$ or $m_{k-1}$. 

If $\ell = m_k$ we are done, so suppose $\ell < m_k$. Then 
$\< b_\ell,b_{\ell+1},\dots,b_{m_k} \>$ (the truncation of the walk constructed above) is a walk in $\< \B,\toB \>$ that projects onto 
$\< a_\ell,a_{\ell+1},\dots,a_{m_k} \>$. 
Consequently, $(b_\ell,b_{m_k}) = (v,b_{m_k}) \in \mathcal R$. 
But by our choice of $\mathcal R$, and our work near the beginning of the proof, it follows that there is a walk 
$\< b^{m_{k-1}},b^{m_{k-1}+1},\dots,b^{m_k} \>$ in $\< \B,\toB \>$ from $v = b^{m_{k-1}}$ to $b_{m_k} = b^{m_k}$ that projects onto 
$\< a_{m_{k-1}},\dots,a_{m_k} \>$. Thus
$$\< b^{m_{k-1}},b^{m_{k-1}+1},\dots,b^{m_k} = b_{m_k},\dots,b_p,b_{p+1},b_{p+2},\dots,b_m \>$$ 
is a walk in $\< \B,\toB \>$ from $v$ to $v$ that projects onto $\< a_{m_{k-1}},\dots,a_m \>$ and has $(b_p,b_{p+1}) = (b,b') = e$ for some $p$ with $m_{k-1} \leq p < m$. 
Therefore $v$ witnesses $(\ddagger)$ for $e$. Because $e$ was arbitrary, $v$ witnesses $(\ddagger)$, as claimed.
\end{proof}

\section{The Lifting Lemma, part 6: completing the proof}\label{sec:Exhale}

In this section we piece together the results from the preceding few sections to obtain a proof of the Lifting Lemma.

\begin{lemma}
For every partition $\A$ of $\pwmff$, there exists a partition $\set{X_a}{a \in \A}$ of $\w$ into infinite sets such that $[X_a] = a$ for all $a \in \A$, and for every $n \in \w$, if $n \in X_a$ and $n+1 \in X_{a'}$ then $a \tosi a'$.
\end{lemma}
\begin{proof}
Consider the special case of Lemma~\ref{lem:Walks2} with $\< \V,\to \> = \< \A,\tosi \>$ and with the isomorphism $\phi$ taken to be the identity map. The lemma asserts there is an infinite walk $\seq{a_n}{n \in \w}$ through $\< \A,\tosi \>$ such that the natural map associated to this walk is the identity: that is, $a = [\set{n \in \w}{a_n = a}]$ for all $a \in \A$. 
Thus taking $X_a = \set{n \in \w}{a_n = a}$ works.
\end{proof}

\begin{theorem}\label{thm:Polarized}
$\< \pwmff,\s \>$ is polarized.
\end{theorem}
\begin{proof}
Let $\A$ be a partition of $\pwmff$, and let $(\phi,\<\V,\to\>)$ be a virtual refinement of $\< \A,\tosi \>$. 

Applying the previous lemma, there is a partition $\set{X_a}{a \in \A}$ of $\w$ into infinite sets such that $[X_a] = a$ for all $a \in \A$, and for every $n \in \w$, if $n \in X_a$ and $n+1 \in X_{a'}$ then $a \tosi a'$. As in the proof of the previous lemma, there is an infinite walk through $\< \A,\tosi \>$ associated to this partition: for every $n \in \w$, let $X(n)$ denote the unique member of $\set{X_a}{a \in \A}$ containing $n$, and then let $a_n = [X(n)]$.

Applying Theorem~\ref{thm:Heart&Soul} with $\< \B,\toB \> = \< \V,\to \>$, either
\begin{enumerate}
\item there is a diligent walk through $\< \V,\to \>$ that almost projects (via $\phi$) onto $\seq{a_n}{n \in \w}$, or
\item there is a strongly connected digraph $\< \C,\toC \>$ and an epimorphism $\psi$ from $\< \C,\toC \>$ onto $\< \A,\tosi \>$ such that $\psi$ is incompatible with $\phi$, and there is a diligent walk $\seq{c_n}{n \in \w}$ through $\< \C,\toC \>$ that almost projects (via $\psi$) onto $\seq{a_n}{n \in \w}$.
\end{enumerate}

First, suppose case $(1)$ holds. Fix a diligent walk $\seq{v_n}{n \in \w}$ through $\< \V,\to \>$ such that $\phi(v_n) = a_n$ for all but finitely many $n \in \w$. As in Definition~\ref{def:Walks}, let $A_v = [\set{n \in \w}{v_n = v}]$ for each $v \in \V$. Let
$\A' = \set{A_v}{v \in \V}.$ 
We claim that this partition of $\pwmff$ realizes $(\phi,\< \V,\to \>)$. 
First, note that because $\phi(v_n) = a_n$ for all but finitely many $n \in \w$, 
\begin{align*}
\textstyle \bigvee \set{A_v}{\phi(v) = a} &\textstyle \,=\, [\bigcup_{v \in \phi^{-1}(a)} \set{n \in \w}{v_n = v}] \\
&\,=\, [\set{n \in \w}{\phi(v_n) = a}] \,=\, [\set{n \in \w}{a_n = a}] \,=\, a
\end{align*}
for every $a \in \A$. 
Thus $\A'$ refines $\A$. For each $v \in \V$, let $\phi'(v) = A_v$. By Lemma~\ref{lem:Walks1}, $\phi'$ is an isomorphism from $\< \V,\to \>$ to $\< \A',\tosi \>$ (the natural isomorphism). If $v \in \V$ and $\phi(v) = a$, then 
\begin{align*}
\phi'(v) \,=\, [\set{n \in \w}{v_n = v}] &\,\leq\, [\set{n \in \w}{\phi(v_n) = a}] \\
&\,=\, [\set{n \in \w}{a_n = a}] \,=\, a.
\end{align*}
Recall that the natural map $\pi: \A' \to \A$ sends $\phi(v)$ to $a$ whenever $\phi(v) \leq a$. The foregoing computation shows $\pi \circ \phi'(v) = \phi(v)$. As $v \in \V$ was arbitrary, $\pi \circ \phi' = \phi$. Thus $\< \A',\tosi \>$ realizes the virtual extension $(\phi,\< \V,\to \>)$, and so alternative $(1)$ of Definition~\ref{def:Polarized} holds. 

Next, suppose case $(2)$ holds in the Theorem~\ref{thm:Heart&Soul} dichotomy. 
Fix a strongly connected digraph $\< \C,\toC \>$ and an epimorphism $\psi$ from $\< \C,\toC \>$ onto $\< \A,\tosi \>$ such that $\psi$ is incompatible with $\phi$, and there is a diligent walk $\seq{c_n}{n \in \w}$ through $\< \C,\toC \>$ that almost projects (via $\psi$) onto $\seq{a_n}{n \in \w}$. 
Like in the previous paragraph, let $A_c = [\set{n \in \w}{c_n = c}]$ for each $c \in \C$, and let
$\A' = \set{A_c}{c \in \C}.$ 
Because $\psi(c_n) = a_n$ for all sufficiently large $n \in \w$, 
\begin{align*}
\textstyle \bigvee \set{A_c}{\psi(c) = a} &\textstyle \,=\, [\bigcup_{c \in \psi^{-1}(a)} \set{n \in \w}{c_n = c}] \\
&\,=\, [\set{n \in \w}{\psi(c_n) = a}] \,=\, [\set{n \in \w}{a_n = a}] \,=\, a
\end{align*}
for every $a \in \A$. 
Thus $\A'$ refines $\A$. For each $c \in \C$, let $\psi'(c) = A_c$, and note that $\psi'$ is an isomorphism from $\< \C,\toC \>$ to $\< \A',\tosi \>$ by Lemma~\ref{lem:Walks1}. If $c \in \C$ and $\phi(c) = a$, then (as above) 
\begin{align*}
\psi'(c) \,=\, [\set{n \in \w}{c_n = c}] &\,\leq\, [\set{n \in \w}{\phi(c_n) = a}] \\
&\,=\, [\set{n \in \w}{a_n = a}] \,=\, a.
\end{align*}
Thus the natural map $\pi: \A' \to \A$ sends $\psi'(c)$ to $a$ for every $c \in \C$. 
Hence $\< \A',\tosi \>$ realizes the virtual extension $(\psi,\< \C,\toC \>)$. 
Because $\psi$ is incompatible with $\phi$, this implies that alternative $(2)$ of Definition~\ref{def:Polarized} holds.

Thus, for the virtual refinement $(\phi,\< \V,\to \>)$, the two alternatives in Theorem~\ref{thm:Heart&Soul} correspond exactly to the two alternatives in Definition~\ref{def:Polarized}, and consequently, $(\phi,\< \V,\to \>)$ is polarized. 
Because $(\phi,\< \V,\to \>)$ was an arbitrary virtual partition of an arbitrary partition of $\< \pwmff,\s \>$, this shows that $\< \pwmff,\s \>$ is polarized. 
\end{proof}

\begin{theorem}\label{thm:elements}
If $\< \AA,\s \>$ is an elementary substructure of $\< \pwmff,\s \>$, then $\< \AA,\s \>$ is polarized.
\end{theorem}
\begin{proof}
The proof strategy can be summarized in a sentence: This follows from the previous theorem, because polarization descends to elementary substructures. 
However, it is not necessarily obvious that this property descends to elementary substructures, because polarization is not obviously expressible in a first-order sentence. 
So let us check the details.

Let $\< \AA,\s \>$ be an elementary substructure of $\< \pwmff,\s \>$, let $\A$ be a partition of $\AA$, and let $(\phi,\< \V,\to \>)$ be a virtual refinement of $\< \A,\tosi \>$. 
By the previous theorem, $(\phi,\< \V,\to \>)$ is polarized in the larger dynamical system $\< \pwmff,\s \>$. Thus one of two things happens:
\begin{enumerate}
\item there is a partition $\A'$ of $\AA$ that realizes $(\phi,\<\V,\to\>)$, or
\item there is a partition $\A'$ of $\AA$ refining $\A$ such that $(\pi,\< \A',\tosi \>)$ is incompatible with $(\phi,\< \V,\to \>)$, where $\pi$ denotes the natural map from $\A'$ onto $\A$. 
\end{enumerate}

First suppose case $(1)$ holds. 
Fix an enumeration $\A = \{a_1,a_2,\dots,a_m\}$ of $\A$ and an enumeration $\A' = \{a_1',a_2',\dots,a_n'\}$ of $\A'$. 
Like in the proof of Lemma~\ref{lem:IotaElementary}, there is a first-order sentence $\varphi(a_1,\dots,a_m,a_1',\dots,a_n',\s)$ about $\< \pwmff,\s \>$ that records precisely the structure of the digraphs $\< \A,\tosi \>$ and $\< \A',\tosi \>$, and also records which members of $\A'$ are below which members of $\A$ (i.e., it records the action of the natural map $\pi: \A' \to \A$). 
Let $\bar \varphi$ be the formula 
$$\exists x_1 \exists x_2 \dots \exists x_n \varphi(a_1,\dots,a_m,x_1,\dots,x_n,\s).$$
This is a true assertion about $\< \pwmff,\s \>$, witnessed by setting $x_i = a'_i$ for each $i \leq n$. 
By elementarity, $\bar \varphi$ must also be a true assertion about $\< \AA,\s \>$. 
Therefore $\< \AA,\s \> \models \varphi(a_1,\dots,a_m,\bar a_1,\dots,\bar a_n,\s)$ for some $\bar a_1, \bar a_2,\dots, \bar a_n \in \AA$. 
Because of our choice of $\varphi$, this means that $\bar \A = \{\bar a_1, \bar a_2,\dots, \bar a_n\}$ is a partition of $\AA$ refining $\A$ and $(\pi,\< \bar \A,\tosi \>)$ is identical to $(\pi,\< \A',\tosi \>)$, where in each case $\pi$ denotes the natural map. 
Consequently, $(\pi,\< \bar \A,\tosi \>)$ realizes $(\phi,\< \V,\to \>)$, because $(\pi,\< \A',\tosi \>)$ does. 
Thus the first case of Definition~\ref{def:Polarized} holds for $\< \V,\to \>$.

Next suppose case $(2)$ holds. 
As in case $(1)$, fix an enumeration $\A = \{a_1,a_2,\dots,a_m\}$ of $\A$ and an enumeration $\A' = \{a_1',a_2',\dots,a_n'\}$ of $\A'$, and fix a first-order sentence $\varphi(a_1,\dots,a_m,a_1',\dots,a_n',\s)$ about $\< \pwmff,\s \>$ that records precisely the structure of $\< \A,\tosi \>$ and $\< \A',\tosi \>$, and the action of the natural map $\pi: \A' \to \A$. 
Let $\bar \varphi$ be the formula
$$\exists x_1 \exists x_2 \dots \exists x_n \varphi(a_1,\dots,a_m,x_1,\dots,x_n,\s).$$
This is a true assertion about $\< \pwmff,\s \>$, so by elementarity, $\< \AA,\s \> \models \varphi(a_1,\dots,a_m,\bar a_1,\dots,\bar a_n,\s)$ for some $\bar a_1, \bar a_2,\dots, \bar a_n \in \AA$. 
By our choice of $\varphi$, this means that $\bar \A = \{\bar a_1, \bar a_2,\dots, \bar a_n\}$ is a partition of $\AA$ refining $\A$ and $(\pi,\< \bar \A,\tosi \>)$ is identical to $(\pi,\< \A',\tosi \>)$, where in each case $\pi$ denotes the natural map. 
Consequently, $(\pi,\< \bar \A,\tosi \>)$ is incompatible with $(\phi,\< \V,\to \>)$, and the second case of Definition~\ref{def:Polarized} holds for $(\phi,\< \V,\to \>)$.

Thus the two cases listed at the beginning of the proof correspond exactly to the two cases in Definition~\ref{def:Polarized}, and $(\phi,\< \V,\to \>)$ is polarized in $\< \AA,\s \>$. 
Because $(\phi,\< \V,\to \>)$ was an arbitrary virtual partition of an arbitrary partition of $\< \AA,\s \>$, this shows that $\< \AA,\s \>$ is polarized. 
\end{proof}


\begin{proof}[Proof of Lemma~\ref{lem:main}]
Let $(\<\AA,\a\>,\<\BB,\b\>,\iota,\eta)$ be an instance of the lifting problem for $\< \pwmff,\s \>$, and suppose that the image of $\eta$ is an elementary substructure of $\< \pwmff,\s \>$. 
In other words, $\eta$ is a conjugacy from $\< \AA,\a \>$ onto $\< \eta[A],\s \> \preceq \< \pwmff,\s \>$. 
The previous theorem implies that $\< \AA,\a \>$ is polarized, because this property is preserved by conjugacies. 
Consequently, $(\<\AA,\a\>,\<\BB,\b\>,\iota,\eta)$ has a solution by Theorem~\ref{thm:FinalStroke}.
\end{proof}

\noindent We showed already, in Section 6, that the main theorem of the paper (Theorem~\ref{thm:main}) follows from the Lifting Lemma (Lemma~\ref{lem:main}). Thus this proof of the Lifting Lemma concludes our proof of the main theorem. 


\section{A few corollaries}\label{sec:Corollaries}

In this final section we deduce a few further results about $\< \pwmff,\s \>$. 
Some of these results are direct consequences of the main theorem, while some others are obtained from the Lifting Lemma instead. 

Two dynamical systems $\< \AA,\a \>$ and $\< \BB,\b \>$ are \emph{elementarily equivalent} if they satisfy the same sentences of first-order logic in the language of dynamical systems: i.e., if $\varphi(\t)$ is a parameter-free sentence in the language of dynamical systems, then
$\< \AA,\a \> \models \varphi(\a)$ if and only if $\< \BB,\b \> \models \varphi(\b)$.

\begin{theorem}\label{thm:ElementaryEquivalence}
The structures $\<\pwmff,\s\>$ and $\<\pwmff,\s^{-1}\>$ are elementarily equivalent.
\end{theorem}
\begin{proof}
We provide two proofs of this theorem. The first proof uses the technique of forcing, and is in some sense ``slicker'' because it gives the present theorem as a nearly immediate immediate consequence of Theorem~\ref{thm:main}. The second proof is forcing-free, and more direct. It gives the theorem not as a consequence of Theorem~\ref{thm:main}, but rather as a consequence of the technique used in its proof. 

For the first proof, let $\PP$ denote the poset of countable partial functions $\w_1 \times \w \to \w$, ordered by extension, i.e., the usual poset for forcing \ch with countable conditions. Let $G$ be a $V$-generic filter on $\PP$. 
Because $\PP$ is a countably closed notion of forcing, no new subsets of $\w$ are added by $\PP$, which means that $V \cap \pwmff = V[G] \cap \pwmff$. Furthermore, because $\s$ is definable by the simple formula $\s([A]) = [A+1]$, the action of $\s$ is the same in $V$ and in $V[G]$. Thus the structures $\< \pwmff,\s \>^V$ and $\< \pwmff,\s \>^{V[G]}$ are one in the same. 
Thus (by the ``absoluteness of the satisfaction relation'')  the theory of $\< \pwmff,\s \>$ is the same in $V$ and in $V[G]$. 
The same argument applies to $\< \pwmff,\s^{-1} \>$, so the theory of $\< \pwmff,\s^{-1} \>$ is also the same in $V$ and in $V[G]$. 
Because $V[G] \models \ch$, Theorem~\ref{thm:main} implies $\< \pwmff,\s \>$ and $\< \pwmff,\s^{-1} \>$ are conjugate in $V[G]$. 
In particular, $\< \pwmff,\s \>$ and $\< \pwmff,\s^{-1} \>$ are elementarily equivalent in $V[G]$.
To summarize: the theory of $\< \pwmff,\s \>$ is the same in $V$ as in $V[G]$, the theory of $\< \pwmff,\s^{-1} \>$ is the same in $V$ as in $V[G]$, and in $V[G]$ these two structures have the same theory. Consequently, they have the same theory in $V$. 

For the second proof, let us take a closer look at the proof of Theorem~\ref{thm:main} earlier in the paper. 
In that proof, we construct a coherent sequence 
$\seq{\phi_\a}{\a < \w_1}$ of maps, where each $\phi_\a$ is a conjugacy from 
$\< \AA_\a,\s \>$ to $\< \BB_\a,\s^{-1} \>$, where 
$\AA_\a$ and $\BB_\a$ are subalgebras of $\pwmff$. 
In fact, each $\AA_\a$ had the property that 
$\< \AA_\a,\s \>$ is a countable elementary substructure of $\< \pwmff,\s \>$, and  
each $\BB_\a$ had the property that there is some subalgebra $\BB^0_{\a+1}$ with $\BB_\a \sub \BB^0_{\a+1} \sub \BB_{\a+1}$ such that 
$\< \BB_{\a+1}^0,\s^{-1} \>$ is a countable elementary substructure of $\< \pwmff,\s^{-1} \>$.

Note that the existence of these $\phi_\a$'s and $\AA_\a$'s, $\BB_\a$'s and $\BB^0_{\a+1}$'s does not depend on \ch: the construction of the sequence can be carried out in \zfc. The assumption of \ch was used only to arrange the construction so that $\bigcup_{\a < \w_1} \AA_\a = \bigcup_{\b < \w_1} \BB_\a = \pwmff$. 

Without \ch, we still get a map $\phi_\w: \AA_\w \to \BB_\w$. 
This map is a conjugacy from $\< \AA_\w,\s \>$ to $\< \BB_\w,\s^{-1} \>$. 
Now recall that an increasing union of elementary substructures is also an elementary substructure (by the Tarski-Vaught criterion). In particular, 
$\< \AA_\w,\s \>$ and $\< \BB_\w,\s^{-1} \>$ 
are elementary substructures of 
$\< \pwmff,\s \>$ and $\< \pwmff,\s^{-1} \>$, respectively, because $\AA_\w = \bigcup_{n < \w}\AA_n$ and $\BB_\w = \bigcup_{n < \w}\BB_{n+1}^0$.  
Because $\< \pwmff,\s \>$ and $\< \pwmff,\s^{-1} \>$ have elementary substructures that are conjugate to one another, they are elementarily equivalent.
\end{proof}

In fact, one may strengthen the conclusion of this theorem, without much altering the proof, by considering infinitary logic.

\begin{theorem}
The structures $\<\pwmff,\s\>$ and $\<\pwmff,\s^{-1}\>$ share the same $\mathcal L_{\w_1,\w_1}$-theory in the language of dynamical systems. 
\end{theorem}

\noindent As with Theorem~\ref{thm:ElementaryEquivalence}, this can be proved in two ways. The first, using forcing, works exactly as above, noting that the $\mathcal L_{\w_1,\w_1}$-theory of $\<\pwmff,\s\>$ and $\<\pwmff,\s^{-1}\>$ is unchanged by a countably closed forcing. The second is like the second method of proof for Theorem~\ref{thm:ElementaryEquivalence} above, but the recursion should be run for $\w_1$ steps instead of just $\w$. Note that this recursion can run for $\w_1$ steps without assuming \ch; the only difference is that without \ch, we cannot have $\AA_{\w_1} = \BB_{\w_1} = \pwmff$ after only $\w_1$ steps.

\vspace{2mm}

The following theorem stems from the observation that in our proof of Theorem~\ref{thm:main}, the embedding at the base step of the recursion can be arbitrary. This gives us some freedom to construct not just one conjugacy from $\<\pwmff,\s\>$ to $\<\pwmff,\s^{-1}\>$ or from $\<\pwmff,\s\>$ to $\<\pwmff,\s\>$, but many.

\begin{theorem}\label{thm:ManyConjugations}
Assume \ch. Let $\< \AA,\s \>$ be a countable elementary substructure of $\< \pwmff,\s \>$, and let $\eta$ be an embedding of $\< \AA,\s \>$ into $\< \pwmff,\s \>$. Then there is a conjugacy $\phi$ from $\< \pwmff,\s \>$ to itself such that $\phi \rest \AA = \eta$.
\end{theorem}
\begin{proof}
We repeat the proof of Theorem~\ref{thm:main} with a few minor differences. 
As in that proof, if $X \sub \pwmff$ and $|X| \leq \aleph_0$, then $\<\hspace{-1mm}\< X \>\hspace{-1mm}\>$ 
denotes some countable subalgebra of $\pwmff$ with $X \sub \<\hspace{-1mm}\< X \>\hspace{-1mm}\>$ such that 
$\big\langle \hspace{-.2mm} \<\hspace{-1mm}\< X \>\hspace{-1mm}\>,\s \!\rest\! \<\hspace{-1mm}\< X \>\hspace{-1mm}\> \hspace{-.2mm} \big\rangle$ 
is an elementary substructure of $\< \pwmff,\s \>$. 
Also as in that proof, we write 
$\< \AA,\s \>$ to abbreviate $\< \AA,\s \rest \AA \>$.

Let $\seq{a_\xi}{\xi < \w_1}$ be an enumeration of $\pwmff$ in order type $\w_1$.  
Using transfinite recursion, we now build a coherent sequence $\seq{\phi_\a}{\a < \w_1}$ of maps such that each $\phi_\a$ is a conjugacy from its domain $\< \AA_\a,\s \>$ to its co-domain $\< \BB_\a,\s \>$; furthermore, we shall have $\AA_\a \sub \AA_\b$ and $\BB_\a \sub \BB_\b$ whenever $\a \leq \b$, and $\AA_\a = \bigcup_{\xi < \a}\AA_\xi$ and $\BB_\a = \bigcup_{\xi < \a}\BB_\xi$ for limit $\a$. 

Let $\< \AA,\s \>$ be a countable elementary substructure of $\< \pwmff,\s \>$, and let $\eta$ be an embedding of $\< \AA,\s \>$ into $\< \pwmff,\s \>$. 
For the base case of the recursion, let $\AA_0 = \AA$, let $\phi_0 = \eta$, and let $\BB_0 = \phi_0[\AA_0]$. 

The remaining stages of the recursion proceed as in the proof of Theorem~\ref{thm:main}. 
At limit stages there is nothing to do: simply take $\AA_\a = \bigcup_{\xi < \a}\AA_\xi$, $\BB_\a = \bigcup_{\xi < \a}\BB_\xi$, and $\phi_\a = \bigcup_{\xi < \a}\phi_\xi$. 
At successor steps, to obtain $\AA_{\a+1}$ and $\phi_{\a+1}$, first let $\BB_{\a+1}^0 = \<\hspace{-1mm}\< \BB_\a \cup \{a_\a\} \>\hspace{-1mm}\>$ and observe that 
$$(\<\BB_\a,\s\>,\<\BB^0_{\a+1},\s\>,\iota,\phi_\a^{-1})$$ 
is an instance of the lifting problem for $\< \pwmff,\s \>$, where $\iota$ denotes the inclusion $\BB_\a \xhookrightarrow{} \BB_{\a+1}^0$. Furthermore, the image of $\phi^{-1}_\a$, namely $\< \AA_\a,\s \>$, is an elementary substructure of $\< \pwmff,\s \>$. 
By Lemma~\ref{lem:main}, there is an embedding $\phi^0_{\a+1}$ from $\< \BB_{\a+1}^0,\s \>$ into $\< \pwmff,\s \>$. 
Let $\AA^0_{\a+1}$ denote the image of this embedding, so that 
$\phi^0_{\a+1}$ is a conjugacy from $\< \BB^0_{\a+1},\s \>$ to $\< \AA^0_{\a+1},\s \>$, and 
$(\phi^0_{\a+1})^{-1}$ is a conjugacy from $\< \AA^0_{\a+1},\s \>$ to $\< \BB^0_{\a+1},\s \>$. 

Second, let $\AA_{\a+1} = \<\hspace{-1mm}\< \AA_\a^0 \cup \{a_\a\} \>\hspace{-1mm}\>$ and observe that 
$$(\<\AA^0_{\a+1},\s\>,\<\AA_{\a+1},\s\>,\iota,(\phi_\a^0)^{-1})$$ 
is another instance of the lifting problem for $\< \pwmff,\s \>$, where $\iota$ denotes the inclusion $\AA^0_{\a+1} \xhookrightarrow{} \AA_{\a+1}$. Furthermore, the image of $(\phi^0_{\a+1})^{-1}$, namely $\< \BB_{\a+1}^0,\s \>$, is an elementary substructure of $\< \pwmff,\s \>$. 
By Lemma~\ref{lem:main}, there is an embedding $\phi_{\a+1}$ from $\< \AA_{\a+1},\s \>$ into $\< \pwmff,\s \>$. 
Let $\BB_{\a+1}$ denote the image of this embedding, so that 
$\phi_{\a+1}$ is a conjugacy from $\< \AA_{\a+1},\s \>$ to $\< \BB_{\a+1},\s \>$, and 
$(\phi_{\a+1})^{-1}$ is a conjugacy from $\< \BB_{\a+1},\s \>$ to $\< \AA_{\a+1},\s \>$. 

This completes the recursion. In the end, $\phi = \bigcup_{\a < \w_1}\phi_\a$ is a conjugacy from $\< \pwmff,\s \>$ to $\< \pwmff,\s \>$, and clearly $\phi \rest \AA = \phi \rest \AA_0 = \phi_0 = \eta$. 
\end{proof} 

We now deduce several consequences of Theorem~\ref{thm:ManyConjugations}. We begin with two model-theoretic consequences. These two results were obtained in collaboration with Ilijas Farah, and are reproduced here with his permission. 

\begin{theorem}\label{thm:AutoElementary}
Suppose $\< \AA,\s \>$ is a countable elementary substructure of $\< \pwmff,\s \>$. Then every embedding of $\< \AA,\s \>$ into $\< \pwmff,\s \>$ is an elementary embedding. 
\end{theorem}
\begin{proof}
First let us show that this is true assuming \ch. (This will enable us afterward to deduce the general result without \ch.) 

Let $\< \AA,\s \>$ be a countable elementary substructure of $\< \pwmff,\s \>$, and let $\eta$ be an embedding of $\< \AA,\s \>$ into $\< \pwmff,\s \>$. 
By the previous theorem, \ch implies there is a conjugacy map $\phi$ from $\< \pwmff,\s \>$ to itself such that $\phi \rest \AA = \eta$. 
But conjugacy maps preserve the truth of all sentences in the language of dynamical systems (because they are the isomorphisms of this category). From this and the fact that $\< \AA,\s \>$ is an elementary substructure of $\< \pwmff,\s \>$, it follows that $\phi \rest \AA = \eta$ is an elementary embedding: for any first-order sentence $\psi$ and any $a_1,\dots,a_n \in \AA$, 
\begin{align*}
\< \AA,\s \> \models \psi(a_1,\dots,a_n,\s) & \ \Leftrightarrow \ \< \pwmff,\s \> \models \psi(a_1,\dots,a_n,\s) \\
& \ \Leftrightarrow \ \< \pwmff,\s \> \models \psi(\phi(a_1),\dots,\phi(a_n),\s) \\
& \ \Leftrightarrow \ \< \pwmff,\s \> \models \psi(\eta(a_1),\dots,\eta(a_n),\s). 
\end{align*}

Thus the corollary is true assuming \ch. To get rid of this assumption, we use the same forcing trick as in the proof of Theorem~\ref{thm:ElementaryEquivalence}. 
Let $\< \AA,\s \>$ be a countable elementary substructure of $\< \pwmff,\s \>$, and let $\eta$ be an embedding of $\< \AA,\s \>$ into $\< \pwmff,\s \>$. 
As in the proof of Theorem~\ref{thm:ElementaryEquivalence}, there is a forcing extension $V[G]$ such that $V[G] \models \ch$, and such that the theory of $\< \AA,\s \>$ or $\< \eta[\AA],\s \>$ or $\< \pwmff,\s \>$ is the same in $V$ and in $V[G]$. 
By the previous paragraph, $\eta$ is an elementary embedding in $V[G]$. But because the theory of $\< \AA,\s \>$ and $\< \eta[\AA],\s \>$ and $\< \pwmff,\s \>$ is the same in $V$ and in $V[G]$, this implies $\eta$ is an elementary embedding in $V$ as well.
\end{proof}

\begin{theorem}
Every embedding between two elementary substructures of $\< \pwmff,\s \>$ is an elementary embedding.
\end{theorem}
\begin{proof}
Let $\< \AA,\s \>$ and $\< \BB,\s \>$ be elementary substructures of $\< \pwmff,\s \>$, and suppose $\eta$ is an embedding of $\< \AA,\s \>$ into $\< \BB,\s \>$. 
Let $\psi$ be any first-order sentence in the language of dynamical systems, let $a_1,\dots,a_n \in \AA$, and suppose $\< \AA,\s \> \models \psi(a_1,\dots,a_n,\s)$. 
Let $\< \AA_0,\s \>$ be a countable elementary substructure of $\< \AA,\s \>$ with $a_1,\dots,a_n \in \AA_0$. 
Then $\< \AA_0,\s \> \models \psi(a_1,\dots,a_n,\s)$, by elementarity. 
By the previous theorem, $\eta \rest \AA_0$ is an elementary embedding of $\< \AA_0,\s \>$ into $\< \pwmff,\s \>$. 
Thus $\< \pwmff,\s \> \models \psi(\eta(a_1),\dots,\eta(a_n),\s)$. 
Because $\< \BB,\s \>$ is an elementary substructure of $\< \pwmff,\s \>$, this implies 
$\< \BB,\s \> \models \psi(\eta(a_1),\dots,\eta(a_n),\s)$ as well. 
Hence $\< \AA,\s \> \models \psi(a_1,\dots,a_n,\s)$ implies $\< \BB,\s \> \models \psi(\eta(a_1),\dots,\eta(a_n),\s)$; i.e., $\eta$ is an elementary embedding.
\end{proof}

Observe that either of the two instances of $\s$ in the theorem (for the domain or for the co-domain) can be replaced with an instance of $\s^{-1}$. Thus, with only superficial modifications to the proofs above, we obtain the following.

\begin{theorem}\label{thm:Allementary}
Every embedding between two elementary substructures of $\< \pwmff,\s \>$ is an elementary embedding; 
every embedding of an elementary substructure of $\< \pwmff,\s \>$ into $\< \pwmff,\s^{-1} \>$ is an elementary embedding; 
and every embedding of an elementary substructure of $\< \pwmff,\s^{-1} \>$ into $\< \pwmff,\s \>$ is an elementary embedding.
\end{theorem}

The last two statements in this theorem are especially intriguing. 
In the proof of Theorem~\ref{thm:main}, and the analogous proof of Theorem~\ref{thm:ManyConjugations}, we kept careful track of which substructures were elementary, using elementarity on one side of our back-and-forth construction to obtain mappings back to the other side. 
Now, from Theorem~\ref{thm:Allementary}, we can see that all the substructures used in the proof were elementary substructures, on both sides of the back-and-forth argument. 
But we see this only in hindsight, after Theorem~\ref{thm:ManyConjugations} is established. 
There is irony in this, something like a security system, that, when installed, acts as a deterrent and is therefore never activated. 

This situation suggests the possibility that our main theorem could have a simpler proof, perhaps a proof based on the model-theoretic notion of $\aleph_1$-saturation. An earlier draft of this paper claimed erroneously, as a corollary to Theorem~\ref{thm:NoGo}, that $\< \pwmff,\s \>$ is not $\aleph_1$-saturated. Ilijas Farah found a mistake in the proof of this corollary, and we now can show that in fact the opposite is true. (A proof appears in the forthcoming paper \cite{BF}.) However, our proof of saturation still relies on the combinatorial arguments in Sections 8 and 10. So while this proof of $\aleph_1$-saturation offers further insight into the structure of the shift map, it does not seem to provide an essentially easier path to proving the main theorem of this paper. 

\vspace{2mm}

We now move on to a question that has been part of the folklore surrounding the shift map for years: \emph{Is there an automorphism of $\pwmff$ that commutes with $\s$ other than the powers of $\s$?} 
We show the answer is yes, and in fact this is a relatively straightforward consequence of Theorem~\ref{thm:ManyConjugations}. 

For the statement and proof of the following lemma, let $E$ and $F$ denote the following subsets of $\w$:
$$E \,=\, \textstyle \bigcup \set{[2^k,2^{k+1})}{k \text{ is even}} \quad \text{and} \quad F \,=\, \bigcup \set{[2^k,2^{k+1})}{k \text{ is odd}}.$$

\begin{lemma}\label{lem:NonTrivialConjugation}
Suppose $\AA$ is a countable subalgebra of $\pwmff$ that is closed under $\s$ and $\s^{-1}$, and with $[\set{kn}{n \in \w}] \in \AA$ for all $k \in \w$, and with $[E],[F] \in \AA$. 
Then there is an embedding $\eta$ of $\< \AA,\s \>$ into $\< \pwmff,\s \>$ such that $\eta \neq \s^n \rest \AA$ for any $n \in \Z$.  
\end{lemma}
\begin{proof}
Suppose $\AA$ is a countably infinite subalgebra of $\pwmff$ that is closed under $\s$ and $\s^{-1}$. 
Let $\mathcal S = \set{A \sub \w}{[A] \in \AA \setminus [\0]}$. 
Then $\mathcal S$ is a countable set of infinite subsets of $\w$. 
By a standard diagonalization argument, $\mathcal S$ is not a \emph{splitting family}; i.e., there is some infinite $D \sub \w$ such that for all $A \in \mathcal S$ either $D \cap A$ is finite or $D \setminus A$ is finite.

We now define a permutation $h: \w \to \w$ by permuting some intervals associated to $D$. Specifically, let $d_0 = 0$ and let $\set{d_i}{i > 0}$ be an increasing enumeration of the set $D+1$. Define $h$ by setting
$$h(i) = \begin{cases}
i + d_{2n+2} - d_{2n+1} & \text{ if }i \in [d_{2n},d_{2n+1}) \text{ for some $n$}, \\
i - d_{2n+1} + d_{2n} & \text{ if }i \in [d_{2n+1},d_{2n+2}) \text{ for some $n$}.
\end{cases}$$
In other words, if we set $I_n = [d_n,d_{n+1})$ for all $n$, then $h$ is the map that simply swaps the intervals $I_{2n}$ and $I_{2n+1}$ for all $n$. 

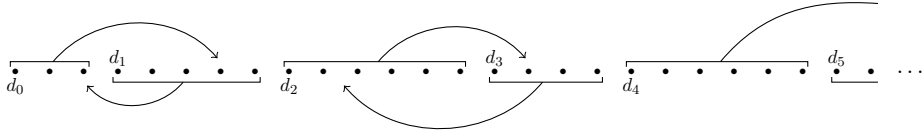
\begin{center}
\begin{tikzpicture}[scale=.455]

\node at (0,0) {\scalebox{.5}{$\bullet$}};
\node at (1,0) {\scalebox{.5}{$\bullet$}};
\node at (2,0) {\scalebox{.5}{$\bullet$}};
\node at (3,0) {\scalebox{.5}{$\bullet$}};
\node at (4,0) {\scalebox{.5}{$\bullet$}};
\node at (5,0) {\scalebox{.5}{$\bullet$}};
\node at (6,0) {\scalebox{.5}{$\bullet$}};
\node at (7,0) {\scalebox{.5}{$\bullet$}};
\node at (8,0) {\scalebox{.5}{$\bullet$}};
\node at (9,0) {\scalebox{.5}{$\bullet$}};
\node at (10,0) {\scalebox{.5}{$\bullet$}};
\node at (11,0) {\scalebox{.5}{$\bullet$}};
\node at (12,0) {\scalebox{.5}{$\bullet$}};
\node at (13,0) {\scalebox{.5}{$\bullet$}};
\node at (14,0) {\scalebox{.5}{$\bullet$}};
\node at (15,0) {\scalebox{.5}{$\bullet$}};
\node at (16,0) {\scalebox{.5}{$\bullet$}};
\node at (17,0) {\scalebox{.5}{$\bullet$}};
\node at (18,0) {\scalebox{.5}{$\bullet$}};
\node at (19,0) {\scalebox{.5}{$\bullet$}};
\node at (20,0) {\scalebox{.5}{$\bullet$}};
\node at (21,0) {\scalebox{.5}{$\bullet$}};
\node at (22,0) {\scalebox{.5}{$\bullet$}};
\node at (23,0) {\scalebox{.5}{$\bullet$}};
\node at (24,0) {\scalebox{.5}{$\bullet$}};
\node at (25,0) {\scalebox{.5}{$\bullet$}};
\node at (26.2,0) {\scalebox{.8}{$\dots$}};

\node at (0,-.4) {\scalebox{.65}{$d_0$}};
\node at (3,.4) {\scalebox{.65}{$d_1$}};
\node at (8,-.4) {\scalebox{.65}{$d_2$}};
\node at (14,.4) {\scalebox{.65}{$d_3$}};
\node at (18,-.4) {\scalebox{.65}{$d_4$}};
\node at (24,.4) {\scalebox{.65}{$d_5$}};

\draw [->] (1.1,.3) to [out=50,in=130] (5.9,.4);
\draw (-.15,.15) -- (-.15,.3) -- (2.15,.3) -- (2.15,.15);
\draw [<-] (2.1,-.4) to [out=-50,in=-130] (4.9,-.3);
\draw (2.85,-.15) -- (2.85,-.3) -- (7.15,-.3) -- (7.15,-.15);

\draw [->] (10.6,.3) to [out=50,in=130] (14.9,.4);
\draw (7.85,.15) -- (7.85,.3) -- (13.15,.3) -- (13.15,.15);
\draw [<-] (9.6,-.4) to [out=-50,in=-130] (15.4,-.3);
\draw (13.85,-.15) -- (13.85,-.3) -- (17.15,-.3) -- (17.15,-.15);

\draw (20.6,.3) to [out=50,in=176] (25.2,2);
\draw (17.85,.15) -- (17.85,.3) -- (23.15,.3) -- (23.15,.15);
\draw (23.85,-.15) -- (23.85,-.3) -- (25.2,-.3);

\end{tikzpicture}
\end{center}

\noindent Define $s: \w \to \w$ by setting $s(i) = h^{-1}(h(i)+1)$ for all $i \in \w$. Checking a few cases, we see that  
$$s(i) = \begin{cases}
i + 1 & \text{ if } i \in [d_{2n},d_{2n+1}-1) \text{ for some $n$}, \\
d_{2n} & \text{ if } i = d_{2n+2}-1 \text{ for some $n$}, \\
d_{2n+3} & \text{ if } i = d_{2n+1}-1 \text{ for some $n$}.
\end{cases}$$
In particular, $s(i) = i+1$ for all $i \notin D$, and $s[D] = \{0\} \cup (D+1) \setminus \{d_1\}$. 

\begin{center}
\begin{tikzpicture}[scale=.455]

\node at (0,0) {\scalebox{.5}{$\bullet$}};
\node at (1,0) {\scalebox{.5}{$\bullet$}};
\node at (2,0) {\scalebox{.5}{$\bullet$}};
\node at (3,0) {\scalebox{.5}{$\bullet$}};
\node at (4,0) {\scalebox{.5}{$\bullet$}};
\node at (5,0) {\scalebox{.5}{$\bullet$}};
\node at (6,0) {\scalebox{.5}{$\bullet$}};
\node at (7,0) {\scalebox{.5}{$\bullet$}};
\node at (8,0) {\scalebox{.5}{$\bullet$}};
\node at (9,0) {\scalebox{.5}{$\bullet$}};
\node at (10,0) {\scalebox{.5}{$\bullet$}};
\node at (11,0) {\scalebox{.5}{$\bullet$}};
\node at (12,0) {\scalebox{.5}{$\bullet$}};
\node at (13,0) {\scalebox{.5}{$\bullet$}};
\node at (14,0) {\scalebox{.5}{$\bullet$}};
\node at (15,0) {\scalebox{.5}{$\bullet$}};
\node at (16,0) {\scalebox{.5}{$\bullet$}};
\node at (17,0) {\scalebox{.5}{$\bullet$}};
\node at (18,0) {\scalebox{.5}{$\bullet$}};
\node at (19,0) {\scalebox{.5}{$\bullet$}};
\node at (20,0) {\scalebox{.5}{$\bullet$}};
\node at (21,0) {\scalebox{.5}{$\bullet$}};
\node at (22,0) {\scalebox{.5}{$\bullet$}};
\node at (23,0) {\scalebox{.5}{$\bullet$}};
\node at (24,0) {\scalebox{.5}{$\bullet$}};
\node at (25,0) {\scalebox{.5}{$\bullet$}};
\node at (26.2,0) {\scalebox{.8}{$\dots$}};

\draw[->] (.1,0)--(.9,0);
\draw[->] (1.1,0)--(1.9,0);
\draw [->] (2.1,.1) to [out=30,in=150] (13.9,.1);
\draw[->] (3.1,0)--(3.9,0);
\draw[->] (4.1,0)--(4.9,0);
\draw[->] (5.1,0)--(5.9,0);
\draw[->] (6.1,0)--(6.9,0);
\draw [<-] (.1,.1) to [out=30,in=150] (6.9,.1);
\draw[->] (8.1,0)--(8.9,0);
\draw[->] (9.1,0)--(9.9,0);
\draw[->] (10.1,0)--(10.9,0);
\draw[->] (11.1,0)--(11.9,0);
\draw[->] (12.1,0)--(12.9,0);
\draw [->] (13.1,.1) to [out=30,in=150] (23.9,.1);
\draw [<-] (8.1,.1) to [out=30,in=150] (16.9,.1);
\draw[->] (14.1,0)--(14.9,0);
\draw[->] (15.1,0)--(15.9,0);
\draw[->] (16.1,0)--(16.9,0);
\draw[->] (18.1,0)--(18.9,0);
\draw[->] (19.1,0)--(19.9,0);
\draw[->] (20.1,0)--(20.9,0);
\draw[->] (21.1,0)--(21.9,0);
\draw[->] (22.1,0)--(22.9,0);
\draw[->] (24.1,0)--(24.9,0);
\draw (25.1,0)--(25.3,0);
\draw [<-] (18.1,.1) to [out=30,in=178] (25.3,2);
\draw (23.1,.1) to [out=30,in=197] (25.3,1);

\end{tikzpicture}
\end{center}

As was mentioned in Section~\ref{sec:Incompressibility}, a function $f: \w \to \w$ naturally induces a function $\pwmff \to \pwmff$, namely $f^*([A]) = [\set{f(n)}{n \in A}]$. Let $h^*$ denote the function $\pwmff \to \pwmff$ induced by $h$, i.e., $h^*([A]) = [\set{h(n)}{n \in \w}]$ for all $A \sub \w$. Note that $h^*$ is an automorphism of $\pwmff$, because $h$ is a bijection $\w \to \w$. 

We shall finish the proof of the lemma by showing that $\eta = h^* \rest \AA$ is the desired embedding. 
Because $\eta$ is the restriction to $\AA$ of an automorphism $\pwmff \to \pwmff$, it is automatically an embedding of Boolean algebras $\AA \to \pwmff$. We must check that $\eta \circ \s(a) = \s \circ \eta(a)$ for all $a \in \AA$, and that $\eta \neq \s^p \rest \AA$ for any $p \in \Z$. 

Let $a \in \pwmff$, and fix some $A \sub \w$ with $a = [A]$. Then, using the fact that $s(n) = h^{-1}(h(n)+1)$ for all $n$, 
\begin{align*}
\s \circ \eta(a) &\,=\, \s([\set{h(n)}{n \in A}]) \,=\, [\set{h(n)+1}{n \in A}]  \\
& \,=\, [\set{h(s(n))}{n \in A}] \,=\, \eta([\set{s(n)}{n \in A}]).
\end{align*}
If $a \in \AA$, then by our choice of $D$ either $D \setminus A$ is finite or $D \cap A$ is finite. 
If $D \cap A$ is finite, then, using the fact that $s(i) = i+1$ when $i \notin D$,
$$[\set{s(n)}{n \in A}] = [\set{s(n)}{n \in A \setminus D}] = [\set{n+1}{n \in A \setminus D}] = [A+1].$$  
On the other hand, if $D \setminus A$ is finite then, using the fact that $s(i) = i+1$ when $i \notin D$ and the fact that $s[D] = \{0\} \cup (D+1) \setminus \{d_1\}$, 
\begin{align*}
[\set{s(n)}{n \in A}] &\,=\, [\set{s(n)}{n \in A \setminus D} \cup \set{s(n)}{n \in D}] \\
&\,=\, [\set{n+1}{n \in A \setminus D} \cup (D+1)] \,=\, [A+1].
\end{align*} 
In either case, if $a \in \AA$ then 
$\s \circ \eta (a) = \eta([\set{s(n)}{n \in A}]) = [A+1] = \eta \circ \s(a).$

Finally, we must show that $\eta \neq \s^p \rest \AA$ for any $p \in \Z$. 

Let $N \in \w$ with $N > 0$, and for each $j \in \{0,1,\dots,N-1\}$ let $N\w+j = \set{Nk+j}{k \in \w}$. 
By hypothesis, $\AA$ contains $[N\w]$. Furthermore, because $\AA$ is closed under $\s$, $\AA$ contains $[N\w+j] = \s^j([N\w])$ for all $j \in \{0,1,\dots,N-1\}$. 
For any such $j$, either $D \setminus (N\w+j)$ is finite or $D \cap (N\w+j)$ is finite. 
Because the sets $N\w+0, N\w+1,\dots,N\w+N-1$ are disjoint, the first alternative must hold for precisely one of them: i.e., there is some particular $j < N$ such that $D \setminus (N\w+j)$ is finite. 
If $n$ is large enough that all members of $D \setminus (N\w+j)$ are below $d_{2n}$, then for $i \in [d_{2n},d_{2n+1})$ we have $h(i) = i + d_{2n+2} - d_{2n+1} \equiv_N i$ (because $d_{2n+2} \equiv_N d_{2n+1}$); and similarly, for $i \in [d_{2n+1},d_{2n})$ we have $h(i) = i - d_{2n+1}+d_{2n} \equiv_N i$ 
(because $d_{2n+1} \equiv_N d_{2n}$). 
Hence $h(i) \equiv_N i$ for all sufficiently large $i$. 
Consequently, $\eta([N\omega]) \leq [N\omega]$. 

Now fix $p \in \Z$. If $p \neq 0$, then for any $N > p$ we have $(N\w) \cap (N\w+p) = \0$ and consequently, 
$$\s^p([N\w]) \wedge [N\w] \,=\, [N\w+p] \wedge [N\w] \,=\, [(N\w) \cap (N\w+p)] \,=\, [\0].$$ 
In particular, $\s^p([N\w]) \not\leq [N\w]$. 
Combined with the previous paragraph, this shows $\eta([N\w]) \neq \s^p([N\w])$. 
Because $[N\w] \in \AA$, this means $\eta \neq \s^p \rest \AA$.

For the remaining case $p = 0$, we must find some $a \in \AA$ with $\eta(a) \neq a$. 
This is where we use the hypothesis that $[E],[F] \in \AA$. 
Because $\AA$ is closed under $\s^{-1}$, we have $[E-1],[F-1] \in \AA$. 
Because $E$ and $F$ partition $\w$, and because of our choice of $D$, exactly one of $D \cap (E-1)$ and $D \cap (F-1)$ is finite; equivalently, 
exactly one of $(D+1) \cap E$ and $(D+1) \cap F$ is finite. 

Suppose $(D+1) \cap F$ is finite. 
Fix $m$ large enough that $d_n \in E$ for all $n \geq m$. 
There are infinitely many $n \geq m$ such that 
$[d_n,d_{n+1}] \not\sub E$: this is the case whenever some $d_n$ is the last member of $D+1$ in one of the intervals comprising $E$, and the next member of $D+1$ is in the next interval of $E$. 
When this happens, one of the intervals comprising $F$ is sandwiched between two consecutive members of $D+1$. 
Thus, for infinitely many $n$, there is an odd $k$ such that $[2^k,2^{k+1}) \sub (d_n,d_{n+1})$. 

If $n$ is odd, then $h$ shifts the interval $[2^k,2^{k+1})$ down by $\ell = d_n-d_{n-1} > 0$, so that $h[F]$ contains, and $h[E]$ is disjoint from, the interval $[2^k-\ell,2^{k+1}-\ell)$. 
But $E$ is not disjoint from any intervals of length $2^k$ that begin below $2^k$, so this means there is some $i \geq 2^k-\ell > d_{n-1}$ with $i \in E \setminus h[E]$. 

If $n$ is even, then $h$ shifts $[2^k,2^{k+1})$ up by $\ell = d_{n+2}-d_{n+1} > 0$, so that $h[F]$ contains, and $h[E]$ is disjoint from, the interval $[2^k+\ell,2^{k+1}+\ell)$. But also $2^k-1,2^{k+1} \in E$, so $2^k-1+\ell, 2^{k+1}+\ell \in h[E]$. Thus $h[E]$ contains a gap of length precisely $2^k$ beginning at $2^k+\ell > 2^k$. Because every gap in $E$ beginning above $2^k$ has length $>\!2^k$, this means there is some $i \geq 2^k > d_n$ with $i \in h[E] \setminus E$. 

Put together, the previous two paragraphs show that there are arbitrarily large values of $i$ that are either in $E \setminus h[E]$ or in $h[E] \setminus E$. It follows that $[E] \neq [\set{h(i)}{i \in E}] = \eta([E])$. 
Thus, in the case that $(D+1) \cap F$ is finite, $\eta$ is not the identity map on $\AA$. 
A similar argument, interchanging the roles of $E$ and $F$, shows that if $[F] \neq \eta([F])$, hence $\eta$ is not the identity map on $\AA$, in the case that $(D+1) \cap E$ is finite. 
\end{proof}

\begin{theorem}\label{thm:NonTrivialConjugation}
Assuming \ch, there is an automorphism $\phi$ of $\pwmff$ such that $\phi \circ \s = \s \circ \phi$, but $\phi \neq \s^n$ for any $n \in \Z$.
\end{theorem}
\begin{proof}
Let $\< \AA,\s \>$ be a countable elementary substructure of $\<\pwmff,\s\>$ such that $[E],[F] \in \AA$. 
Applying the previous lemma, let $\eta$ be an embedding of $\< \AA,\s \>$ into $\< \pwmff,\s \>$ such that $\eta \neq \s^n \rest \AA$ for any $n \in \Z$. 
By Theorem~\ref{thm:ManyConjugations}, there is a conjugacy map $\phi$ from $\<\pwmff,\s\>$ to itself such that $\phi \rest \AA = \eta$. 
Thus $\phi \circ \s = \s \circ \phi$, but $\phi \neq \s^n$ for any $n \in \Z$, because $\phi \rest \AA = \eta \neq \s^n \rest \AA$ for any $n \in \Z$. 
\end{proof}

Let us note that there are several possible ways of proving Lemma~\ref{lem:NonTrivialConjugation}. We have presented the shortest and simplest we could find. But another worth mentioning is that, by constructing appropriate sequences to induce the embeddings as in Section 4, one can show that for any countable elementary substructure $\< \AA,\s \>$ of $\< \pwmff,\s \>$, there are $2^{\aleph_0}$ distinct embeddings of $\< \AA,\s \>$ in $\< \pwmff,\s \>$. A simple counting argument then implies that most of these embeddings are not equal to $\s^n \rest \AA$ for any $n$. 

\vspace{2mm}

Next we address another folklore question: \emph{Can a trivial automorphism be conjugate to a nontrivial one?} 
We show the answer is again yes, this time as a relatively straightforward consequence of Theorem~\ref{thm:ManyConjugations}. The following lemma was suggested by Ilijas Farah. 

\begin{lemma}
Let $\s$ denote the automorphism $\pzmff \to \pzmff$ induced by the shift map $n \mapsto n+1$ on $\Z$. There is no trivial automorphism $\t$ of $\pzmff$ such that $\t \circ \t = \s$.
\end{lemma}
\begin{proof}
Aiming for a contradiction, suppose there is a map $t: \Z \to \Z$ that induces an automorphism $\t$ of $\pzmff$ with $\t \circ \t = \s$, in the sense that $\t([A]) = [t(A)]$ for all $A \sub \Z$. 

There is some finite $F \sub \Z$ such that $t \rest (\Z \setminus F)$ is injective, since otherwise $t$ would not induce an automorphism in the manner described in the previous paragraph. 
Similarly, there is some finite $G \sub \Z$ such that $t \circ t(n) = n+1$ for all $n \notin G$. (Otherwise there one can find an infinite $A \sub \Z$ such that $t \circ t(A) \cap (A+1) = \0$, which means $\t \circ \t([A]) \wedge [A] = [\0]$ and in particular, because $[A] \neq [\0]$, $\t \circ \t([A]) \neq [A]$.)

Let $n > \max (F \cup G \cup t(F) \cup t^{-1}(F) \cup t^{-1}(G))$. Observe that $t(n) > \max G$, because $n > \max ( t^{-1}(G))$. 

We claim that $t^{2k}(n) = n+k$ for all $k \geq 0$. This is proved by a simple induction on $k$: the base case is trivial, and assuming it is true for $k$, then $t^{2k}(n) = n+k > n > \max G$, and this implies that $t^{2(k+1)}(n) = t \circ t (t^{2k}(n)) = t^{2k}(n)+1 = n+k+1$. 
By an essentially identical inductive argument, we also have $t^{2k+1}(n) = t^{2k}(t(n)) = t(n)+k$ for all $k \geq 0$. 

If $n < t(n)$, then there is some $k>0$ such that $t(n) = n+k = t^{2k}(n)$, and in particular $t(n) = t(t^{2k-1}(n))$. 
If $n > t(n)$, then there is some $k>0$ such that $n = t(n)+k = t^{2k+1}(n)$, and in particular, $t(t^{-1}(n)) = t(t^{2k}(n))$. 
However, $n,t^{-1}(n) > \max F$, and $t^{2k}(n) = n+k, t^{2k-1}(n) = t(n)+k-1 > \max F$, so this contradicts our choice of $F$.
\end{proof}

\begin{theorem}\label{thm:Ilijas}
Suppose that $\< \pwmff,\s \>$ and $\< \pwmff,\s^{-1} \>$ are conjugate. Then there is a nontrivial autormorphism $\z$ of $\pwmff$ such that $\< \pwmff,\s \>$ and $\< \pwmff,\z \>$ are conjugate. In particular, \ch implies $\s$ is conjugate to a nontrivial automorphism of $\pwmff$.
\end{theorem}
\begin{proof}
Let $t$ be a bijection on $\w$ with one $\Z$-like orbit: for concreteness, take $t(0)=1$, otherwise $t(2n) = 2n-2$, and $t(2n+1) = 2n+3$. 
Let $\t$ be the automorphism of $\pwmff$ induced by $t$. Then $\< \pwmff,\t \>$ is conjugate to $\< \pzmff,\s \>$. 
Observe that $\s^2$ is the automorphism of $\pwmff$ induced by the function $n \mapsto n+2$ on $\w$. This function has two $\N$-like orbits. 
Supposing $\< \pwmff,\s \>$ and $\< \pwmff,\s^{-1} \>$ are conjugate (e.g., under \ch), it is clear that $\< \pwmff,\t \>$ and $\< \pwmff,\s^2 \>$ are also conjugate. 
Let $\phi$ be an automorphism of $\pwmff$ such that $\phi \circ \t = \s^2 \circ \phi$. 
Define $\z = \phi^{-1} \circ \s \circ \phi$. Then 
\begin{align*}
\z \circ \z &\,=\, \phi^{-1} \circ \s \circ \phi \circ \phi^{-1} \circ \s \circ \phi \,=\, \phi^{-1} \circ \s^2 \circ \phi \,=\, \t. 
\end{align*}
Because $\< \pwmff,\t \>$ is conjugate to $\< \pzmff,\s \>$, this implies via the previous lemma that $\z$ is a nontrivial automorphism of $\pwmff$. And $\< \pwmff,\s \>$ is conjugate to $\< \pwmff,\z \>$, because $\phi \circ \z = \s \circ \phi$.
\end{proof}

These last two theorems can be summarized by saying that under \ch, there is both a nontrivial automorphism of $\pwmff$ that commutes with $\s$, as well as a nontrivial automorphism that is conjugate to $\s$. We end this paper by showing that, intriguingly, at least one of these two things must be true as soon as there are any nontrivial automorphisms at all. 

\begin{theorem}\label{thm:trivia}
If there is a nontrivial automorphism of $\pwmff$, then either there is a nontrivial automorphism that commutes with $\s$, or else there is a nontrivial automorphism that is conjugate to $\s$. 
\end{theorem}
\begin{proof}
Suppose there is a nontrivial automorphism $\phi$ of $\pwmff$, and let $\psi = \phi \circ \s \circ \phi^{-1}$. Observe that $\psi$ is an automorphism of $\pwmff$, and that $\< \pwmff,\psi \>$ is conjugate to $\< \pwmff,\s \>$, with $\phi$ as the conjugacy map, because
$$\phi \circ \s \,=\, \phi \circ \s \circ \phi^{-1} \circ \phi \,=\, \psi \circ \phi.$$ 
If $\psi$ is nontrivial then the second alternative of the theorem holds, and we are done. 
So suppose $\psi$ is trivial. Because $\< \pwmff,\s \>$ is incompressible and $\< \pwmff,\psi \>$ is conjugate to $\< \pwmff,\s \>$ by the previous paragraph, $\< \pwmff,\psi \>$ is incompressible. But, up to a bijection on $\w$, there are only two incompressible trivial automorphisms of $\pwmff$, namely $\s$ and $\s^{-1}$ (see \cite[Lemm 5.1]{Brian0}). More precisely, there is a trivial automorphism $\t$ of $\pwmff$ such that either $\psi = \t^{-1} \circ \s \circ \t$, or else $\psi = \t^{-1} \circ \s^{-1} \circ \t$. But then 
$$\phi \circ \s \circ \phi^{-1} \,=\, \psi \,=\, \t^{-1} \circ \s^{\pm 1} \circ \t,$$
which implies that
$$\s \,=\, \phi^{-1} \circ \t^{-1} \circ \s^{\pm 1} \circ \t \circ \phi \,=\, (\t \circ \phi)^{-1} \circ \s^{\pm 1} \circ (\t \circ \phi),$$
i.e., $\s$ is conjugate to either $\s$ or $\s^{-1}$ via $\t \circ \phi$. 
If $\s$ is conjugate to $\s^{-1}$, then Theorem~\ref{thm:Ilijas} implies the second alternative of the theorem. 
If not, then observe that the composition of a trivial and nontrivial automorphism is nontrivial, so $\t \circ \phi$ is a nontrivial automorphism commuting with $\s$.
\end{proof}

We do not know which of these two alternatives, if either, actually follows from the existence of a nontrivial automorphism of $\pwmff$.




\begin{thebibliography}{99}
\bibitem{YCor} Y. de Cornulier, ``Near permutation $n \mapsto n+1$ not conjugate to its inverse on the Stone-\v{C}ech remainder?,'' MathOverflow question 353000.
\end{thebibliography}
\end{document}